%% file: dual_3-connected_facesize.tex
\documentclass[11pt,a4paper]{article}
\usepackage{amsmath, amsfonts, amsthm, amssymb, xcolor}
\usepackage{enumerate, hyperref}
\usepackage{tikz}
\usepackage{tkz-euclide}

\newtheorem{thm}{Theorem}

\newtheorem{cor}{Corollary}

\newtheorem{lem}{Lemma}

\newtheorem{rem}{Remark}

\theoremstyle{empty}

\begin{document}

\title{Face sizes and the connectivity of the dual}

\author{
  Gunnar Brinkmann \thanks{Applied Mathematics, Computer Science, and Statistics, Ghent University, Krijgslaan 281 S9, B 9000 Ghent.
    E-mail address: {\tt gunnar.brinkmann@ugent.be}},
Kenta Noguchi\thanks{{\bf Corresponding author} \\Department of Information Sciences,
Tokyo University of Science,
2641 Yamazaki, Noda, Chiba 278-8510, Japan.
E-mail address: {\tt noguchi@rs.tus.ac.jp}
} , 
Heidi Van den Camp\thanks{
Applied Mathematics, Computer Science, and Statistics, Ghent University, Krijgslaan 281 S9, B 9000 Ghent.
    E-mail address: {\tt heidi.vandencamp@ugent.be}
} 
}

\date{}
\maketitle

\noindent
\begin{abstract}

  For each $c\ge 1$ we prove tight lower bounds on face sizes that must be present to allow $1$- or $2$-cuts in simple duals of $c$-connected maps.
  Using these bounds, we determine the smallest genus on which a $c$-connected map can have a simple dual with a $2$-cut and give lower
  and some upper bounds for the smallest genus on which a $c$-connected map can have a simple dual with a $1$-cut.

\end{abstract}


\section{Introduction}
In this article we use the word {\em map} for a simple connected graph with a cellular embedding in an orientable surface. In cases where multi-edges are allowed this is explicitly mentioned.
Applying well known operations like e.g.\ dual, truncation, ambo, or others to maps 
can change important properties of the maps. An important
property never changed when operations are applied to polyhedra (plane, 3-connected maps) is 3-connectivity. This property is not only relevant in topological graph theory
when wanting to guarantee polyhedral embeddings, but also in structural graph theory where graphs with cuts of size one or two are often considered trivial -- e.g.\ when discussing snarks or
in some cases when talking about hamiltonian cycles, where such cuts imply a partition of each hamiltonian cycle into two parts with a priori known sets of vertices or even forbid hamiltonian cycles.
Unless girth restrictions are imposed on the graph, 3-connectivity of the dual is also the best one can hope for, as triangles induce 3-cuts in the dual.

Among all symmetry preserving operations (characterized e.g. in \cite{gocox}), the operation dual is not only the most important and best known, but also the
one that can have the strongest impact on the connectivity of the graph. In \cite{3con-op} it is shown that, when applied to maps with a simple dual, all operations are guaranteed to preserve
3-connectivity -- except for the dual itself. In \cite{BBZ} it has been shown that also requiring higher connectivity does not help: for each $c>0$ there are
$c$-connected maps with a simple dual that has a 2-cut or even a 1-cut. The first can already happen for maps with a genus only one higher than the minimum
genus $g_{min}(c)$ of an orientable surface on which a $c$-connected graph can be embedded. We have $g_{min}(c)=0$ for $c\le 5$ and $g_{min}(c)=\lceil
(c-2)(c-3)/12 \rceil$ for $c\ge 6$ \cite{PZ}.  Following \cite{BBZ}, for $k\in \{1,2\}$, we define $\delta_k(c)$ as the smallest genus of an orientable surface on which 
$c$-connected maps can be embedded with a simple dual with a $k$-cut.

This problematic behaviour of the operation dual makes it relevant to find criteria guaranteeing that also the dual is 3-connected. In \cite{preserve_polyh} it
is shown that in case the map is polyhedral, polyhedrality and therefore also 3-connectivity is preserved by all operations -- including the dual. 
Polyhedrality is a very strong requirement and many graphs do not allow polyhedral embeddings on any orientable surface, so less restrictive criteria, that still guarantee 3-connectivity,
can be very useful.

In this paper we will give a criterion guaranteeing 3-connectivity, resp.\ 2-connectivity, of simple duals only based on face sizes. We will denote the size of a face $f$ in a map as $d(f)$. As a
corollary of the main results, also questions from \cite{BBZ} are answered: while it was shown in \cite{BBZ} that for $c\ge 7$ $\delta_2(c) \le g_{min}(c)+1$, it was
explicitly asked whether this bound is sharp or whether for some $c>2$ we have $\delta_2(c) = g_{min}(c)$. We will show that $\delta_2(c) = g_{min}(c)+1$ for
all $c>2$. We will also give better bounds on $\delta_1(c)$: the upper bound for $\delta_1(c)$ proven in \cite{BBZ} is $(c^2+6c-5)/4$ -- so instead of a small additive constant like for $\delta_2(c)$, the bound is a constant factor
away from $g_{min}(c)$. We will -- at least for $c=12s+8$ with $s\ge 2$ -- improve this bound to $\delta_1(c)\le g_{min}(c)+3$ and prove that for all $c$ we have
$\delta_1(c)\ge g_{min}(c)+2>\delta_2(c)$.

\section{Basic results and notation}

Graphs $G=(V,E)$ embedded in an orientable surface $\mathbb{S}_g$ of genus $g$ -- or with other words {\em maps} of genus $g$ -- can be described by a function
$\phi: G \to \mathbb{S}_g$ describing the graph as a subset of a topological surface, or (if the embedding is cellular) by a rotation system \cite{GT,MT}.  In
the first case the embedding is {\em cellular} if $\mathbb{S}_g \backslash \phi(G)$ is a disjoint union of open disks (2-cell regions). These regions are called
faces. A rotation system represents each edge by two oppositely directed oriented edges and for each vertex a rotational order of edges starting at this vertex
is defined.  In this case the faces are closed directed walks (called facial walks) of directed edges $e_1,e_2,\dots ,e_k=e_1$ obtained by following an oriented
edge $e$ in one direction, then taking the next in rotational order of the inverse of $e$ as new edge until one returns to the first edge and traverses it in
the same direction. By abuse of language we sometimes refer to {\em an edge} when in fact {\em one of the oriented versions of the edge} is meant. This will not
cause misunderstandings.  An {\em angle} of a directed cyclic walk $e_1,e_2,\dots ,e_k=e_1$ is a pair $(e_i,e_{i+1})$. For details and the equivalence of the
two approaches we refer the reader to \cite{GT,MT}.  If one interprets the order of edges as clockwise, the face can be said to be {\em on the left} of the
facial walk. We will often by abuse of language use the graph notion $G=(V,E)$ and when talking about a map $G$ implicitly assume a rotation system or the
function describing the graph on the surface to be given.  The dual $G^*=(V^*,E^*)$ of a map $G=(V,E)$ is the map where the faces of $G$ are the vertices,
together with a rotation system (resp.\ a function $\phi^*: G^* \to \mathbb{S}_g$) induced by the rotation system (the function $\phi$) of $G$. There is a
natural bijection between $E$ and $E^*$. We denote the edge of $G^*$ corresponding to the edge $e$ of $G$ as $e^*$ and analogously for sets of edges. As the
vertices of a dual map $G^*=(V^*,E^*)$ are the faces of the original map $G$, for an element $f\in V^*$ we can also refer to the face $f$ of $G$ (and analogously for
vertices of $G$). For details on the dual see again \cite{GT, MT}.

Let $G=(V,E)$ be a connected graph and $X, Y\subseteq V$, $X\cap Y=\emptyset$.
Then $E[X, Y]$ denotes the set of edges of $G$ with one end in $X$ and the other end in $Y$. 
The set $E[X, V\backslash X]$ is an \emph{edge cut} of $G$ induced by $X$. 
If an edge cut consists of $k$ edges, then it is called a \emph{$k$-edge cut}. 
A \emph{cut-set}, or more specific \emph{$k$-cut}, of a connected graph is a set $T\subseteq V$ of $k$ vertices such that $G\backslash T$ is not connected. 
A \emph{separating $k$-cycle} of $G$ is a cycle so that the vertices in the cycle are a $k$-cut.

The following lemma about the relation between an edge-cut in the dual and separating cycles in the original graph is folklore for plane graphs, but needs a few words in the
general case. For a set $Y$ of faces of a graph $G=(V,E)$ we define $V(Y)=\{v\in V\mid \exists f\in Y, e\in f, v\in e\}$ and $E(Y)=\{e\in E\mid \exists f\in Y, e\in f\}$.
For $Y$ a set of edges we define $V(Y)=\{v\in V\mid \exists e\in Y, v\in e\}$.

We say that a set $K$ of edges is {\em dual-separating} in a map $G=(V,E)$, if $K^*$ is an edge-cut in $G^*$.

\begin{lem} \label{lem:cut_dual}
  Let $G=(V,E)$ be a map with a simple dual $G^*=(V^*,E^*)$ and let $K\subset E$ be a dual-separating set of edges. Let $X_f\subset V^*$ be a set of vertices
  inducing the edge cut $K^*\subset E^*$ of $G^*$. Note that the elements of $X_f$ are faces of $G$.

  Then we have the following properties:

  \begin{description}

  \item[(i):] $\forall e\in K$:  $e$ is incident with one face in $X_f$ and one face in $X_f^c$.

  \item[(ii):]  If $v,v'\not\in V(K)$,  $v\in V(X_f)$, $v'\in V(X_f^c)$, then there is no path from $v$ to $v'$ without vertices in $V(K)$.
    So if such $v,v'$ exist, $V(K)$ is a cut-set.

  \item[(iii):]  $K$ can be decomposed into cyclic directed walks $Z_1,\dots ,Z_m$, so that $\{e\in E(X_f) \mid e\cap V(K)\not= \emptyset\}$ is exactly
    $K$ together with all edges $e$ for which an angle $(e',e'')$ in one of $Z_1,\dots ,Z_m$ exists, so that $e$ comes after the inverse of $e'$ and before $e''$ in the cyclic order
    around the endpoint of $e'$ -- or informally speaking: together with all edges on the left of and incident to $Z_1,\dots ,Z_m$. The cyclic walks  $Z_1,\dots ,Z_m$ are facial walks
    in the components of the graph induced by $K$ with the embedding induced by $G$.

    \item[(iv:)] $|V(K)|\le |K|$.

    \end{description}

\end{lem}

\begin{proof}
  Part (i) is a trivial consequence of the fact that edges $e^*\in K^*$ have one endpoint in $X_f$ and one in $X_f^c$.

  For part (ii) let $v=v_1,v_2, \dots ,v_n=v'$ be a path from $v$ to $v'$. Any vertex belonging to a face of $X_f$ as well as to a face of $X_f^c$ is contained in
  an edge of $K$, as in the cyclic order of faces around the vertex there must be edges separating $X_f$ and $X_f^c$. So $v_1$ is only in a face of $X_f$ and $v_n$ is only in a face of
  $X_f^c$. If $v_l$ is the last vertex only in a face of $X_f$, then $v_{l+1}$ is in a face of $X_f^c$ -- and as all edges incident to $v_l$ are in faces of $X_f$, so
  is the edge $\{v_l,v_{l+1}\}$ and $v_{l+1}$ is in faces of both -- $X_f$ and $X_f^c$ -- and therefore in an edge of $K$.

  For part (iii) we choose the cyclic walks as follows: we consider the components of the graph formed by $K$ as embedded subgraphs of $G$. Let $C$ be such a
  subgraph.  For each edge $e$ of $C$ we choose the facial walk with the face of $X_f$ neighbouring $e$ on the left as one of the cyclic walks. As each facial
  walk neighbours only faces of $X_f$ or only faces of $X_f^c$ on the left (otherwise these faces would be separated by an edge) and as each edge in $C$ has a face
  of $X_f$ on exactly one side, we choose each edge of $C$ for exactly one cyclic walk.

  Let now $Z$ be one of the cyclic walks chosen this way and $(e',e'')$ be an angle of $Z$.  Then edges on the left hand side -- that is: all edges occuring in
  the cyclic order after the inverse of $e'$ and before $e''$ -- belong to faces in $X_f$, as otherwise $e''$ would not be the next edge to the inverse of $e'$ in the
  order. If on the other hand an edge $e\in E(X_f)$ is incident to a vertex $v$ of $V(K)$ and $e\not\in K$, we can follow the cyclic order around $v$ in counterclockwise direction until the first
  edge in $K$ (that is the inverse of $e'$) and in clockwise direction until we find $e''$. The pair $(e',e'')$ is an angle of one of the cyclic directed walks.

  Part (iv) follows directly from (iii).

\end{proof}

We will investigate $c$-connected graphs. The more faces an embedding has, the smaller the genus, so the theoretically optimal case (which is of course not
always possible) is when all faces are triangles. For an embedded graph $G=(V,E)$ with $F$ the set of faces we call
\[f^+(G)=\displaystyle\sum_{f\in F(G)} (d(f)-3)\]

the {\em face excess}.
The smallest possible number of vertices of a $c$-connected graph is $c+1$. The vertex excess is defined as 

\[v^+(G) = |V|-(c+1) .\]

With this notation, we can formulate the following lemma, which will be used later:

\begin{lem}\label{lem:vf_excess}
  Let $c\ge 6$ be an integer and $G$ be a simple map of genus $g$ with minimum degree at least $c$, and let $v_x,f_x$ be so that $v^+(G)\ge v_x$, and $f^+(G)\ge f_x$. 

   Then $g\ge  \lceil \frac{(c-2)(c-3)}{12} +     \frac{(c-6)v_x}{12} + \frac{f_x}{6} \rceil \ge    g_{min}(c)+ \lfloor  \frac{(c-6)v_x}{12} + \frac{f_x}{6} \rfloor$.
\end{lem}

\begin{proof}
  Counting edges in each face we have $3|F(G)|+f_x\le 2|E(G)|$. As $G$ has minimum degree $c$ we have $|E(G)| \ge \frac{c|V(G)|}{2}$, and by definition
  $|V(G)| \ge c+1+v_x$.

Substituting these inequalities into Euler's formula, we get
\begin{eqnarray*}
  2-2g &\le& |V(G)| - |E(G)| + \frac{2|E(G)|-f_x}{3} =  |V(G)| - \frac{|E(G)|}{3} - \frac{f_x}{3}\le |V(G)|-\frac{c|V(G)|}{6}- \frac{f_x}{3} \\
\end{eqnarray*}
so
\begin{eqnarray*}
  g &\ge& \frac{(c-6)|V(G)|}{12}+ \frac{f_x}{6}+1 \ge \frac{(c-6)(c+1+v_x)}{12}+ \frac{f_x}{6}+1  \\
  & = & \frac{(c-2)(c-3)}{12} +     \frac{(c-6)v_x}{12} + \frac{f_x}{6} 
\end{eqnarray*}
\end{proof}

\begin{rem}\label{rem:cutsize}
  \begin{itemize}
\item Let $G=(V,E)$ be a $2$-connected graph with a $2$-cut $\{v_1, v_2\}$. 
  Then there exists a $k$-edge cut with $k \le \left\lfloor \frac{d(v_1)+d(v_2)}{2} \right\rfloor$ induced by a set $X\subset V$ containing $v_1$ and $v_2$ and all edges in the
  edge cut are incident with one of  $v_1,v_2$.

\item Let $G=(V,E)$ be a connected graph with cut-vertex $v$.
Then there exists a $k$-edge cut with $k \le \left\lfloor \frac{d(v)}{2} \right\rfloor$ induced by a set $X\subset V$ containing $v$ and all edges in the
  edge cut are incident with $v$.
\end{itemize}
\end{rem}

\begin{proof}
Let $C_1$ be the vertex set of a component of $G\backslash \{v_1, v_2\}$ (resp.\  $G\backslash \{v\}$).
Then at least one of $X = \{v_1, v_2\} \cup C_1$ (resp.\  $X = \{v\} \cup C_1$) and $X = V\backslash C_1$ has the property in the remark.
\end{proof}

In the following proofs we often have to distinguish between those dual-separating cycles $Z$ that are also separating cycles and those where all vertices in faces of one of
the parts are on $Z$. To this end we define an {\em empty $k$-circuit}, resp.\ an {\em empty $k$-pair}  as a map on at most $k$ vertices with a  spanning $k$-face $f_e$,
resp.\ a spanning set of two faces $f_e,f'_e$ with $d(f_e)+d(f'_e)=k$ not sharing an edge, spanning the graph. For empty $k$-circuits and empty $k$-pairs, it is not required that the dual is simple,
but it must not have loops and if the dual contains multi-edges, they must all contain $f_e$, resp.\ one of $f_e,f'_e$. In cases where the spanning face must be a cycle, we also write
{\em empty $k$-cycle} instead of empty $k$-circuit.
Note that empty $k$-circuits and empty $k$-pairs are embedded subgraphs of $K_{k'}$ with $k'\le k$ the number of vertices of the map.
In the context of empty circuits and pairs we call edges not on $f_e$ or $f'_e$ {\em internal}.

\pagebreak[2]

\begin{rem}\label{rem:emptyfaces}~
    
   \begin{description}

  \item[(i):] For $k\in \{3,4,5\}$ there are no empty $k$-circuits with $f_e$ neighbouring $k$ pairwise different faces.

   \item[(ii):] For $k\in \{3,4,5\}$ there are no empty $k$-circuits containing a face $f$ not sharing an edge with $f_e$.

  \item[(iii):]  Empty $6$-circuits with $f_e$ neighbouring $6$ different faces have $6$ vertices and at least $13$ edges.

  \item[(iv):]  Empty $6$-circuits with at least $3$ faces and $f_e$ neighbouring only one face $f$ do not exist.

   \item[(v):]  An empty $6$-pair with $f_e,f'_e$ neighbouring $6$ different faces has $6$ vertices and at least $12$ edges.

    \item[(vi):] An empty $6$-pair with at least $4$ faces, so that $f_e,f'_e$ share only edges with the same face $f$
      does not exist.

         \item[(vii):]  Empty $7$-circuits with $f_e$ neighbouring $7$ different faces have $7$ vertices and at least $15$ edges.
 
    \item[(viii):] Empty $7$-circuits with at least $3$ faces, and $f_e$ neighbouring only one face $f$ do not exist.

   \item[(ix):]  An empty $7$-pair with $f_e,f'_e$ neighbouring $7$ different faces has $7$ vertices and at least $14$ edges.

    \item[(x):] An empty $7$-pair with at least $4$ faces, so that $f_e,f'_e$ share only edges with the same face $f$
        does not exist.

       \end{description}
\end{rem}

\begin{proof}

  (i) Such empty circuits would have genus at least one, as otherwise they would be outerplanar graphs and therefore have a vertex of degree 2, so that the two incident edges
  would be contained in the same two faces. As under the assumptions in (i) the empty circuit would have at least $k+1$ faces and maximum degree at most $4$,
  we get $|V|+|F|>|E|$ and therefore 
  $|V|-|E|+|F|>0$ in contradiction with genus at least $1$.

  (ii) The edges of $f$ must form a cycle of edges not on $f_e$, which implies that the minimum degree of the vertices on $f$ must be $4$. This only leaves
  $k=5$, where the only possible face $f$ is another $5$-cycle. As these two cycles contain all edges of $K_5$ and as $f$ must be adjacent to $5$ different faces, the (at least $10$)
  edges of these faces not in $f$ must be in $f_e$, which is a pentagon -- a contradiction.

  (iii) Again such an empty circuit can not be plane, so we have $|V|-|E|+|F|\le 0$ and $|V|\le 6$, $|F|\ge 7$. For $|V|=5$ we get $|E|\ge 12$ -- a contradiction -- and for $|V|=6$ we get the result.

  (iv) To avoid double edges in the dual, vertices with internal edges have exactly three internal edges, which implies immediately that no two vertices of the
  $6$-face are the same, as in that case there were only $5$ vertices and a vertex with degree $5$. Furthermore it implies that if such an empty $6$-circuit
  existed, it would be an embedded $K_6$.  So we have $6$ oriented edges leaving $f_e$ in the cyclic walk of $f$ and therefore also $6$ different faces not sharing an
  edge with $f_e$. These $6$ faces can contain at most $30-(6+12)=12$ oriented edges -- a contradiction.

  (v) We have $5$ or $6$ vertices and at least $8$ faces, so we have $5-|E|+8\le 2$, so $|E|\ge 11$ (a contradiction) or $6-|E|+8\le 2$ and therefore $|E|\ge 12$.

  (vi) As there are at least $4$ faces, we get analogously to (iv) that there is a vertex with $3$ internal edges, implying that $f_e$ and $f'_e$ are disjoint
  and that a possible empty pair would be an embedding of $K_6$. Now the same computation as in (iv) gives a contradiction.

  (vii) If the empty $7$-circuit contained a vertex twice, it had at most $6$ vertices, so maximum degree $5$ and the vertex occurring twice had at most one internal edge, so that
  two edges of $f_e$ would neighbour the same face. So we have $7$ vertices, at least $8$ faces, and $7-|E|+|F|\le 0$, implying the result.

  (viii) In case of (at least) two vertices in the facial walk of the $7$-gon being the same vertex $v$, this vertex can not have interior edges, as it would need $3$ interior
  edges in that place, giving a total degree of at least $7$ while we have at most $6$ vertices. On the other hand, there must be a vertex with interior edges
  and as each of these vertices must have $3$ interior edges, we have $6$ vertices and the boundary consists of a $3$-cycle and a $4$-cycle sharing a vertex. In
  this configuration, for each neighbour of $v$ there are exactly $3$ vertices to which the internal edges can go without creating double edges. This implies
  that the only vertex that is not a neighbour of $v$ has also interior edges -- but it has only $2$ possible neighbours, as $v$ can not have internal edges. So
  $f_e$ is a simple 7-cycle.

  If the $7$-cycle $f_e$ has a vertex $v$ without interior edges, we can argue similar to (iv): as there must be a vertex with an interior edge --
  and therefore at least $3$ interior edges -- a vertex at distance $2$ of $v$ in the $7$-cycle must have an interior edge, as the $4$ vertices different from $v$ and the vertices at distance $2$ of $v$
  would have to form a $K_4$ with the interior edges, while one
  of the edges is already present on $f_e$. This again implies that all
  vertices but $v$ must have interior edges. So $f$ is at least a $13$-gon and the graph is a subgraph $G'$ of $K_6$ with an additional vertex connected to $2$
  vertices of $G'$. So the graph has at most $34$ oriented edges. The faces different from $f$ and $f_e$ together have at most $34-(7+13)=14$ oriented edges, so there are at most
  $4$ of these faces and one of them must share at least $2$ edges with $f$ -- a contradiction.
  
  In case of $f_e$ being a 7-cycle and all vertices having interior edges, there are $7$ interior oriented edges in the cyclic walk of $f$, so that $f$ is at
  least a $14$-gon. The faces $f$ and $f_e$ together have at least $21$ oriented edges, so there are at most $21$ oriented edges left for the other faces. If
  the set of other faces is $F_o$, we have $|F_o|< 7$ if $|E|<21$ or if $f$ contains more than $14$ oriented edges,
  and $|F_o|\le 7$ if $|E|=21$.  If $|E|=21$ we have an embedding of $K_7$, so if $|F_o|=7$, 
  together with $f$ and $f_e$ there are exactly $9$ faces and the Euler characteristic is an odd number. So also in this case we have $|F_o|< 7$ and we again have a double edge between $f$ and a
  vertex corresponding to a face in $F_o$.
  
   \begin{figure}
     \begin{center}
       \includegraphics[width=55mm]{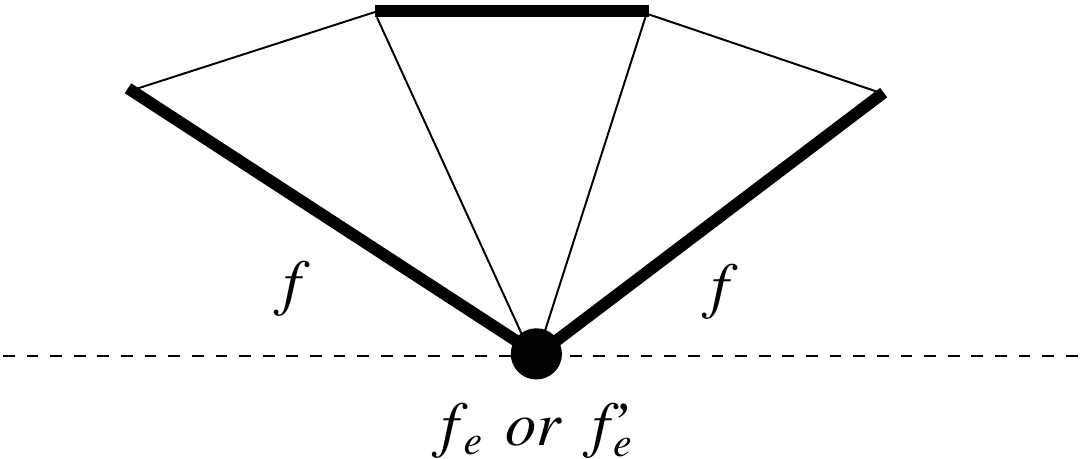}
     \end{center}
     \caption{The rotational order around the vertices in case (x) for an embedded $K_7$. Lines between $f_e$ or $f'_e$ and $f$ are dashed and lines between $f$ and faces in $F_o$ are bold.}\label{fig:viii}
   \end{figure}

  (ix)  In case of the two faces sharing a vertex, the common vertex would have degree at least $6$, while there are only $6$ vertices, so we have $7$ vertices and at least $9$ faces, so we have $7-|E|+9\le 7-|E|+|F| \le 2$ and therefore $|E|\ge 14$.
  
   (x) The arguments and computations are until the point when we have an embedding of $K_7$ as the last open case -- except for using $f_e,f'_e$ instead of just
   $f_e$ -- very similar to those in the proof of (viii), so we will not repeat them.  At that point we have that $|F_o|=7$, that $f$ is a $14$-gon, and that
   each face in $F_o$ is a triangle and shares one edge with $f$ and two edges with other faces from $F_o$. Using the fact that the vertices have degree $6$, we
   get the situation depicted in Figure~\ref{fig:viii}.  For the map formed by the vertices and edges of $F_o$ we have $7$ vertices, $14$ edges and at least $8$
   faces: next to the $7$ faces in $F_o$ the one or two faces formed by the $7$ interior edges in the facial walk of $f$. Due to the parity of the Euler
   characteristic we get that these $7$ oriented edges form two faces: a triangle $t$ and a quadrangle $q$.  Each vertex is contained in $2$ triangles of $F_o$
   with an edge in one of $t,q$ and one with an edge in the other (as $t,q$ are disjoint).  As each vertex of $t$ is only contained in one triangle with an edge
   in $q$, at least one of the $4$ triangles in $F_o$ with an edge in $q$ can not have the third vertex in $t$ -- a contradiction.
  
\end{proof}

\section{Guaranteeing 3-connectedness of the dual}

In order to allow small cuts in a simple dual of a simple graph, the sizes of faces and how they are situated with respect to each other play a crucial
role. In this chapter we want to determine what the minimum sum of the sizes of two faces of a $c$-connected simple map is, so that the dual can be simple and
have a $2$-cut. In fact we will prove a stronger version only considering faces that intersect in a special way. To this end we define

\[min^2_f(c)=\left\{ \begin{array}{cc}
    7 & \mbox{ if } c\in \{1,2\} \\
   10 & \mbox{ if } c\in\{3,4\} \\
  12 & \mbox{ if }  c \ge 5 \\
 \end{array}
 \right. .
 \]

 We call two faces {\em doubly intersecting}, if in the facial walk of one of the faces one or more vertices of the other occur at at least two places. So especially two faces
 sharing $2$ or more different vertices are doubly intersecting.
We will now state and prove our main theorem:

\begin{thm}\label{thm:mainthm}
  ~
  
  \vspace*{-0.2cm}
  
   \begin{description}

   \item[(i):] For each $c$ there are $c$-connected maps $G$ with two doubly intersecting faces $f,f'$, so that $d(f)+d(f') =min^2_f(c)$, $d(f'')=3$ for all $f''\not\in \{f,f'\}$,
     and the dual $G^*$ is simple and has a 2-cut.

   \item[(ii):] If $G=(V,E)$ is a $c$-connected map with a simple dual $G^*$ and for every two doubly intersecting faces $f,f'$ we have $d(f)+d(f') <min^2_f(c)$, then $G^*$ is 3-connected.
  
     \end{description}
 
\end{thm}

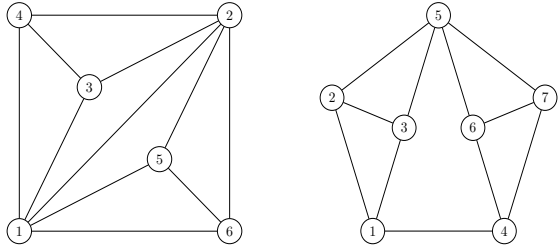
\begin{figure}[tb]
  \begin{center}
  \begin{minipage}{4cm}
\begin{center}	
  \resizebox{0.8\textwidth}{0.8\textwidth}
        {
          \input{c2_minf_7.tikz}
        }
\end{center}
  \end{minipage}
  \begin{minipage}{4cm}
\begin{center}	
  \resizebox{0.8\textwidth}{0.8\textwidth}
        {
          \input{c2_minf_7_dual.tikz}
        }
\end{center}
  \end{minipage}
  \end{center}
	\caption{ A 2-connected map and its dual. The map has two faces $f,f'$ that share two vertices, so that $d(f)+d(f')= 7$ and $d(f'')=3$ for all $f''\not\in \{f,f'\}$. The dual is simple and has a 2-cut.}
	\label{fig:c2minf_7}
\end{figure}

\begin{figure}[tb]
\begin{minipage}{8cm}
\begin{center}	
  \resizebox{0.7\textwidth}{0.7\textwidth}
        {
          \input{v8e20.tikz}
        }
\end{center} 
\end{minipage}\hfill
\begin{minipage}{8cm}
\begin{center}	
        \resizebox{0.7\textwidth}{0.7\textwidth}
        {
          \input{v8e20_dual.tikz}
        }
\end{center}
\end{minipage}
	\caption{ A 4-connected map and its dual. The map has two faces $f,f'$ that share two vertices, so that $d(f)+d(f') = 10$ and $d(f'')=3$ for all $f''\not\in \{f,f'\}$. The dual is simple and has a 2-cut.}
	\label{fig:4conexcess4}
\end{figure}
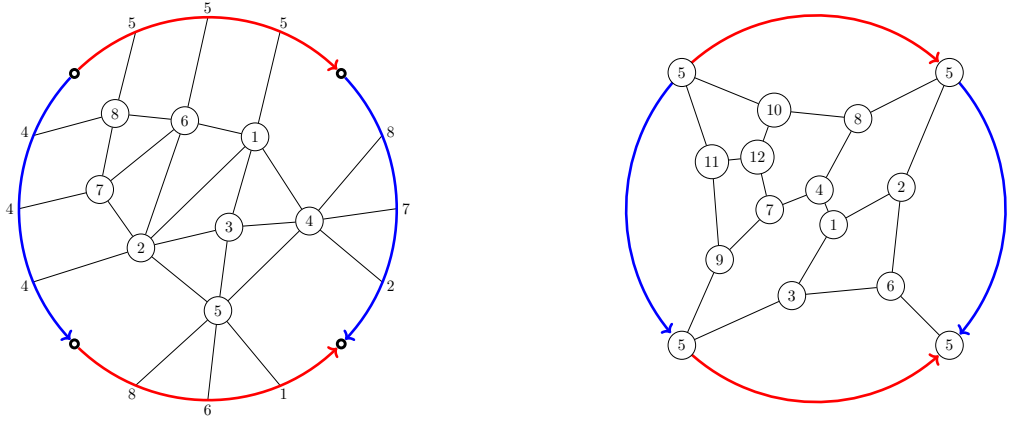

\begin{proof}

  For part (i) and $c\le 4$, examples are given in Figure~\ref{fig:c2minf_7} and Figure~\ref{fig:4conexcess4}. As shown in  \cite{BBZ}, for $c\ge 5$ the $H$-operation described in
  \cite{BBZ} can produce embeddings of complete graphs for arbitrarily large $c$ that are triangular except for two faces of size $6$ sharing four
  vertices. These maps have a simple dual with a 2-cut.

  For part (ii) and $c\in \{1,2\}$, note that simple triangulations of any surface are $3$-connected, so that at least one face of size $d(f)\ge 4$ is needed to have a $1$- or $2$-cut,
  proving the case for maps with a $1$- or $2$-cut. For $1$- and for $2$-connected graphs that are also $3$-connected, the result follows from the stronger result for $c=3$.

  For $c\in \{3,4\}$, it is sufficient to prove the result for $c=3$. So assume that we have a $3$-connected graph $G$ with a simple dual, $d(f)+d(f') \le 9$
  for any two doubly intersecting faces, and that $G^*$ has a $1$- or $2$-cut $C^*$. Assume first that $G^*$ has a $1$-cut or a $2$-cut $C^*$ with $d(f)+d(f')
  \le 9$ for the two vertices in $C^*$ -- no matter whether they are doubly intersecting or not. As each face can have size at most six -- as neighbouring faces
  have size at least three -- we get by Remark \ref{rem:cutsize} for 1-cuts as well as for 2-cuts that there exists a set $X^*$ containing $C^*$, so that
  $E(X^*,(X^*)^c)$ is a $k$-edge cut in $G^*$ with $k \le 4$, in which all the edges are incident to a vertex of $C^*$. Let $X$ be the map formed by all
  vertices and edges of faces of $G$ that are in $X^*$. We refer to the faces in $X$ that correspond to vertices in $C^*$ as {\em cut-faces}. Due to
  Lemma~\ref{lem:cut_dual}, the cycle $Z$ corresponding to $E(X^*,(X^*)^c)$ is a $3$- or $4$-cycle and each edge of $Z$ is contained in one of the faces of
  $C^*$. As $X^*$ contains a vertex not in $C^*$, due to Remark~\ref{rem:emptyfaces} (ii) the map $X$ is not an empty cycle. It contains a vertex $v$ and as
  each vertex on $Z$ has three vertex disjoint paths to $v$, at least $3$ vertices on $Z$ are incident to edges in the cut-faces not on $Z$ -- we call those
  interior edges. In case of $Z$ a $3$-gon, there are no edges between vertices of $Z$ except for those on $Z$. So (compare Figure~\ref{fig:4cyc}) all paths in
  the facial walks of $f,f'$ e.g. leaving $Z$ with the internal edge $e_{1,R}$ have length at least two before reaching another vertex of $Z$.  This gives a
  minimum length of $9$ for the sum of the lengths of the paths. In case of only one cut-face, this is already a contradiction.  In case of two cut-faces, at
  least one of the paths must return to the same vertex (e.g. start at $e_{1,R}$ and enter at $e_{3,L}$), which needs a path of length at least $3$, so a total
  length of at least $10$.
  
  In case of $Z$ a $4$-gon and three vertices with interior edges, we can argue similar to the case of a $3$-gon with one subdivided edge:
  now one of the outgoing interior edges can go directly to $Z$, but only one orientation of this edge can be part of a cut-face, as otherwise the vertices not in this edge would form a 2-cut.
  So the length of the paths between vertices of $Z$, but not on $Z$, can be one shorter, while $Z$ is one edge longer -- and we get the same contradiction.
  
  In case of $4$ vertices with interior edges, we have also $4$ outgoing interior paths -- and at most two of them can have length $1$, as again none of these
  edges can be in two outgoing interior paths without having a $2$-cut. This gives a total length of at least $10$ -- a contradiction proving the case $c\in
  \{3,4\}$ for a $1$-cut or a $2$-cut $C^*$ with $d(f)+d(f') \le 9$.

  The last case is that $d(f)+d(f') > 9$. In this case we still have $d(f)+d(f') \le 12$ as the maximum face size is six and we know that $f$ and $f'$ are not
  doubly intersecting.  Then, by Remark~\ref{rem:cutsize}, there exists a set $X^*$ containing $C^*$, so that $E(X^*,(X^*)^c)$ is a $k$-edge cut with $k \le 6$
  in $G^*$, in which each edge is incident to a vertex of $C^*$. By Lemma~\ref{lem:cut_dual} and as $G$ is simple, these edges correspond to a dual separating
  cyclic walk or two cyclic walks $Z,Z'$ in $G$ (as described in the proof of Lemma~\ref{lem:cut_dual}) that are dual separating.  If one of the cyclic walks
  (say $Z_1$) would contain edges of both cut-faces -- so especially if we only have one cyclic walk, $f$ and $f'$ would be doubly intersecting: there would be
  at least two different places in the cyclic walk where edges of $f$ and edges of $f'$ meet. If these two places are at different vertices, $f$ and $f'$ are
  doubly intersecting, so assume to the contrary that there is a directed subpath of edges of $f$ in $Z_1$ starting and ending at the same vertex $v$. If $e$ is
  the first directed edge of the subpath (so it is starting in $v$) and $e'$ is the last edge (so it is ending in $v$), by construction $e,e'$ can not be an
  angle of the facial walk of $f$ -- as otherwise it would also be the next edge in $Z_1$, so these two occurences are at different positions in the facial walk
  of $f$ and again we have the contradiction that $f$ and $f'$ are doubly intersecting.

  To this end we have that $Z$ is neighbouring $f$ and $Z'$ is neighbouring $f'$ on one side and both are $3$-cycles. We will first show that the faces of $X^*$
  contain a vertex not on $f$ or $f'$: As $f$ has also neighbours in $X^*$, there exists an edge of $f$ that contains a vertex $v$ of $Z$, but is not on $Z$. As a single such edge would
  imply a loop in $G^*$, there are at least two such edges at $v$ and at least one must have the other vertex $v'$ not on $Z'$ as $f,f'$ are not doubly
  intersecting. So $v'$ is in a face of $X^*$, but neither on $Z$ nor on $Z'$. So there are $3$ vertex disjoint paths from $v'$ to the set of vertices of $Z$ and of $Z'$
  in $G$. Taking them as short as possible, they are entirely in faces of $X^*$, so at least one of $Z,Z'$ -- say $Z$ -- has at least two vertices with internal edges. As at each of
  these vertices we have at least $2$ internal edges belonging to $f$ and all $4$ of these edges are pairwise distinct, $f$ has at least size $7$ -- a final contradiction proving the
  statement for $c=3$.

  \medskip

  For $c\ge 5$ assume that we have a $5$-connected graph $G$ with a simple dual, $d(f)+d(f') \le 11$ for any two faces doubly intersecting $f,f'$, and the dual
  $G^*$ is simple and has a $1$- or $2$-cut $C^*$. Assume first that we have a 1-cut or that we have a 2-cut with $d(f)+d(f') \le 11$ for the two faces $f,f'$
  corresponding to the cut-vertices (without assuming that they are doubly intersecting at this point).  As each face can have size at most eight, we get by Remark
  \ref{rem:cutsize} for 1-cuts as well as for 2-cuts that there exists a set $X^*$ containing $C^*$, so that $E(X^*,(X^*)^c)$ is a $k$-edge cut in $G^*$ with
  $k  \le 5$, in which all the edges are incident to a vertex of $C^*$. By Lemma~\ref{lem:cut_dual} and as $G$ is simple, these edges correspond to a 3-, 4- or
  5-cycle $Z$ in $G$ neighbouring the set $C$ of one or two faces corresponding to $C^*$ on one side. We write $X$ for the map of all vertices and edges of
  faces in $X^*$. In case of a $2$-cut and if $Z$ contains edges of both faces, we also have that they are doubly intersecting.

  We will use the following observation:

\begin{itemize}
	\item As $G^*$ is simple, each face in $(X^*)^c$ can contain at most two edges of $Z$, as on the side of $X$ there are at most two different faces.
\end{itemize}

As $G$ is 5-connected, we can immediately conclude that in case of $Z$ being a 3- or 4-cycle, at least one of the sides must be an empty cycle.
 As $X$ contains
at least one face that is not a cut-face, Remark~\ref{rem:emptyfaces} (ii) implies that this must be the part with faces of $X^c$. The
observation then implies that $Z$ is no 3-cycle and in case of a 4-cycle has exactly one edge on the side of $X^c$. This configuration is depicted in
Figure~\ref{fig:4cyc}. The direction of the edges shows the way they are traversed in a facial walk with the face on the left. As we are on an orientable surface,
in a facial walk one always goes from an edge
$e_{i,R}$ to an edge $e_{i',L}$. Note that incoming and outgoing edges at a vertex can not correspond to the same undirected edge.

If all outgoing edges are different from incoming edges, or if the one possible outgoing edge identical to an incoming edge is present, but the facial walks of faces in $C$ contain additional edges
not incident with vertices of $Z$, the sum of the lengths of the faces in $C$
is at least $12$ -- contradicting the assumption. 
These arguments imply that -- up to symmetry -- $e_{1,R}$ must be the same as $e_{2,L}$. If $e_{2,R}$ would go to $e_{3,L}$ or $e_{4,L}$ in the facial walk,
$f$ or $f'$ would have a double edge in the dual contradicting the fact that $G^*$ is simple. So it must go to $e_{1,L}$, but -- as
$G$ is simple -- this is only possible with an additional edge in between, which is again a contradiction.

\begin{figure}
  \begin{center}
  \begin{minipage}{5cm}
\begin{center}
\includegraphics[width=40mm]{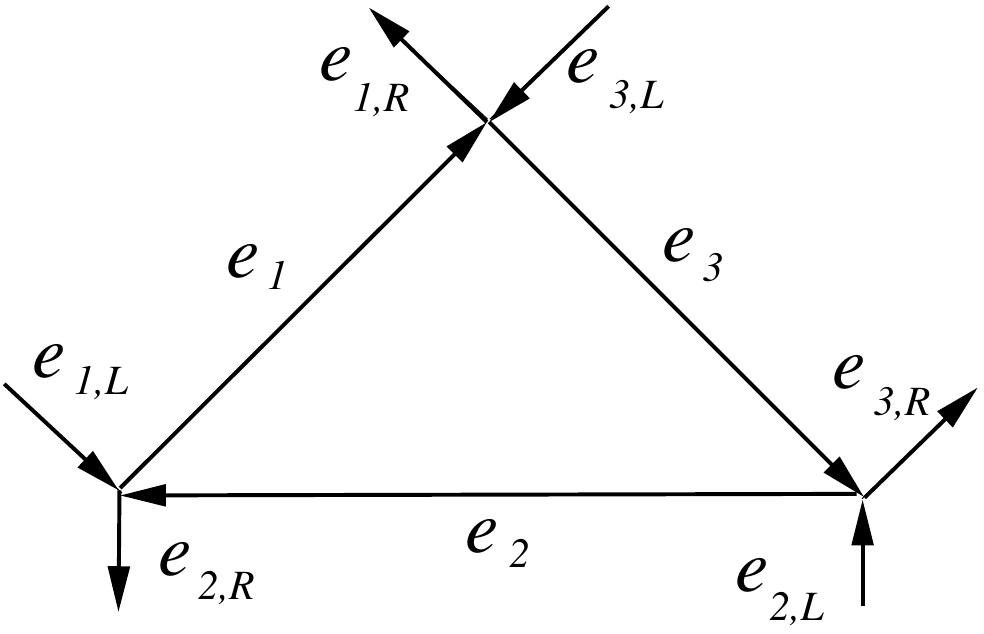}
\end{center}
  \end{minipage}
  \begin{minipage}{5cm}
\begin{center}
\includegraphics[width=40mm]{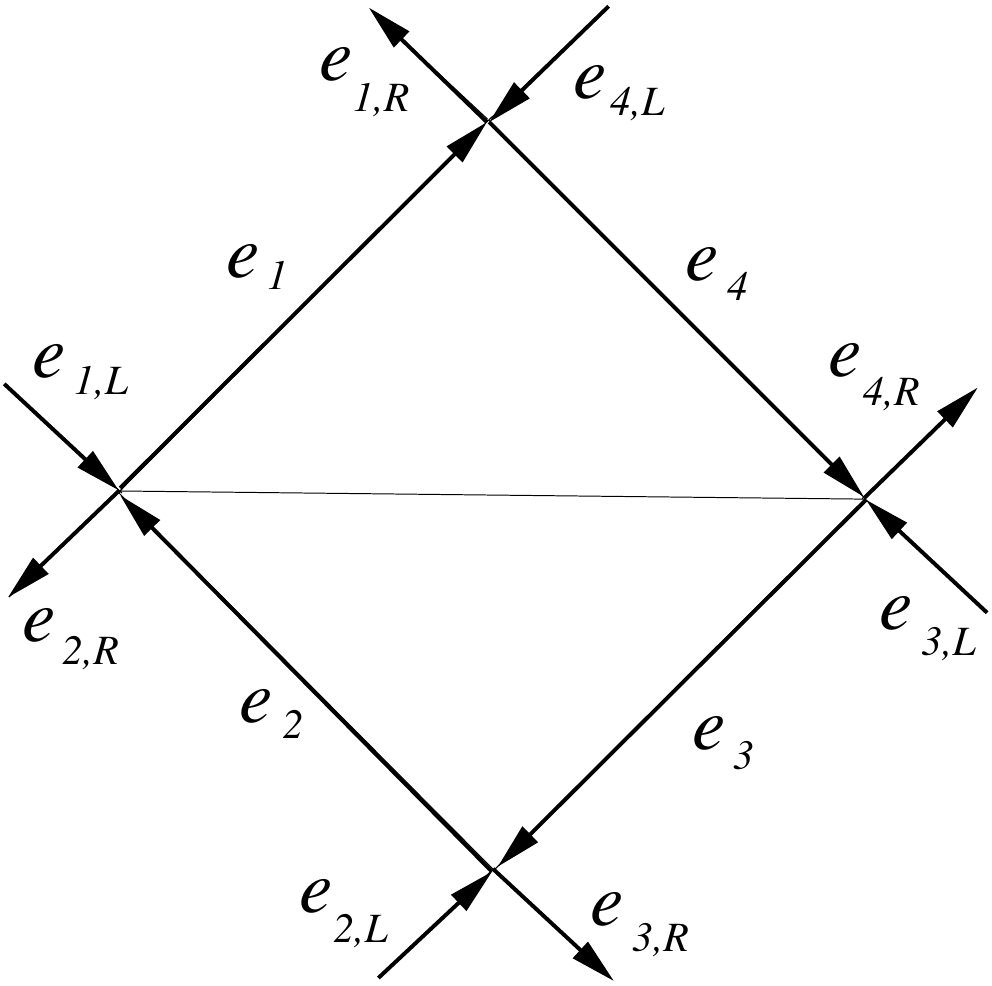}
\end{center}
  \end{minipage}
  \end{center}
  \caption{The configurations in case of a dual-separating 3- or 4-cycle. The configuration for the $4$-cycle is for no vertex on the side not containing the cut-faces. Internal edges are drawn on the outside for better visibility.}\label{fig:4cyc}
\end{figure}

The configuration for a dual-separating 5-cycle is depicted in Figure~\ref{fig:5cyc}(a). We first assume that all five vertices of $Z$ are incident to an interior edge.
Figure~\ref{fig:5cyc}(b) shows the possibilities how boundary vertices can be directly connected without creating double edges. E.g.:\ a directed edge
from $e_1$ to $e_2$ denotes that $e_{1,R}$ can be the same edge as $e_{2,L}$, so that the boundary walk can be $\dots , e_1,e_{1,R}=e_{2,L},e_2,\dots$. Even if
all edges of the cut-faces not on $Z$ would connect vertices of $Z$, the sum of the lengths would already be 10. To this end we know that there are two
cut-faces, as one cut-face together with an arbitrary neighbouring face would already have length at least 13. Furthermore this implies that there is at most one
occurrence of a vertex not on $Z$ in the cyclic walks of the cut-faces. The edges of the facial walks with both endpoints on $Z$ induce two directed cycles or
one directed cycle and one directed path in the directed graph in Figure~\ref{fig:5cyc}(b) containing all five oriented edges that the cut-faces share with $Z$ as vertices.
A directed edge between these vertices indicates that they can follow each other (with one intermediate edge) in the facial walk of the cut-faces.
As there is no directed 2-cycle, there
must be one cycle and one path and we can conclude that the cut-faces contain a vertex not on $Z$. Edges drawn parallel in Figure~\ref{fig:5cyc} (b) can not be
contained in the same cycle or path, as they would correspond to edges with the same endpoints -- this is illustrated for $v_1$ and $v_3$. As there are no
double edges, they would correspond to the same edge -- traversed in opposite direction -- and therefore to a loop in the dual graph. As all directed $4$-cycles contain such edges,
we do not have a directed $4$-cycle.

\begin{figure}
\begin{center}
\includegraphics[width=130mm]{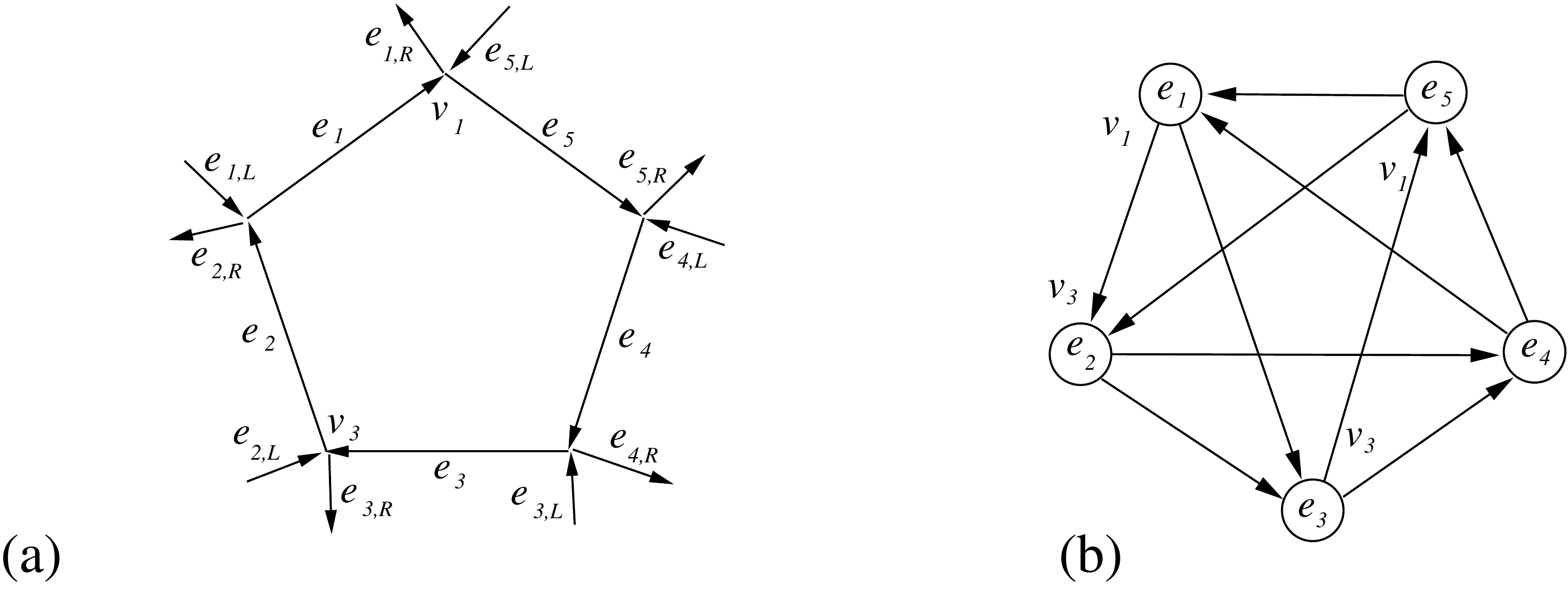}
\end{center}
\caption{The configuration in case of a dual-separating 5-cycle. The right hand side shows how boundary vertices can be directly connected without creating
  double edges. E.g.:\ a directed edge from $e_1$ to $e_2$ denotes that $e_{1,R}$ can be the same edge as $e_{2,L}$, so that the boundary walk can be
  $\dots ,  e_1,e_{1,R}=e_{2,L},e_2,\dots$ }\label{fig:5cyc}
\end{figure}

These arguments imply that we have one directed cycle with three vertices (corresponding to a face of size $6$) and a directed path with two vertices (corresponding
to a face of size $5$ containing one vertex not on $Z$). Up to symmetry we can assume that these are $e_1\to e_2 \to e_4\to e_1$ and $e_3\to e_5$. This means
that the edge $\{v_1,v_3\}$ has a cut-face on both sides and that $e_{1,R}=e_{5,L}$, so that $v_1$ and $v_3$ have only one -- and in fact common -- edge to the side containing the cut-faces.
Removing $v_2,v_4$, and $v_5$, the vertices $v_1$ and $v_3$ are separated from the internal vertices -- e.g.\ from the vertex of the cut-face that is not on $Z$ -- contradicting the
$5$-connectivity of $G$.

\medskip

So assume now that there is a vertex $v$ of $Z$ without an edge of a cut-face that is not on $Z$. As $G$ is 5-connected, the other four vertices form no cutset,
so at least one of the parts separated by $C$ must not contain a vertex. As $v$ has degree at least $5$, it must have a neighbour not on $Z$, so the empty part
must be $X$, but due to Remark~\ref{rem:emptyfaces} (ii) $X$ can not be empty.

\medskip

The last remaining case is that there are two cut-faces $f,f'$ and $d(f)+d(f') > 11$. In this case we have $d(f)+d(f') \le 16$ and that $f$ and $f'$ are not doubly intersecting.  Then, by
Remark~\ref{rem:cutsize}, there exists a set $X^*$ containing $C^*$, so that $E(X^*,(X^*)^c)$ is a $k$-edge cut with $k \le 8$ in $G^*$, in which all the edges
are incident to a vertex of $C^*$. By Lemma~\ref{lem:cut_dual} and as $G$ is simple, these edges correspond to one dual separating cyclic walk or to two cyclic walks $Z,Z'$ in $G$
that are dual separating.  With the same arguments as for $c=3$, we get that there are two cyclic walks and that $Z$ is neighbouring $f$ and $Z'$ is neighbouring $f'$ on one
side.  W.l.o.g.\ we either have a 3-cycle $Z$ and a 3-, 4- or 5-cycle $Z'$, or two 4-cycles.

We will first show that the faces of $X^*$ also contain vertices not on $Z$ or $Z'$. Assume to the contrary that there is no such vertex.  The cycle $Z$ can not contain all
edges of $f$, as $f$ must have a neighbour in $X^*$ (as a vertex of $G^*$).  So there is an edge of $f$ that is not on $Z$, but is incident to a vertex $v$
on $Z$. This can not be the only such edge adjacent to $v$, as otherwise there would be a loop in $G^*$. If $Z$ is a 3-cycle, both such edges must have endpoints on
$Z'$, so we have two different vertices on $f$ as well as on $f'$ and therefore $d(f)+d(f') \le 11$ -- in contradiction to the assumptions in this case.  If $Z$
is a 4-cycle, at most one of the edges can also have the other endpoint on $Z$, so we have at least one endpoint $w$ that is on $f$ and not on $Z$ (so on $Z'$
and also on $f'$).  As $Z'$ is also a 4-cycle, we also get a vertex $w'$ that is on $f'$ and not on $Z'$ (so also on $Z$ and $f$). As these vertices are
different, $f$ and $f'$ are doubly intersecting and again we have a contradiction. So the faces of $X^*$ also contain vertices not on $Z$ or $Z'$.

If both -- $Z$ as well as $Z'$ -- contained at most two vertices with neighbours in $X^*$ and not on $Z$ or $Z'$, we had a 4-cut separating the remaining vertices --
that is: vertices in faces in $(X^*)^c$ and possibly also on $Z,Z'$ -- from the vertices in $X^*$ that are not on $Z,Z'$.  So at least one of the cycles -- call it $Z''$ (adjacent to
face $f''$) -- has three vertices with edges not on $Z''$ but in a cut-face. At these vertices $f''$ has at least two edges not on $Z''$. With $l(Z'')$ the
length of $Z''$ and $d(Z'')$ the maximum number of edges that can connect vertices in a set of three fixed vertices of $Z''$, a lower bound for the size of $f$ is
$l(Z'')+6-d(Z'')$.  We have $l(Z'')\in \{3,4,5\}$ and $d(Z'')=l(Z'')-3$ in each of these cases, so we get 9 as a lower bound for $d(f'')$ and together with an
arbitrary face $f'''$ sharing an edge with $f''$ we get the contradiction $d(f'')+d(f''') > 11$.

\end{proof}

As a corollary we have

\begin{cor}
  All $3$-connected maps with maximum face size $4$ (e.g.\ quadrangulations) as well as all $5$-connected maps with maximum face size $5$ (e.g.\ pentagulations)
  with a simple dual have a 3-connected simple dual.
\end{cor}

Of course Theorem~\ref{thm:mainthm} can also guarantee $3$-connectedness in case of larger faces if they are properly distributed among smaller faces.
  
Theorem~\ref{thm:mainthm} can also be used to answer a question from \cite{BBZ}: 

\begin{cor}\label{cor:delta2}
  For all $c>2$: $\delta_2(c)=g_{min}(c)+1$.
\end{cor}

\begin{proof}
  
  In \cite{BBZ} this equality has been shown for $c\le 6$. For all $c>2$ it has been shown that $\delta_2(c)\le g_{min}(c)+1$. By Theorem~\ref{thm:mainthm}, 5-connected
  maps with a simple dual that are not 3-connected must have $f^+ \ge 6$ and by Lemma~\ref{lem:vf_excess} have genus at least $g_{min}(c)+1$ if $c\ge 6$.

  \end{proof}

\section{Guaranteeing 2-connectedness of the dual}

The following simple remark will be used in the proof of Theorem~\ref{thm:delta1}.

\begin{rem}\label{rem:interior}
  For each $c$ and $l\ge 3$, there exists a $c$-connected simple graph with one face $f$ of size $l$ that is a cycle and with all other faces triangles, so that $f$ has no chords
  and each face shares at most one edge with $f$.
\end{rem}

\begin{proof}
  W.l.o.g. let $c>1$.
Let $c'\ge l\cdot (c-1) $ be so that $(c'-2)(c'-3)\equiv 0 \mod 12$. Then
$K_{c'+1}$ has a triangular minimum genus embedding \cite{Ri} and removing one vertex we get a $(c'-1)$-connected, so also $c$-connected, map with one face of size $c'$ and the rest triangles.
The face $f$ of size $c'$ is a simple cycle with vertices $v_0,\dots, v_{c'-1}$. We now add vertices $w_0,\dots ,w_{l-1}$ in $f$ and for $0\le i\le l-2$
connect $w_i$ with $v_{i(c-1)},\dots ,v_{(i+1)(c-1)}$, and $w_{l-1}$ with the remaining vertices as well as $v_0$ and $v_{c'-1}$. As each vertex is connected to at least $c$ vertices
of the $c$-connected graph, the result is $c$-connected. We now add edges to form a cycle $w_0,\dots ,w_{l-1}$. It is easy to see what the rotational order around the new vertices must
be to satisfy the requirements of the remark.
\end{proof}

Similar to the case of duals with a 2-cut, we will also precisely determine which face sizes are necessary to allow $1$-cuts in the dual. In this case  we consider the size of one face and define

\[min^1_f(c)=\left\{ \begin{array}{cc}
  6 & \mbox{ if } c=1 \\
  9 & \mbox{ if } c\in\{2,3\} \\
  10 & \mbox{ if }  c\in\{4,5\} \\
 12 & \mbox{ if } c=6 \\
 14 & \mbox{ if } c=7 \\
  15 & \mbox{ if } c\ge 8 \\
\end{array}
 \right. .
 \]

 \begin{thm}\label{thm:delta1}
   ~
   \vspace*{-0.2cm}
   
   \begin{description}

\item[(i):] For each $c$ there are $c$-connected maps with a single face of size $min^1_f(c)$ and all other faces triangles, for which the dual is simple and has a $1$-cut.
\item[(ii):] If $G=(V,E)$ is a $c$-connected map with a simple dual $G^*$ and every face of $G$ has size smaller than $min^1_f(c)$, then $G^*$ is 2-connected.
  
     \end{description}
 \end{thm}

 \begin{proof}

   For part (i) we will give explicit constructions of such maps. For $c=1$, two plane copies of $K_4$ identified at a vertex is an example.
   For $c \in  \{2,4\}$ the result will follow from the cases $c=3$, resp.\ $c=5$.
   For $c \in  \{3,5,6,7\}$
   we will use the construction visualized for $c=5$ in Figure~\ref{fig:cyclicconstr}: Let $Z_c$ be the map (possibly first with multi-edges) with vertex
   set $\{0,\dots, c-1\}$ where the rotational order around vertex $i$ is (modulo $c$) $i-2,i-1,i+1,i+2$. For $c=3$ this construction produces three pairs of
   double edges. We subdivide the edges $\{i,i+2\}$. In each case we get one face (we call {\em central}) of size $9$ if $c=3$ and size $2c$ otherwise
   containing all edges.  Except for the central face, there is one hexagon and one triangle for $c=3$, two faces of size $c$ if $c\in \{5,7\}$ and one hexagon
   and two triangles for $c=6$. All faces different from the central face share only edges with the central face. For odd $c$ we can now identify each
   non-central face of length $l$ with the corresponding $l$-face in a separate copy of the $c$-connected maps from Remark~\ref{rem:interior}. As each such
   non-central face contains all vertices of $Z_c$, the result is $c$-connected, but in the dual the vertex corresponding to the central face is a
   cut-vertex. Furthermore all faces except the central face are triangles and share at most one edge with the central face, so that the dual is simple. For
   $c=6$ we do the same for the hexagon, but for the two triangles (which together also contain all vertices) we identify them with two triangles at distance at least $2$
   of a sufficiently large
   $6$-connected map on the torus.

   \begin{figure}
     \begin{center}
       \includegraphics[width=35mm]{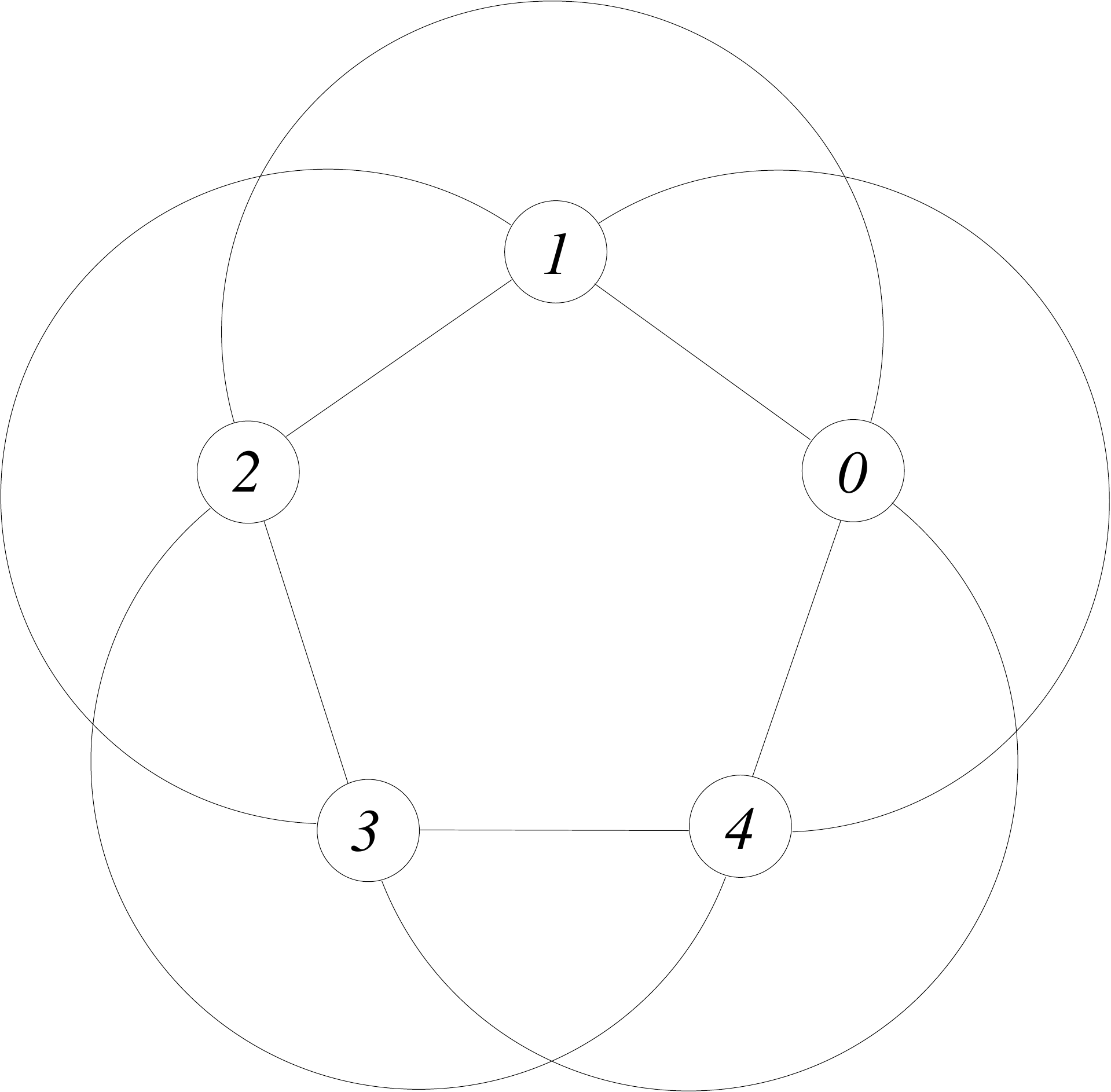}
     \end{center}
     \caption{The central part of a construction of maps with a cut-vertex in the dual, visualized for $c=5$.}\label{fig:cyclicconstr}
   \end{figure}
   
   For $c\ge 8$ we use the empty $9$-cycle in Figure~\ref{fig:15face}. It has triangles, one spanning $9$-gon $f_e$, and one $15$-gon $f$. It is a map of genus $3$.
   The face $f_e$ shares only edges with $f$. 
   The dual of this graph has double edges, but all of them are incident to the vertex corresponding to $f_e$, 
   so the map is in fact an empty $9$-cycle.
   Identifying this $9$-gon with a map from  Remark~\ref{rem:interior} for $c$ and $l=9$, we get a $c$-connected map with a simple dual where $f$ is a cut-vertex separating
   the triangles in the empty $9$-cycle from the other faces. By construction, all
   faces except $f$ -- that has size $15$ -- are triangles.

 \begin{figure}
\begin{center}
\includegraphics[width=90mm]{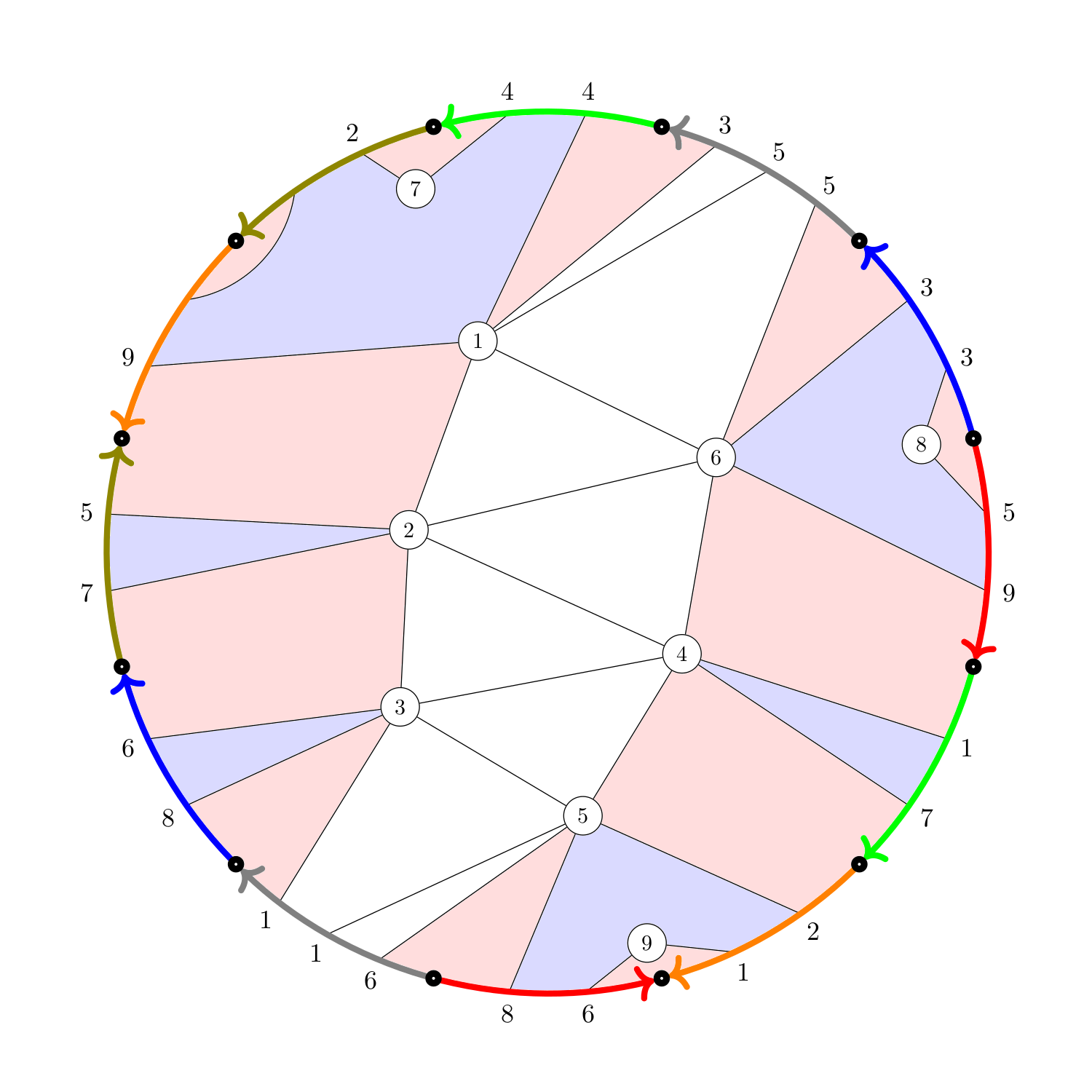}
\end{center}
\caption{An empty $9$-cycle with genus $3$. It has a $15$-gon (red) and a spanning $9$-gon (blue) sharing only edges with the $15$-gon. All other faces are triangles. In the (not simple) dual the $15$-gon is a cut-vertex.
  All double edges in the dual are between the $9$-gon and the $15$-gon.  }
\label{fig:15face}
\end{figure}

 \bigskip

 For part (ii) the case $c=1$ is trivial as no maps with a simple dual, a 1-cut, and all faces smaller than $6$ exist, as all facial walks in the components have length at least $3$. For $1$-connected maps that are also $2$-connected, it follows from the case $c=2$.
 
So let $c\ge 2$.
Assume that a $c$-connected map with a simple dual $G^*$ with a cut-vertex $f^*$ is given and that $f$ is the face corresponding to $f^*$ with
size $d< min^1_f(c)$.
Due to Remark~\ref{rem:cutsize} 
 there exists a set $X^*$ containing $f^*$, so that $E(X^*,(X^*)^c)$
  is a $k$-edge cut $K^*$ in $G^*$ with $k\le d/2\le 7$, in which all the edges are incident to $f^*$. Let $X$, resp. $X^c$ be the maps formed by all vertices and edges of
  faces of $G$ that are in $X^*$, resp. $(X^*)^c$. Due to Lemma~\ref{lem:cut_dual}, the corresponding dual-separating set $K$ can be decomposed into two cycles $Z_1,Z_2$ or one cyclic walk $Z$.
  The map $X$ has a face not containing an edge of $K$ and in $X^c$ each face shares at most one edge with $K$, as on the side of $X$ there is always the same face $f$.

  If $X$ is not an empty circuit or pair, it contains a vertex $v$ not contained in an edge of $K$. Then there are -- even if $X^c$ is empty -- at least
  $o_1=\min\{k,c\}$ vertices on $K$ adjacent to an {\em interior} edge as each vertex on $K$ has $c$ disjoint paths in $G$ to $v$. If there is no path starting
  with an interior edge, there are $c$ paths via vertices of $K$ and the last vertices of $K$ on these paths must be pairwise disjoint and there must be at
  least $c$ of them. Nevertheless it is also in case of $c=2$ not possible that there are just $2$ vertices with an interior edge (no matter whether $X$ is
  empty or not), as a vertex with an interior edge must have more than one interior edge, as the edge would otherwise correspond to a loop in the dual. So let $o=\max\{o_1,3\}$.

  Due to Remark~\ref{rem:emptyfaces} (i),(ii) for $c\in \{2,3,4,5,6\}$ (so $k\le 5$) neither $X$ nor $X^c$ are empty, which gives an immediate contradiction for $c\in \{5,6\}$, as in
  these cases we have one cycle $Z$, $l(Z)<c$, and each path connecting vertices in the two parts must pass a vertex on $K$.

  For $c=4$ we get (also using Remark~\ref{rem:emptyfaces}) $l(Z)=4$ and as no edge can have $f$ on both sides, at most two of the outgoing edges can be
  adjacent to other vertices of $Z$ and we get $d(f)\ge 4 +4 +(4-2)=10$ -- a contradiction.

For $c\in \{2,3\}$ we get in case of $l(Z)=3$ that  $d(f)\ge 3 + 6=9 $, as none of the outgoing edges can go directly to a vertex on $Z$ without creating a double edge. If $l(Z)=4$
we get $d(f)\ge 4 + 3 + (3-1) =9 $ in case of $3$ vertices with an interior edge (note that in this case only one interior edge can connect vertices of $Z$), resp.\ $d(f)\ge 4 + 4 +2 =10 $ in case
of $4$ vertices with an interior edge -- in each case a contradiction.

For $c\ge 7$ we get $k< c$, so one of $X$, $X^c$ must be empty and due to Remark~\ref{rem:emptyfaces} (i),(ii) we get $k=6$ in case $c=7$ and $k\in\{6,7\}$ in case
$c=8$. Due to Remark~\ref{rem:emptyfaces} (iii),(iv) in case of one cycle of length $6$ we get that $X^c$ must be empty and contain at least $13$ edges -- so only two edges of
the complete graph $K_6$ can be missing and be interior edges in $f$. Note that outgoing edges in the cyclic walk are pairwise different, as otherwise $G$ had a loop. We get $d(f)\ge 6 + 6 + (6-2) =16$ -- a contradiction. In case of two cycles,
Remark~\ref{rem:emptyfaces} (v),(vi) imply that $X^c$ is empty and has at least $12$ edges, so we get $d(f)\ge 6 + 6 + (6-3) =15$ -- also a contradiction.

The last remaining case is $c\ge 8$ and $k=7$. Again it follows immediately that one of $X,X^c$ must be empty. In case of one $7$-gon, Remark~\ref{rem:emptyfaces} (vii),(viii) 
imply that the empty part is $X^c$ and that it contains at least $15$ edges, so at most $6$ edges are missing from $K_7$. We get $d(f)\ge 7 + 7 + (7-6) =15$ -- a contradiction.

In case of an empty $7$-pair we have due to Remark~\ref{rem:emptyfaces} (ix),(x) that the empty part is again $X^c$ and that it contains $7$ vertices and
at least $14$ edges. In case of more than $14$ edges, we have an immediate contradiction as above. In case of exactly $14$ edges and $7$ vertices, there are $7$ edges between
the $3$-gon and the $4$-gon, so one vertex $v$ of the $3$-gon is already in $X^c$ adjacent to $5$ vertices and at least one of the incoming and outgoing edges of $f$ is not between vertices of
$K$. We have $d(f)\ge 7 + 7 + 1 =15$.

\end{proof}

While for $\delta_2(c)$ already in \cite{BBZ} it was proven that $\delta_2(c)$ is at most one larger than $g_{min}(c)$, the bound for $\delta_1(c)$ proven there
is a constant factor away from $g_{min}(c)$. Nevertheless it was not proven that $\delta_1(c)$ is strictly larger than $\delta_2(c)$ for sufficiently large $c$. This is now an easy corollary of
Theorem~\ref{thm:delta1}:

\begin{cor}\label{cor:delta1}
 For all $c>3$: $\delta_1(c)\ge g_{min}(c)+2$, so especially  $\delta_1(c)>  \delta_2(c)$.
\end{cor}

\begin{proof}
  In \cite{BBZ} the values of $\delta_1(c)$ were determined for $c\le 6$. We have $\delta_1(1)=0=g_{min}(0)$, $\delta_1(c)=1=g_{min}(c)+1$ for $c\in\{2,3\}$.
  For $c\in\{4,5\}$ it is proven that $\delta_1(c)=2=g_{min}(c)+2$ and for $c=6$ we have $\delta_1(c)=3=g_{min}(c)+2$.

  For $c=7$, Theorem~\ref{thm:delta1} implies that for a $7$-connected map with a simple dual with a $1$-cut we have $f^+\ge 11$, so by Lemma~\ref{lem:vf_excess}, the genus is at least

  \[ \left\lceil \frac{(7-2)(7-3)}{12}+\frac{11}{6} \right\rceil = \left\lceil \frac{42}{12} \right\rceil = 4 = g_{min}(7)+2\].

  For $c\ge 8$  Theorem~\ref{thm:delta1} implies  $f^+\ge 12$, so with Lemma~\ref{lem:vf_excess} we have that the genus is at least $g_{min}(c)+2$.

  \end{proof}

Due to the term $\frac{(c-6)v_x}{12}$ in Lemma~\ref{lem:vf_excess} there can only be a general upper bound $g_{min}(c)+d$ for $\delta_1(c)$ 
with $d$ a constant, if -- except for a finite number of values of $c$ -- there are embeddings of $K_{c+1}$ of genus at most $g_{min}(c)+d$ and a simple dual with a
1-cut. So it is especially interesting to study embeddings of the complete graph. We will now give such a constant $d$ -- although unfortunately only for special values of $c$.
Let $K_n-E(K_6)$ denote the graph obtained from $K_n$ by deleting the edges of a complete subgraph on $6$ vertices and let $c(s)=12s+8$. We will use the following theorem,
which has been proven by Guy and Ringel \cite{GR} for $s\ge 4$ and was later improved by Sun \cite[Subsection 6.4]{Sun2020} to $s\ge 2$.

\begin{thm}[\cite{GR, Sun2020}]
\label{triangular}
For $s\ge 2$, $K_{c(s)+1}-E(K_6)$ has a triangular embedding of genus $g_{min}(c(s))-3$.
\end{thm}

By adding edges to such an embedding of $K_{c(s)+1}-E(K_6)$ we will prove:

\begin{thm}
\label{thm:delta1limit}
For $s\ge 2$ we have $\delta_1(c(s))\le g_{min}(c(s))+3$.
\end{thm}

\begin{proof}

  Let $s\ge 2$ and $G=K_{c(s)+1}-E(K_6)$ with a triangular embedding of genus $g_{min}(c(s))-3$. Let $v_1,\dots ,v_6$ be the six pairwise non-adjacent vertices.
  As all faces are triangles, the dual is simple.  As $s\ge 2$, we have $c(s)+1\ge 33$ and we can choose six vertex disjoint triangular faces $t_1,\dots ,t_6$
  with for $1\le i\le 6 $ $v_i\in t_i$. We insert the edges of a $6$-cycle into these triangles in the way displayed in Figure~\ref{fig:addc6}. Following the
  oriented edges we see that the six triangles are connected to form one $24$-gon $f_{24}$ and that on the other side of the six new edges there is a $6$-gon
  $f_6$ sharing only edges with the $24$-gon. So we have $6$ new edges and we lose $4$ faces, so that the Euler characteristic is decreased by $10$ and the
  genus is increased by $5$ to $g_{min}(c(s))+2$.

   \begin{figure}
     \begin{center}
       \includegraphics[width=45mm]{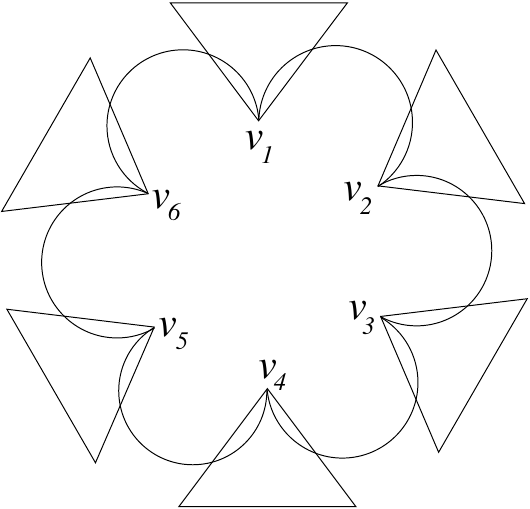}
     \end{center}
     \caption{The rotational order around the vertices when adding a $6$-cycle to an embedding of $K_{c(s)+1}-E(K_6)$. Six triangles are connected to form one $24$-gon. The face on the other side of the new edges is a $6$-gon. In the dual there are only double edges between the $6$-gon and the $24$-gon.}\label{fig:addc6}
   \end{figure}

   Removing one vertex of the embedding of $K_7$ on the torus, we get an embedding of $K_6$ on the torus, which is an empty $6$-circuit with -- except for the
   $6$-gon -- only triangles, so that all faces sharing an edge with the $6$-gon are pairwise different. Identifying this $6$-gon with $f_6$, the dual is simple with $f_{24}$ as a cut-vertex in the dual
   separating
   the triangles of $K_6$ from the other triangles. The genus is $g_{min}(c(s))+3$. Note that no double edges are created, as the graph is $K_{c(s)+1}$: exactly those edges
   have been inserted that were deleted to form $K_{c(s)+1}-E(K_6)$.
  
\end{proof}

For $c\le 6$ the value of $\delta_1(c)$ was already determined in \cite{BBZ}. Except for the first three values, all known values are $g_{min}(c)+2$.
As for large $c$ only complete graphs can possibly
realize the value $g_{min}(c)+2$, it should be noted that for $c<8$, complete graphs can not realize this bound:
computing the number $n_f$ of faces of a complete graph when embedded on genus $g_{min}(c)+2$, we get for $c<8$ that
$n_f\le min^1_f(c)$, so that by Theorem~\ref{thm:delta1} for $c<8$ complete graphs embedded on genus $g_{min}(c)+2$ with a sufficiently large face to allow a $1$-cut in the dual
have multi-edges or loops in the dual. In a final remark we will prove some new values for $\delta_1(c)$.

\begin{rem} ~

  $\delta_1(1)=0=g_{min}(1)$.

  For $2\le c \le 3$ we have $\delta_1(c)=1=g_{min}(c)+1$.

  For $4\le c \le 9$ we have $\delta_1(c)=g_{min}(c)+2$.

  For $10\le c $ we have $\delta_1(c)\ge g_{min}(c)+2$.
  
\end{rem}

\begin{proof}

  For $c\le 6$, the value of $\delta_1(c)$ was already determined in \cite{BBZ}. For $7\le c \le 9$, Corollary~\ref{cor:delta1} together with Figures~\ref{fig:k9_minmatch}, \ref{fig:k9_5}, and 
  \ref{fig:k10_6} prove the remaining statements.

  Note that for $c\ge 8$ the examples are complete graphs.

\end{proof}

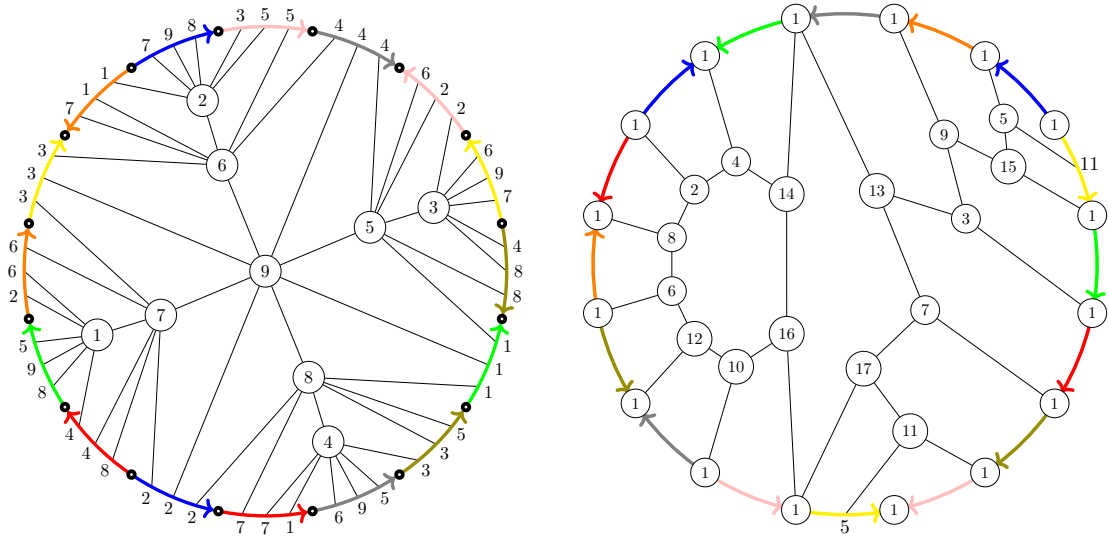
\begin{figure}[tb]
  \begin{center}
  \begin{minipage}{7.5cm}
\begin{center}	
  \resizebox{0.95\textwidth}{0.95\textwidth}
        {
          \input{K9minmaxmatch_gen4.tikz}
        }
\end{center}
  \end{minipage}
  \begin{minipage}{7.5cm}
\begin{center}	
  \resizebox{0.95\textwidth}{0.95\textwidth}
        {
          \input{K9minmaxmatch_gen4_dual.tikz}
        }
\end{center}
  \end{minipage}
  \end{center}
	\caption{ An embedding of $K_9$ minus a maximum matching of genus $4$ and its dual, that is simple and has a $1$-cut.}
	\label{fig:k9_minmatch}
\end{figure}

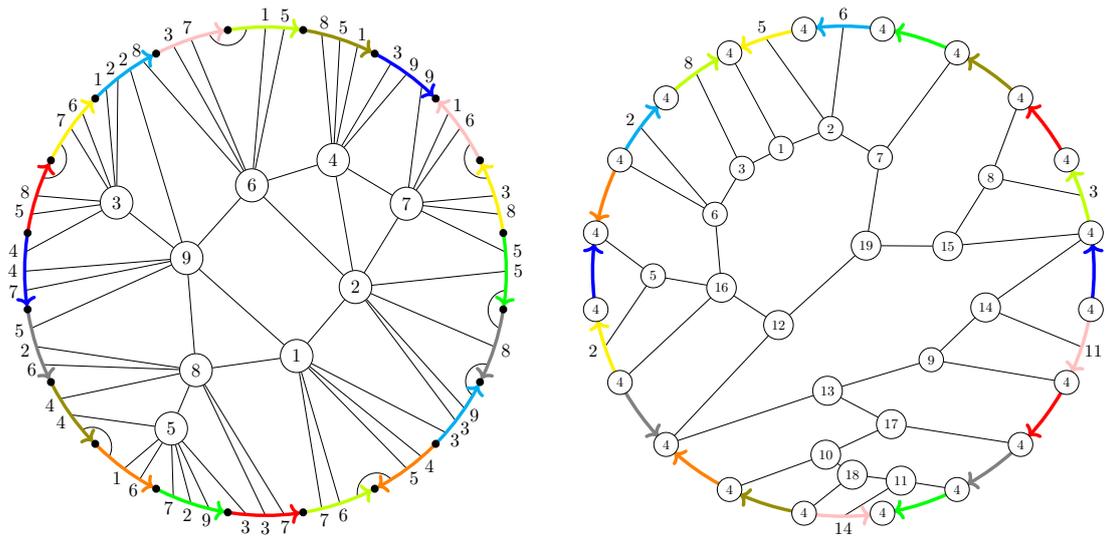
\begin{figure}[tb]
  \begin{center}
  \begin{minipage}{7.5cm}
\begin{center}	
  \resizebox{0.95\textwidth}{0.95\textwidth}
        {
          \input{K9_gen5_dual_1conn.tikz}
        }
\end{center}
  \end{minipage}
  \begin{minipage}{7.5cm}
\begin{center}	
  \resizebox{0.95\textwidth}{0.95\textwidth}
        {
          \input{K9_gen5_dual_1conn_dual.tikz}
        }
\end{center}
  \end{minipage}
  \end{center}
	\caption{ An embedding of $K_9$ of genus $5$ and its dual, that is simple and has a $1$-cut.}
	\label{fig:k9_5}
\end{figure}

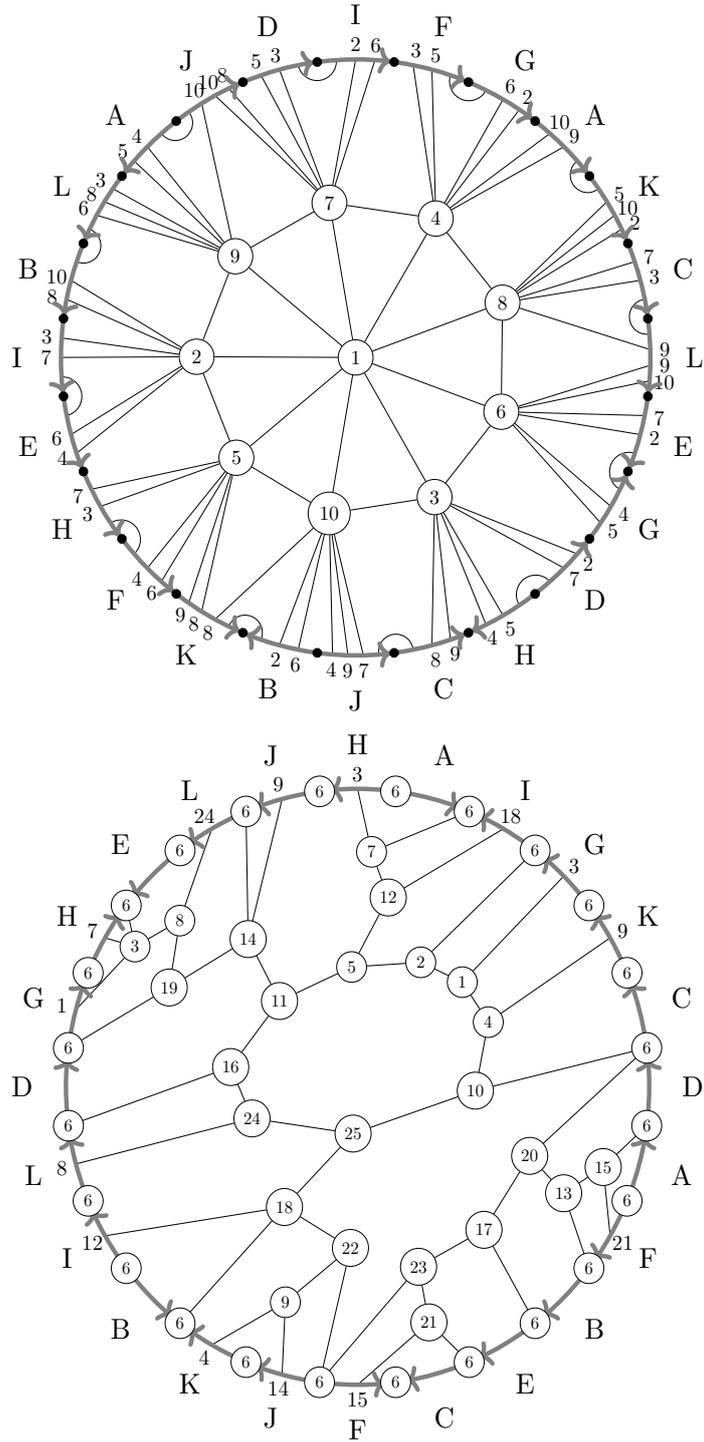
\begin{figure}[tb]
  \begin{center}
  \begin{minipage}{9.6cm}
\begin{center}	
  \resizebox{1.0\textwidth}{1.0\textwidth}
        {
          \input{K10_gen6.tikz}
        }
\end{center}
  \end{minipage}
  \begin{minipage}{9.6cm}
\begin{center}	
  \resizebox{1.0\textwidth}{1.0\textwidth}
        {
          \input{K10_gen6_dual.tikz}
        }
\end{center}
  \end{minipage}
  \end{center}
	\caption{ An embedding of $K_{10}$ of genus $6$ and its dual, that is simple and has a $1$-cut. Due to the large number of edges to be identified, the edges are not coloured, but labeled, to be more easily distinguished.}
	\label{fig:k10_6}
\end{figure}

\section{Conclusion and further work}

We determined exactly which face sizes must be present to allow $1$- or $2$-cuts in simple duals of $c$-connected maps. These results also helped to answer some
questions from \cite{BBZ} and allowed to determine the value of $\delta_2(c)$ for all $c$. For $\delta_1(c)$ we have proven a lower bound that is larger than
$\delta_2(c)$ and have proven an upper bound for some special values of $c$ that deviates by just one from the lower bound. The question whether the lower bound -- or some value close to it -- 
is the exact value of  $\delta_1(c)$ for sufficiently large $c$ is open. In order to solve this problem for large $c$, it is necessary to study embeddings of complete graphs on surfaces
with a genus at least two above the minimum possible genus.

%\vspace{10mm}

\end{document}

%% file: c2_minf_7.tikz
 \begin{tikzpicture}[scale=0.07]
\node [circle,black,draw,scale=1.70] (1) at (-70.71068,-70.71068) {1};
\node [circle,black,draw,scale=1.70] (2) at (70.71068,70.71068) {2};
\node [circle,black,draw,scale=1.70] (3) at (-23.57023,23.57023) {3};
\node [circle,black,draw,scale=1.70] (4) at (-70.71068,70.71068) {4};
\node [circle,black,draw,scale=1.70] (5) at (23.57023,-23.57023) {5};
\node [circle,black,draw,scale=1.70] (6) at (70.71068,-70.71068) {6};
\draw [black] (1) to (3);
\draw [black] (1) to (2);
\draw [black] (1) to (5);
\draw [black] (1) to (6);
\draw [black] (1) to (4);
\draw [black] (2) to (3);
\draw [black] (2) to (4);
\draw [black] (2) to (6);
\draw [black] (2) to (5);
\draw [black] (3) to (4);
\draw [black] (5) to (6);
\end{tikzpicture}

%% file: c2_minf_7_dual.tikz
\begin{tikzpicture}[scale=0.07]
\node [circle,black,draw,scale=2.20] (1) at (-58.77853,-80.90170) {1};
\node [circle,black,draw,scale=2.20] (2) at (-95.10565,30.90170) {2};
\node [circle,black,draw,scale=2.20] (3) at (-30.71669,5.69997) {3};
\node [circle,black,draw,scale=2.20] (4) at (58.77853,-80.90170) {4};
\node [circle,black,draw,scale=2.20] (5) at (-0.00000,100.00000) {5};
\node [circle,black,draw,scale=2.20] (6) at (30.71669,5.69997) {6};
\node [circle,black,draw,scale=2.20] (7) at (95.10565,30.90170) {7};
\draw [black] (1) to (2);
\draw [black] (1) to (3);
\draw [black] (1) to (4);
\draw [black] (2) to (5);
\draw [black] (2) to (3);
\draw [black] (3) to (5);
\draw [black] (4) to (6);
\draw [black] (4) to (7);
\draw [black] (5) to (7);
\draw [black] (5) to (6);
\draw [black] (6) to (7);
\end{tikzpicture}

%% file: v8e20.tikz
\begin{tikzpicture}[scale=0.065]
\node [circle,black,draw,scale=1.40] (1) at (25.03581,37.57013) {1};
\node [circle,black,draw,scale=1.40] (2) at (-35.47443,-20.50699) {2};
\node [circle,black,draw,scale=1.40] (3) at (11.22351,-9.66530) {3};
\node [circle,black,draw,scale=1.40] (4) at (53.85929,-6.46749) {4};
\node [circle,black,draw,scale=1.40] (5) at (5.36809,-53.29914) {5};
\node [circle,black,draw,scale=1.40] (6) at (-12.19856,46.01569) {6};
\node [circle,black,draw,scale=1.40] (7) at (-57.28030,10.01382) {7};
\node [circle,black,draw,scale=1.40] (8) at (-49.13282,49.78166) {8};
\tkzDefPoint(-92.38795,38.26834){9}
\tkzDefPoint(-70.71068,-70.71068){10}
\tkzDefPoint(-100.00000,-0.00000){11}
\tkzDefPoint(-92.38795,-38.26834){12}
\tkzDefPoint(-38.26834,-92.38795){13}
\tkzDefPoint(0.00000,-100.00000){14}
\tkzDefPoint(38.26834,-92.38795){15}
\tkzDefPoint(92.38795,38.26834){16}
\tkzDefPoint(100.00000,0.00000){17}
\tkzDefPoint(92.38795,-38.26834){18}
\tkzDefPoint(-38.26834,92.38795){19}
\tkzDefPoint(-0.00000,100.00000){20}
\tkzDefPoint(38.26834,92.38795){21}
\tkzDefPoint(-70.71068,70.71068){22}
\tkzDefPoint(70.71068,70.71068){23}
\tkzDefPoint(70.71068,-70.71068){24}
\draw [black] (1) to (3);
\draw [black] (1) to (2);
\draw [black] (1) to (6);
\draw [black] (1) to (21);
\node [draw=none,fill=none,scale=1.40] () at (40.18176,97.00735) {5};
\draw [black] (1) to (4);
\draw [black] (2) to (3);
\draw [black] (2) to (5);
\draw [black] (2) to (12);
\node [draw=none,fill=none,scale=1.40] () at (-97.00735,-40.18176) {4};
\draw [black] (2) to (7);
\draw [black] (2) to (6);
\draw [black] (3) to (4);
\draw [black] (3) to (5);
\draw [black] (4) to (16);
\node [draw=none,fill=none,scale=1.40] () at (97.00735,40.18176) {8};
\draw [black] (4) to (17);
\node [draw=none,fill=none,scale=1.40] () at (105.00000,0.00000) {7};
\draw [black] (4) to (18);
\node [draw=none,fill=none,scale=1.40] () at (97.00735,-40.18176) {2};
\draw [black] (4) to (5);
\draw [black] (5) to (15);
\node [draw=none,fill=none,scale=1.40] () at (40.18176,-97.00735) {1};
\draw [black] (5) to (14);
\node [draw=none,fill=none,scale=1.40] () at (0.00000,-105.00000) {6};
\draw [black] (5) to (13);
\node [draw=none,fill=none,scale=1.40] () at (-40.18176,-97.00735) {8};
\draw [black] (6) to (7);
\draw [black] (6) to (8);
\draw [black] (6) to (20);
\node [draw=none,fill=none,scale=1.40] () at (-0.00000,105.00000) {5};
\draw [black] (7) to (11);
\node [draw=none,fill=none,scale=1.40] () at (-105.00000,-0.00000) {4};
\draw [black] (7) to (8);
\draw [black] (8) to (9);
\node [draw=none,fill=none,scale=1.40] () at (-97.00735,40.18176) {4};
\draw [black] (8) to (19);
\node [draw=none,fill=none,scale=1.40] () at (-40.18176,97.00735) {5};
\tkzDefPoint(-68.55786,72.79986){A}
\tkzDefPoint(68.55786,72.79986){B}
\tkzDefPoint(0.0,0.0){C}
\tkzDrawArc[<-,line width=0.9mm, red](C,B)(A)
\tkzDefPoint(72.79986,68.55786){A}
\tkzDefPoint(72.79986,-68.55786){B}
\tkzDefPoint(0.0,0.0){C}
\tkzDrawArc[<-,line width=0.9mm, blue](C,B)(A)
\tkzDefPoint(68.55786,-72.79986){A}
\tkzDefPoint(-68.55786,-72.79986){B}
\tkzDefPoint(0.0,0.0){C}
\tkzDrawArc[->,line width=0.9mm, red](C,B)(A)
\tkzDefPoint(-72.79986,-68.55786){A}
\tkzDefPoint(-72.79986,68.55786){B}
\tkzDefPoint(0.0,0.0){C}
\tkzDrawArc[->,line width=0.9mm, blue](C,B)(A)
\node [black,circle,draw,fill=white,scale=0.75,line width=1mm] (10) at (-70.71068,-70.71068) {};
\node [black,circle,draw,fill=white,scale=0.75,line width=1mm] (22) at (-70.71068,70.71068) {};
\node [black,circle,draw,fill=white,scale=0.75,line width=1mm] (23) at (70.71068,70.71068) {};
\node [black,circle,draw,fill=white,scale=0.75,line width=1mm] (24) at (70.71068,-70.71068) {};
\end{tikzpicture}

%% file: v8e20_dual.tikz
\begin{tikzpicture}[scale=0.065]
\node [circle,black,draw,scale=1.40] (1) at (9.49391,-8.20535) {1};
\node [circle,black,draw,scale=1.40] (2) at (45.30382,11.20229) {2};
\node [circle,black,draw,scale=1.40] (3) at (-12.58107,-44.96505) {3};
\node [circle,black,draw,scale=1.40] (4) at (2.00636,9.80395) {4};
\tkzDefPoint(-70.71068,-70.71068){5}
\node [circle,black,draw,scale=1.40] (6) at (39.65540,-40.10693) {6};
\node [circle,black,draw,scale=1.40] (7) at (-24.29019,-0.40595) {7};
\node [circle,black,draw,scale=1.40] (8) at (22.40516,46.76998) {8};
\node [circle,black,draw,scale=1.40] (9) at (-50.73203,-26.28935) {9};
\node [circle,black,draw,scale=1.40] (10) at (-21.91272,51.21376) {10};
\node [circle,black,draw,scale=1.40] (11) at (-54.80374,24.43763) {11};
\node [circle,black,draw,scale=1.40] (12) at (-30.77933,26.84002) {12};
\tkzDefPoint(-70.71068,70.71068){13}
\tkzDefPoint(70.71068,70.71068){14}
\tkzDefPoint(70.71068,-70.71068){15}
\draw [black] (1) to (2);
\draw [black] (1) to (3);
\draw [black] (1) to (4);
\draw [black] (2) to (14);
\draw [black] (2) to (6);
\draw [black] (3) to (6);
\draw [black] (3) to (5);
\draw [black] (4) to (7);
\draw [black] (4) to (8);
\draw [black] (5) to (9);
\draw [black] (6) to (15);
\draw [black] (7) to (9);
\draw [black] (7) to (12);
\draw [black] (8) to (10);
\draw [black] (8) to (14);
\draw [black] (9) to (11);
\draw [black] (10) to (13);
\draw [black] (10) to (12);
\draw [black] (11) to (13);
\draw [black] (11) to (12);
\tkzDefPoint(-78.55637,-66.08825){A}
\tkzDefPoint(-74.15637,67.08825){B}
\tkzDefPoint(0.0,0.0){C}
\tkzDrawArc[->,line width=0.9mm, blue](C,B)(A)
\tkzDefPoint(-67.08825,74.15637){A}
\tkzDefPoint(65.08825,76.15637){B}
\tkzDefPoint(0.0,0.0){C}
\tkzDrawArc[<-,line width=0.9mm, red](C,B)(A)
\tkzDefPoint(74.15637,67.08825){A}
\tkzDefPoint(76.15637,-65.08825){B}
\tkzDefPoint(0.0,0.0){C}
\tkzDrawArc[<-,line width=0.9mm, blue](C,B)(A)
\tkzDefPoint(65.08825,-76.15637){A}
\tkzDefPoint(-67.08825,-74.15637){B}
\tkzDefPoint(0.0,0.0){C}
\tkzDrawArc[->,line width=0.9mm, red](C,B)(A)
\node [circle,black,draw,fill=white,scale=1.40] (5) at (-70.71068,-70.71068) {5};
\node [circle,black,draw,fill=white,scale=1.40] (13) at (-70.71068,70.71068) {5};
\node [circle,black,draw,fill=white,scale=1.40] (14) at (70.71068,70.71068) {5};
\node [circle,black,draw,fill=white,scale=1.40] (15) at (70.71068,-70.71068) {5};
\end{tikzpicture}

%% file: k9minmaxmatch_gen4.tikz
\begin{tikzpicture}[scale=0.065]
\def\vertexscale{1.25}
\def\labelscale{1.25}
\node [circle,black,draw,scale=\vertexscale] (1) at (-69.74807,-26.03880) {1};
\node [circle,black,draw,scale=\vertexscale] (2) at (-26.03880,69.74807) {2};
\node [circle,black,draw,scale=\vertexscale] (3) at (69.74807,26.03880) {3};
\node [circle,black,draw,scale=\vertexscale] (4) at (26.03880,-69.74807) {4};
\node [circle,black,draw,scale=\vertexscale] (5) at (43.35887,17.96130) {5};
\node [circle,black,draw,scale=\vertexscale] (6) at (-17.96130,43.35887) {6};
\node [circle,black,draw,scale=\vertexscale] (7) at (-43.35887,-17.96130) {7};
\node [circle,black,draw,scale=\vertexscale] (8) at (17.96130,-43.35887) {8};
\node [circle,black,draw,scale=\vertexscale] (9) at (-0.00000,0.00000) {9};
\tkzDefPoint(-47.13967,-88.19213){10}
\tkzDefPoint(-19.50903,-98.07853){11}
\tkzDefPoint(-38.26834,-92.38795){12}
\tkzDefPoint(-29.02847,-95.69403){13}
\tkzDefPoint(-9.80171,-99.51847){14}
\tkzDefPoint(0.00000,-100.00000){15}
\tkzDefPoint(9.80171,-99.51847){16}
\tkzDefPoint(88.19213,-47.13967){17}
\tkzDefPoint(92.38795,-38.26834){18}
\tkzDefPoint(95.69403,-29.02847){19}
\tkzDefPoint(-63.43933,77.30105){20}
\tkzDefPoint(-70.71068,70.71068){21}
\tkzDefPoint(-77.30105,63.43933){22}
\tkzDefPoint(77.30105,-63.43933){23}
\tkzDefPoint(70.71068,-70.71068){24}
\tkzDefPoint(63.43933,-77.30105){25}
\tkzDefPoint(47.13967,-88.19213){26}
\tkzDefPoint(38.26834,-92.38795){27}
\tkzDefPoint(29.02847,-95.69403){28}
\tkzDefPoint(63.43933,77.30105){29}
\tkzDefPoint(70.71068,70.71068){30}
\tkzDefPoint(77.30105,63.43933){31}
\tkzDefPoint(95.69403,29.02847){32}
\tkzDefPoint(92.38795,38.26834){33}
\tkzDefPoint(88.19213,47.13967){34}
\tkzDefPoint(-47.13967,88.19213){35}
\tkzDefPoint(-38.26834,92.38795){36}
\tkzDefPoint(-29.02847,95.69403){37}
\tkzDefPoint(-63.43933,-77.30105){38}
\tkzDefPoint(-70.71068,-70.71068){39}
\tkzDefPoint(-77.30105,-63.43933){40}
\tkzDefPoint(29.02847,95.69403){41}
\tkzDefPoint(38.26834,92.38795){42}
\tkzDefPoint(47.13967,88.19213){43}
\tkzDefPoint(99.51847,9.80171){44}
\tkzDefPoint(100.00000,0.00000){45}
\tkzDefPoint(99.51847,-9.80171){46}
\tkzDefPoint(-88.19213,-47.13967){47}
\tkzDefPoint(-92.38795,-38.26834){48}
\tkzDefPoint(-95.69403,-29.02847){49}
\tkzDefPoint(-95.69403,29.02847){50}
\tkzDefPoint(-92.38795,38.26834){51}
\tkzDefPoint(-88.19213,47.13967){52}
\tkzDefPoint(-9.80171,99.51847){53}
\tkzDefPoint(-0.00000,100.00000){54}
\tkzDefPoint(9.80171,99.51847){55}
\tkzDefPoint(-99.51847,-9.80171){56}
\tkzDefPoint(-100.00000,-0.00000){57}
\tkzDefPoint(-99.51847,9.80171){58}
\tkzDefPoint(-55.55702,-83.14696){59}
\tkzDefPoint(-55.55702,83.14696){60}
\tkzDefPoint(-98.07853,-19.50903){61}
\tkzDefPoint(98.07853,-19.50903){62}
\tkzDefPoint(83.14696,-55.55702){63}
\tkzDefPoint(-83.14696,-55.55702){64}
\tkzDefPoint(19.50903,-98.07853){65}
\tkzDefPoint(19.50903,98.07853){66}
\tkzDefPoint(55.55702,83.14696){67}
\tkzDefPoint(55.55702,-83.14696){68}
\tkzDefPoint(98.07853,19.50903){69}
\tkzDefPoint(-98.07853,19.50903){70}
\tkzDefPoint(-83.14696,55.55702){71}
\tkzDefPoint(83.14696,55.55702){72}
\tkzDefPoint(-19.50903,98.07853){73}
\draw [black] (1) to (49);
\node [draw=none,fill=none,scale=\labelscale] () at (-100.47874,-30.47989) {5};
\draw [black] (1) to (56);
\node [draw=none,fill=none,scale=\labelscale] () at (-104.49440,-10.29180) {2};
\draw [black] (1) to (57);
\node [draw=none,fill=none,scale=\labelscale] () at (-105.00000,-0.00000) {6};
\draw [black] (1) to (7);
\draw [black] (1) to (40);
\node [draw=none,fill=none,scale=\labelscale] () at (-81.16610,-66.61129) {4};
\draw [black] (1) to (47);
\node [draw=none,fill=none,scale=\labelscale] () at (-92.60173,-49.49666) {8};
\draw [black] (1) to (48);
\node [draw=none,fill=none,scale=\labelscale] () at (-97.00735,-40.18176) {9};
\draw [black] (2) to (37);
\node [draw=none,fill=none,scale=\labelscale] () at (-30.47989,100.47874) {8};
\draw [black] (2) to (53);
\node [draw=none,fill=none,scale=\labelscale] () at (-10.29180,104.49440) {3};
\draw [black] (2) to (54);
\node [draw=none,fill=none,scale=\labelscale] () at (-0.00000,105.00000) {5};
\draw [black] (2) to (6);
\draw [black] (2) to (20);
\node [draw=none,fill=none,scale=\labelscale] () at (-66.61129,81.16610) {1};
\draw [black] (2) to (35);
\node [draw=none,fill=none,scale=\labelscale] () at (-49.49666,92.60173) {7};
\draw [black] (2) to (36);
\node [draw=none,fill=none,scale=\labelscale] () at (-40.18176,97.00735) {9};
\draw [black] (3) to (32);
\node [draw=none,fill=none,scale=\labelscale] () at (100.47874,30.47989) {7};
\draw [black] (3) to (44);
\node [draw=none,fill=none,scale=\labelscale] () at (104.49440,10.29180) {4};
\draw [black] (3) to (45);
\node [draw=none,fill=none,scale=\labelscale] () at (105.00000,0.00000) {8};
\draw [black] (3) to (5);
\draw [black] (3) to (31);
\node [draw=none,fill=none,scale=\labelscale] () at (81.16610,66.61129) {2};
\draw [black] (3) to (34);
\node [draw=none,fill=none,scale=\labelscale] () at (92.60173,49.49666) {6};
\draw [black] (3) to (33);
\node [draw=none,fill=none,scale=\labelscale] () at (97.00735,40.18176) {9};
\draw [black] (4) to (28);
\node [draw=none,fill=none,scale=\labelscale] () at (30.47989,-100.47874) {6};
\draw [black] (4) to (16);
\node [draw=none,fill=none,scale=\labelscale] () at (10.29180,-104.49440) {1};
\draw [black] (4) to (15);
\node [draw=none,fill=none,scale=\labelscale] () at (0.00000,-105.00000) {7};
\draw [black] (4) to (8);
\draw [black] (4) to (25);
\node [draw=none,fill=none,scale=\labelscale] () at (66.61129,-81.16610) {3};
\draw [black] (4) to (26);
\node [draw=none,fill=none,scale=\labelscale] () at (49.49666,-92.60173) {5};
\draw [black] (4) to (27);
\node [draw=none,fill=none,scale=\labelscale] () at (40.18176,-97.00735) {9};
\draw [black] (5) to (43);
\node [draw=none,fill=none,scale=\labelscale] () at (49.49666,92.60173) {4};
\draw [black] (5) to (29);
\node [draw=none,fill=none,scale=\labelscale] () at (66.61129,81.16610) {6};
\draw [black] (5) to (30);
\node [draw=none,fill=none,scale=\labelscale] () at (74.24621,74.24621) {2};
\draw [black] (5) to (46);
\node [draw=none,fill=none,scale=\labelscale] () at (104.49440,-10.29180) {8};
\draw [black] (5) to (19);
\node [draw=none,fill=none,scale=\labelscale] () at (100.47874,-30.47989) {1};
\draw [black] (5) to (9);
\draw [black] (6) to (52);
\node [draw=none,fill=none,scale=\labelscale] () at (-92.60173,49.49666) {3};
\draw [black] (6) to (22);
\node [draw=none,fill=none,scale=\labelscale] () at (-81.16610,66.61129) {7};
\draw [black] (6) to (21);
\node [draw=none,fill=none,scale=\labelscale] () at (-74.24621,74.24621) {1};
\draw [black] (6) to (55);
\node [draw=none,fill=none,scale=\labelscale] () at (10.29180,104.49440) {5};
\draw [black] (6) to (41);
\node [draw=none,fill=none,scale=\labelscale] () at (30.47989,100.47874) {4};
\draw [black] (6) to (9);
\draw [black] (7) to (10);
\node [draw=none,fill=none,scale=\labelscale] () at (-49.49666,-92.60173) {2};
\draw [black] (7) to (38);
\node [draw=none,fill=none,scale=\labelscale] () at (-66.61129,-81.16610) {8};
\draw [black] (7) to (39);
\node [draw=none,fill=none,scale=\labelscale] () at (-74.24621,-74.24621) {4};
\draw [black] (7) to (58);
\node [draw=none,fill=none,scale=\labelscale] () at (-104.49440,10.29180) {6};
\draw [black] (7) to (50);
\node [draw=none,fill=none,scale=\labelscale] () at (-100.47874,30.47989) {3};
\draw [black] (7) to (9);
\draw [black] (8) to (17);
\node [draw=none,fill=none,scale=\labelscale] () at (92.60173,-49.49666) {1};
\draw [black] (8) to (23);
\node [draw=none,fill=none,scale=\labelscale] () at (81.16610,-66.61129) {5};
\draw [black] (8) to (24);
\node [draw=none,fill=none,scale=\labelscale] () at (74.24621,-74.24621) {3};
\draw [black] (8) to (14);
\node [draw=none,fill=none,scale=\labelscale] () at (-10.29180,-104.49440) {7};
\draw [black] (8) to (13);
\node [draw=none,fill=none,scale=\labelscale] () at (-30.47989,-100.47874) {2};
\draw [black] (8) to (9);
\draw [black] (9) to (18);
\node [draw=none,fill=none,scale=\labelscale] () at (97.00735,-40.18176) {1};
\draw [black] (9) to (12);
\node [draw=none,fill=none,scale=\labelscale] () at (-40.18176,-97.00735) {2};
\draw [black] (9) to (51);
\node [draw=none,fill=none,scale=\labelscale] () at (-97.00735,40.18176) {3};
\draw [black] (9) to (42);
\node [draw=none,fill=none,scale=\labelscale] () at (40.18176,97.00735) {4};
\tkzDefPoint(-57.17815,-82.04060){A}
\tkzDefPoint(-82.04060,-57.17815){B}
\tkzDefPoint(0.0,0.0){C}
\tkzDrawArc[<-,line width=0.9mm, red](C,B)(A)
\tkzDefPoint(-84.22130,-53.91450){A}
\tkzDefPoint(-97.67676,-21.43015){B}
\tkzDefPoint(0.0,0.0){C}
\tkzDrawArc[<-,line width=0.9mm, green](C,B)(A)
\tkzDefPoint(-98.44252,-17.58040){A}
\tkzDefPoint(-98.44252,17.58040){B}
\tkzDefPoint(0.0,0.0){C}
\tkzDrawArc[<-,line width=0.9mm, orange](C,B)(A)
\tkzDefPoint(-97.67676,21.43015){A}
\tkzDefPoint(-84.22130,53.91450){B}
\tkzDefPoint(0.0,0.0){C}
\tkzDrawArc[<-,line width=0.9mm, yellow](C,B)(A)
\tkzDefPoint(-82.04060,57.17815){A}
\tkzDefPoint(-57.17815,82.04060){B}
\tkzDefPoint(0.0,0.0){C}
\tkzDrawArc[->,line width=0.9mm, orange](C,B)(A)
\tkzDefPoint(-53.91450,84.22130){A}
\tkzDefPoint(-21.43015,97.67676){B}
\tkzDefPoint(0.0,0.0){C}
\tkzDrawArc[<-,line width=0.9mm, blue](C,B)(A)
\tkzDefPoint(-17.58040,98.44252){A}
\tkzDefPoint(17.58040,98.44252){B}
\tkzDefPoint(0.0,0.0){C}
\tkzDrawArc[<-,line width=0.9mm, pink](C,B)(A)
\tkzDefPoint(21.43015,97.67676){A}
\tkzDefPoint(53.91450,84.22130){B}
\tkzDefPoint(0.0,0.0){C}
\tkzDrawArc[<-,line width=0.9mm, gray](C,B)(A)
\tkzDefPoint(57.17815,82.04060){A}
\tkzDefPoint(82.04060,57.17815){B}
\tkzDefPoint(0.0,0.0){C}
\tkzDrawArc[->,line width=0.9mm, pink](C,B)(A)
\tkzDefPoint(84.22130,53.91450){A}
\tkzDefPoint(97.67676,21.43015){B}
\tkzDefPoint(0.0,0.0){C}
\tkzDrawArc[->,line width=0.9mm, yellow](C,B)(A)
\tkzDefPoint(98.44252,17.58040){A}
\tkzDefPoint(98.44252,-17.58040){B}
\tkzDefPoint(0.0,0.0){C}
\tkzDrawArc[<-,line width=0.9mm, olive](C,B)(A)
\tkzDefPoint(97.67676,-21.43015){A}
\tkzDefPoint(84.22130,-53.91450){B}
\tkzDefPoint(0.0,0.0){C}
\tkzDrawArc[->,line width=0.9mm, green](C,B)(A)
\tkzDefPoint(82.04060,-57.17815){A}
\tkzDefPoint(57.17815,-82.04060){B}
\tkzDefPoint(0.0,0.0){C}
\tkzDrawArc[->,line width=0.9mm, olive](C,B)(A)
\tkzDefPoint(53.91450,-84.22130){A}
\tkzDefPoint(21.43015,-97.67676){B}
\tkzDefPoint(0.0,0.0){C}
\tkzDrawArc[->,line width=0.9mm, gray](C,B)(A)
\tkzDefPoint(17.58040,-98.44252){A}
\tkzDefPoint(-17.58040,-98.44252){B}
\tkzDefPoint(0.0,0.0){C}
\tkzDrawArc[->,line width=0.9mm, red](C,B)(A)
\tkzDefPoint(-21.43015,-97.67676){A}
\tkzDefPoint(-53.91450,-84.22130){B}
\tkzDefPoint(0.0,0.0){C}
\tkzDrawArc[->,line width=0.9mm, blue](C,B)(A)
\node [black,circle,draw,fill=white,scale=0.49,line width=1mm] (11) at (-19.50903,-98.07853) {};
\node [black,circle,draw,fill=white,scale=0.49,line width=1mm] (59) at (-55.55702,-83.14696) {};
\node [black,circle,draw,fill=white,scale=0.49,line width=1mm] (60) at (-55.55702,83.14696) {};
\node [black,circle,draw,fill=white,scale=0.49,line width=1mm] (61) at (-98.07853,-19.50903) {};
\node [black,circle,draw,fill=white,scale=0.49,line width=1mm] (62) at (98.07853,-19.50903) {};
\node [black,circle,draw,fill=white,scale=0.49,line width=1mm] (63) at (83.14696,-55.55702) {};
\node [black,circle,draw,fill=white,scale=0.49,line width=1mm] (64) at (-83.14696,-55.55702) {};
\node [black,circle,draw,fill=white,scale=0.49,line width=1mm] (65) at (19.50903,-98.07853) {};
\node [black,circle,draw,fill=white,scale=0.49,line width=1mm] (66) at (19.50903,98.07853) {};
\node [black,circle,draw,fill=white,scale=0.49,line width=1mm] (67) at (55.55702,83.14696) {};
\node [black,circle,draw,fill=white,scale=0.49,line width=1mm] (68) at (55.55702,-83.14696) {};
\node [black,circle,draw,fill=white,scale=0.49,line width=1mm] (69) at (98.07853,19.50903) {};
\node [black,circle,draw,fill=white,scale=0.49,line width=1mm] (70) at (-98.07853,19.50903) {};
\node [black,circle,draw,fill=white,scale=0.49,line width=1mm] (71) at (-83.14696,55.55702) {};
\node [black,circle,draw,fill=white,scale=0.49,line width=1mm] (72) at (83.14696,55.55702) {};
\node [black,circle,draw,fill=white,scale=0.49,line width=1mm] (73) at (-19.50903,98.07853) {};
\end{tikzpicture}

%% file: k9minmaxmatch_gen4_dual.tikz
\begin{tikzpicture}[scale=0.065]
\def\vertexscale{1.10}
\def\labelscale{1.30}
\tkzDefPoint(-19.50903,-98.07853){1}
\node [circle,black,draw,scale=\vertexscale] (2) at (-59.76754,29.80231) {2};
\node [circle,black,draw,scale=\vertexscale] (3) at (47.93130,18.18351) {3};
\node [circle,black,draw,scale=\vertexscale] (4) at (-43.26809,41.00102) {4};
\node [circle,black,draw,scale=\vertexscale] (5) at (62.94989,58.18222) {5};
\node [circle,black,draw,scale=\vertexscale] (6) at (-68.75268,-10.71687) {6};
\node [circle,black,draw,scale=\vertexscale] (7) at (31.78681,-18.30642) {7};
\node [circle,black,draw,scale=\vertexscale] (8) at (-68.75540,10.75805) {8};
\node [circle,black,draw,scale=\vertexscale] (9) at (39.29644,51.82029) {9};
\node [circle,black,draw,scale=\vertexscale] (10) at (-43.25201,-40.91710) {10};
\node [circle,black,draw,scale=\vertexscale] (11) at (25.84680,-66.62006) {11};
\node [circle,black,draw,scale=\vertexscale] (12) at (-59.76141,-29.75312) {12};
\node [circle,black,draw,scale=\vertexscale] (13) at (12.62633,29.28264) {13};
\node [circle,black,draw,scale=\vertexscale] (14) at (-23.19104,28.01900) {14};
\node [circle,black,draw,scale=\vertexscale] (15) at (64.82792,38.98912) {15};
\node [circle,black,draw,scale=\vertexscale] (16) at (-23.16981,-27.73544) {16};
\node [circle,black,draw,scale=\vertexscale] (17) at (7.32703,-41.68700) {17};
\tkzDefPoint(0.00000,-100.00000){18}
\tkzDefPoint(92.38795,38.26834){19}
\tkzDefPoint(83.14696,55.55702){20}
\tkzDefPoint(-83.14696,55.55702){21}
\tkzDefPoint(98.07853,-19.50903){22}
\tkzDefPoint(-55.55702,83.14696){23}
\tkzDefPoint(55.55702,83.14696){24}
\tkzDefPoint(-98.07853,-19.50903){25}
\tkzDefPoint(83.14696,-55.55702){26}
\tkzDefPoint(-98.07853,19.50903){27}
\tkzDefPoint(19.50903,98.07853){28}
\tkzDefPoint(-55.55702,-83.14696){29}
\tkzDefPoint(55.55702,-83.14696){30}
\tkzDefPoint(-83.14696,-55.55702){31}
\tkzDefPoint(-19.50903,98.07853){32}
\tkzDefPoint(98.07853,19.50903){33}
\tkzDefPoint(19.50903,-98.07853){34}
\draw [black] (1) to (16);
\draw [black] (1) to (17);
\draw [black] (2) to (21);
\draw [black] (2) to (4);
\draw [black] (2) to (8);
\draw [black] (3) to (22);
\draw [black] (3) to (13);
\draw [black] (3) to (9);
\draw [black] (4) to (23);
\draw [black] (4) to (14);
\draw [black] (5) to (24);
\draw [black] (5) to (19);
\node [draw=none,fill=none,scale=\labelscale] () at (97.00735,40.18176) {11};
\draw [black] (5) to (15);
\draw [black] (6) to (25);
\draw [black] (6) to (8);
\draw [black] (6) to (12);
\draw [black] (7) to (26);
\draw [black] (7) to (17);
\draw [black] (7) to (13);
\draw [black] (8) to (27);
\draw [black] (9) to (28);
\draw [black] (9) to (15);
\draw [black] (10) to (29);
\draw [black] (10) to (12);
\draw [black] (10) to (16);
\draw [black] (11) to (30);
\draw [black] (11) to (18);
\node [draw=none,fill=none,scale=\labelscale] () at (0.00000,-105.00000) {5};
\draw [black] (11) to (17);
\draw [black] (12) to (31);
\draw [black] (13) to (32);
\draw [black] (14) to (32);
\draw [black] (14) to (16);
\draw [black] (15) to (33);
\tkzDefPoint(-24.38653,-96.98091){A}
\tkzDefPoint(-51.33198,-85.81974){B}
\tkzDefPoint(0.0,0.0){C}
\tkzDrawArc[->,line width=0.9mm, pink](C,B)(A)
\tkzDefPoint(-59.64321,-80.26636){A}
\tkzDefPoint(-80.26636,-59.64321){B}
\tkzDefPoint(0.0,0.0){C}
\tkzDrawArc[<-,line width=0.9mm, gray](C,B)(A)
\tkzDefPoint(-85.81974,-51.33198){A}
\tkzDefPoint(-96.98091,-24.38653){B}
\tkzDefPoint(0.0,0.0){C}
\tkzDrawArc[->,line width=0.9mm, olive](C,B)(A)
\tkzDefPoint(-98.93100,-14.58277){A}
\tkzDefPoint(-98.93100,14.58277){B}
\tkzDefPoint(0.0,0.0){C}
\tkzDrawArc[<-,line width=0.9mm, orange](C,B)(A)
\tkzDefPoint(-96.98091,24.38653){A}
\tkzDefPoint(-85.81974,51.33198){B}
\tkzDefPoint(0.0,0.0){C}
\tkzDrawArc[->,line width=0.9mm, red](C,B)(A)
\tkzDefPoint(-80.26636,59.64321){A}
\tkzDefPoint(-59.64321,80.26636){B}
\tkzDefPoint(0.0,0.0){C}
\tkzDrawArc[<-,line width=0.9mm, blue](C,B)(A)
\tkzDefPoint(-51.33198,85.81974){A}
\tkzDefPoint(-24.38653,96.98091){B}
\tkzDefPoint(0.0,0.0){C}
\tkzDrawArc[->,line width=0.9mm, green](C,B)(A)
\tkzDefPoint(-14.58277,98.93100){A}
\tkzDefPoint(14.58277,98.93100){B}
\tkzDefPoint(0.0,0.0){C}
\tkzDrawArc[->,line width=0.9mm, gray](C,B)(A)
\tkzDefPoint(24.38653,96.98091){A}
\tkzDefPoint(51.33198,85.81974){B}
\tkzDefPoint(0.0,0.0){C}
\tkzDrawArc[->,line width=0.9mm, orange](C,B)(A)
\tkzDefPoint(59.64321,80.26636){A}
\tkzDefPoint(80.26636,59.64321){B}
\tkzDefPoint(0.0,0.0){C}
\tkzDrawArc[->,line width=0.9mm, blue](C,B)(A)
\tkzDefPoint(85.81974,51.33198){A}
\tkzDefPoint(96.98091,24.38653){B}
\tkzDefPoint(0.0,0.0){C}
\tkzDrawArc[<-,line width=0.9mm, yellow](C,B)(A)
\tkzDefPoint(98.93100,14.58277){A}
\tkzDefPoint(98.93100,-14.58277){B}
\tkzDefPoint(0.0,0.0){C}
\tkzDrawArc[<-,line width=0.9mm, green](C,B)(A)
\tkzDefPoint(96.98091,-24.38653){A}
\tkzDefPoint(85.81974,-51.33198){B}
\tkzDefPoint(0.0,0.0){C}
\tkzDrawArc[<-,line width=0.9mm, red](C,B)(A)
\tkzDefPoint(80.26636,-59.64321){A}
\tkzDefPoint(59.64321,-80.26636){B}
\tkzDefPoint(0.0,0.0){C}
\tkzDrawArc[<-,line width=0.9mm, olive](C,B)(A)
\tkzDefPoint(51.33198,-85.81974){A}
\tkzDefPoint(24.38653,-96.98091){B}
\tkzDefPoint(0.0,0.0){C}
\tkzDrawArc[<-,line width=0.9mm, pink](C,B)(A)
\tkzDefPoint(14.58277,-98.93100){A}
\tkzDefPoint(-14.58277,-98.93100){B}
\tkzDefPoint(0.0,0.0){C}
\tkzDrawArc[->,line width=0.9mm, yellow](C,B)(A)
\node [circle,black,draw,fill=white,scale=\vertexscale] (1) at (-19.50903,-98.07853) {1};
\node [circle,black,draw,fill=white,scale=\vertexscale] (20) at (83.14696,55.55702) {1};
\node [circle,black,draw,fill=white,scale=\vertexscale] (21) at (-83.14696,55.55702) {1};
\node [circle,black,draw,fill=white,scale=\vertexscale] (22) at (98.07853,-19.50903) {1};
\node [circle,black,draw,fill=white,scale=\vertexscale] (23) at (-55.55702,83.14696) {1};
\node [circle,black,draw,fill=white,scale=\vertexscale] (24) at (55.55702,83.14696) {1};
\node [circle,black,draw,fill=white,scale=\vertexscale] (25) at (-98.07853,-19.50903) {1};
\node [circle,black,draw,fill=white,scale=\vertexscale] (26) at (83.14696,-55.55702) {1};
\node [circle,black,draw,fill=white,scale=\vertexscale] (27) at (-98.07853,19.50903) {1};
\node [circle,black,draw,fill=white,scale=\vertexscale] (28) at (19.50903,98.07853) {1};
\node [circle,black,draw,fill=white,scale=\vertexscale] (29) at (-55.55702,-83.14696) {1};
\node [circle,black,draw,fill=white,scale=\vertexscale] (30) at (55.55702,-83.14696) {1};
\node [circle,black,draw,fill=white,scale=\vertexscale] (31) at (-83.14696,-55.55702) {1};
\node [circle,black,draw,fill=white,scale=\vertexscale] (32) at (-19.50903,98.07853) {1};
\node [circle,black,draw,fill=white,scale=\vertexscale] (33) at (98.07853,19.50903) {1};
\node [circle,black,draw,fill=white,scale=\vertexscale] (34) at (19.50903,-98.07853) {1};
\end{tikzpicture}

%% file: K9_gen5_dual_1conn.tikz
\begin{tikzpicture}[scale=0.065]
\node [circle,black,draw,scale=1.30] (1) at (12.95425,-34.72644) {1};
\node [circle,black,draw,scale=1.30] (2) at (37.40155,-6.46390) {2};
\node [circle,black,draw,scale=1.30] (3) at (-61.84014,28.13013) {3};
\node [circle,black,draw,scale=1.30] (4) at (28.13566,45.51970) {4};
\node [circle,black,draw,scale=1.30] (5) at (-39.12811,-64.44600) {5};
\node [circle,black,draw,scale=1.30] (6) at (-5.65093,35.16008) {6};
\node [circle,black,draw,scale=1.30] (7) at (58.61106,27.63660) {7};
\node [circle,black,draw,scale=1.30] (8) at (-28.84969,-40.53243) {8};
\node [circle,black,draw,scale=1.30] (9) at (-32.71466,5.38368) {9};
\tkzDefPoint(-99.69173,7.84591){10}
\tkzDefPoint(-15.64345,-98.76883){11}
\tkzDefPoint(-100.00000,-0.00000){12}
\tkzDefPoint(-99.69173,-7.84591){13}
\tkzDefPoint(-7.84591,-99.69173){14}
\tkzDefPoint(0.00000,-100.00000){15}
\tkzDefPoint(7.84591,-99.69173){16}
\tkzDefPoint(-38.26834,-92.38795){17}
\tkzDefPoint(-30.90170,-95.10565){18}
\tkzDefPoint(-23.34454,-97.23699){19}
\tkzDefPoint(-52.24986,-85.26402){20}
\tkzDefPoint(-58.77853,-80.90170){21}
\tkzDefPoint(-64.94480,-76.04060){22}
\tkzDefPoint(-85.26402,-52.24986){23}
\tkzDefPoint(-80.90170,-58.77853){24}
\tkzDefPoint(-76.04060,-64.94480){25}
\tkzDefPoint(-92.38795,-38.26834){26}
\tkzDefPoint(-95.10565,-30.90170){27}
\tkzDefPoint(-97.23699,-23.34454){28}
\tkzDefPoint(85.26402,52.24986){29}
\tkzDefPoint(80.90170,58.77853){30}
\tkzDefPoint(76.04060,64.94480){31}
\tkzDefPoint(92.38795,38.26834){32}
\tkzDefPoint(95.10565,30.90170){33}
\tkzDefPoint(97.23699,23.34454){34}
\tkzDefPoint(38.26834,-92.38795){35}
\tkzDefPoint(30.90170,-95.10565){36}
\tkzDefPoint(23.34454,-97.23699){37}
\tkzDefPoint(75.01111,-66.13119){38}
\tkzDefPoint(79.01550,-61.29071){39}
\tkzDefPoint(82.70806,-56.20834){40}
\tkzDefPoint(86.07420,-50.90414){41}
\tkzDefPoint(52.24986,85.26402){42}
\tkzDefPoint(58.77853,80.90170){43}
\tkzDefPoint(64.94480,76.04060){44}
\tkzDefPoint(97.23699,-23.34454){45}
\tkzDefPoint(95.10565,-30.90170){46}
\tkzDefPoint(92.38795,-38.26834){47}
\tkzDefPoint(23.34454,97.23699){48}
\tkzDefPoint(30.90170,95.10565){49}
\tkzDefPoint(38.26834,92.38795){50}
\tkzDefPoint(64.94480,-76.04060){51}
\tkzDefPoint(58.77853,-80.90170){52}
\tkzDefPoint(52.24986,-85.26402){53}
\tkzDefPoint(99.69173,7.84591){54}
\tkzDefPoint(100.00000,0.00000){55}
\tkzDefPoint(99.69173,-7.84591){56}
\tkzDefPoint(-97.23699,23.34454){57}
\tkzDefPoint(-95.10565,30.90170){58}
\tkzDefPoint(-92.38795,38.26834){59}
\tkzDefPoint(-7.84591,99.69173){60}
\tkzDefPoint(-0.00000,100.00000){61}
\tkzDefPoint(7.84591,99.69173){62}
\tkzDefPoint(-66.13119,75.01111){63}
\tkzDefPoint(-61.29071,79.01550){64}
\tkzDefPoint(-56.20834,82.70806){65}
\tkzDefPoint(-50.90414,86.07420){66}
\tkzDefPoint(-85.26402,52.24986){67}
\tkzDefPoint(-80.90170,58.77853){68}
\tkzDefPoint(-76.04060,64.94480){69}
\tkzDefPoint(-38.26834,92.38795){70}
\tkzDefPoint(-30.90170,95.10565){71}
\tkzDefPoint(-23.34454,97.23699){72}
\tkzDefPoint(-98.76883,15.64345){73}
\tkzDefPoint(45.39905,89.10065){74}
\tkzDefPoint(-70.71068,-70.71068){75}
\tkzDefPoint(70.71068,-70.71068){76}
\tkzDefPoint(-70.71068,70.71068){77}
\tkzDefPoint(89.10065,45.39905){78}
\tkzDefPoint(-45.39905,89.10065){79}
\tkzDefPoint(89.10065,-45.39905){80}
\tkzDefPoint(-89.10065,-45.39905){81}
\tkzDefPoint(15.64345,98.76883){82}
\tkzDefPoint(45.39905,-89.10065){83}
\tkzDefPoint(-45.39905,-89.10065){84}
\tkzDefPoint(98.76883,15.64345){85}
\tkzDefPoint(-89.10065,45.39905){86}
\tkzDefPoint(15.64345,-98.76883){87}
\tkzDefPoint(-15.64345,98.76883){88}
\tkzDefPoint(70.71068,70.71068){89}
\tkzDefPoint(-98.76883,-15.64345){90}
\tkzDefPoint(98.76883,-15.64345){91}
\draw [black] (1) to (2);
\draw [black] (1) to (38);
\node [draw=none,fill=none,scale=1.30] () at (78.76166,-69.43775) {3};
\draw [black] (1) to (51);
\node [draw=none,fill=none,scale=1.30] () at (68.19205,-79.84263) {4};
\draw [black] (1) to (52);
\node [draw=none,fill=none,scale=1.30] () at (61.71745,-84.94678) {5};
\draw [black] (1) to (36);
\node [draw=none,fill=none,scale=1.30] () at (32.44678,-99.86093) {6};
\draw [black] (1) to (37);
\node [draw=none,fill=none,scale=1.30] () at (24.51176,-102.09884) {7};
\draw [black] (1) to (8);
\draw [black] (1) to (9);
\draw [black] (2) to (6);
\draw [black] (2) to (4);
\draw [black] (2) to (7);
\draw [black] (2) to (55);
\node [draw=none,fill=none,scale=1.30] () at (105.00000,0.00000) {5};
\draw [black] (2) to (46);
\node [draw=none,fill=none,scale=1.30] () at (99.86093,-32.44678) {8};
\draw [black] (2) to (40);
\node [draw=none,fill=none,scale=1.30] () at (86.84346,-59.01875) {9};
\draw [black] (2) to (39);
\node [draw=none,fill=none,scale=1.30] () at (82.96628,-64.35524) {3};
\draw [black] (3) to (63);
\node [draw=none,fill=none,scale=1.30] () at (-69.43775,78.76166) {1};
\draw [black] (3) to (64);
\node [draw=none,fill=none,scale=1.30] () at (-64.35524,82.96628) {2};
\draw [black] (3) to (9);
\draw [black] (3) to (10);
\node [draw=none,fill=none,scale=1.30] () at (-104.67632,8.23821) {4};
\draw [black] (3) to (57);
\node [draw=none,fill=none,scale=1.30] () at (-102.09884,24.51176) {5};
\draw [black] (3) to (58);
\node [draw=none,fill=none,scale=1.30] () at (-99.86093,32.44678) {8};
\draw [black] (3) to (68);
\node [draw=none,fill=none,scale=1.30] () at (-84.94678,61.71745) {7};
\draw [black] (3) to (69);
\node [draw=none,fill=none,scale=1.30] () at (-79.84263,68.19205) {6};
\draw [black] (4) to (50);
\node [draw=none,fill=none,scale=1.30] () at (40.18176,97.00735) {1};
\draw [black] (4) to (42);
\node [draw=none,fill=none,scale=1.30] () at (54.86235,89.52722) {3};
\draw [black] (4) to (43);
\node [draw=none,fill=none,scale=1.30] () at (61.71745,84.94678) {9};
\draw [black] (4) to (7);
\draw [black] (4) to (6);
\draw [black] (4) to (48);
\node [draw=none,fill=none,scale=1.30] () at (24.51176,102.09884) {8};
\draw [black] (4) to (49);
\node [draw=none,fill=none,scale=1.30] () at (32.44678,99.86093) {5};
\draw [black] (5) to (21);
\node [draw=none,fill=none,scale=1.30] () at (-61.71745,-84.94678) {1};
\draw [black] (5) to (24);
\node [draw=none,fill=none,scale=1.30] () at (-84.94678,-61.71745) {4};
\draw [black] (5) to (8);
\draw [black] (5) to (14);
\node [draw=none,fill=none,scale=1.30] () at (-8.23821,-104.67632) {3};
\draw [black] (5) to (19);
\node [draw=none,fill=none,scale=1.30] () at (-24.51176,-102.09884) {9};
\draw [black] (5) to (18);
\node [draw=none,fill=none,scale=1.30] () at (-32.44678,-99.86093) {2};
\draw [black] (5) to (17);
\node [draw=none,fill=none,scale=1.30] () at (-40.18176,-97.00735) {7};
\draw [black] (5) to (20);
\node [draw=none,fill=none,scale=1.30] () at (-54.86235,-89.52722) {6};
\draw [black] (6) to (61);
\node [draw=none,fill=none,scale=1.30] () at (-0.00000,105.00000) {1};
\draw [black] (6) to (62);
\node [draw=none,fill=none,scale=1.30] () at (8.23821,104.67632) {5};
\draw [black] (6) to (9);
\draw [black] (6) to (66);
\node [draw=none,fill=none,scale=1.30] () at (-53.44935,90.37791) {8};
\draw [black] (6) to (70);
\node [draw=none,fill=none,scale=1.30] () at (-40.18176,97.00735) {3};
\draw [black] (6) to (71);
\node [draw=none,fill=none,scale=1.30] () at (-32.44678,99.86093) {7};
\draw [black] (7) to (31);
\node [draw=none,fill=none,scale=1.30] () at (79.84263,68.19205) {1};
\draw [black] (7) to (30);
\node [draw=none,fill=none,scale=1.30] () at (84.94678,61.71745) {6};
\draw [black] (7) to (33);
\node [draw=none,fill=none,scale=1.30] () at (99.86093,32.44678) {3};
\draw [black] (7) to (34);
\node [draw=none,fill=none,scale=1.30] () at (102.09884,24.51176) {8};
\draw [black] (7) to (54);
\node [draw=none,fill=none,scale=1.30] () at (104.67632,8.23821) {5};
\draw [black] (7) to (44);
\node [draw=none,fill=none,scale=1.30] () at (68.19205,79.84263) {9};
\draw [black] (8) to (16);
\node [draw=none,fill=none,scale=1.30] () at (8.23821,-104.67632) {7};
\draw [black] (8) to (15);
\node [draw=none,fill=none,scale=1.30] () at (0.00000,-105.00000) {3};
\draw [black] (8) to (23);
\node [draw=none,fill=none,scale=1.30] () at (-89.52722,-54.86235) {4};
\draw [black] (8) to (26);
\node [draw=none,fill=none,scale=1.30] () at (-97.00735,-40.18176) {6};
\draw [black] (8) to (27);
\node [draw=none,fill=none,scale=1.30] () at (-99.86093,-32.44678) {2};
\draw [black] (8) to (9);
\draw [black] (9) to (28);
\node [draw=none,fill=none,scale=1.30] () at (-102.09884,-24.51176) {5};
\draw [black] (9) to (13);
\node [draw=none,fill=none,scale=1.30] () at (-104.67632,-8.23821) {7};
\draw [black] (9) to (12);
\node [draw=none,fill=none,scale=1.30] () at (-105.00000,-0.00000) {4};
\draw [black] (9) to (65);
\node [draw=none,fill=none,scale=1.30] () at (-59.01875,86.84346) {2};
\tkzDefPoint(-64.94480,-76.04060){A}
\tkzDefPoint(-66.26314,-66.26314){B}
\tkzDefPoint(-76.04060,-64.94480){C}
\tkzCircumCenter(A,B,C)\tkzGetPoint{D}
\tkzDrawArc[black](D,A)(C)
\tkzDefPoint(85.26402,52.24986){A}
\tkzDefPoint(83.49643,42.54355){B}
\tkzDefPoint(92.38795,38.26834){C}
\tkzCircumCenter(A,B,C)\tkzGetPoint{D}
\tkzDrawArc[black](D,A)(C)
\tkzDefPoint(38.26834,-92.38795){A}
\tkzDefPoint(42.54355,-83.49643){B}
\tkzDefPoint(52.24986,-85.26402){C}
\tkzCircumCenter(A,B,C)\tkzGetPoint{D}
\tkzDrawArc[black](D,C)(A)
\tkzDefPoint(86.07420,-50.90414){A}
\tkzDefPoint(83.87721,-41.91107){B}
\tkzDefPoint(92.38795,-38.26834){C}
\tkzCircumCenter(A,B,C)\tkzGetPoint{D}
\tkzDrawArc[black](D,C)(A)
\tkzDefPoint(97.23699,-23.34454){A}
\tkzDefPoint(92.55650,-14.65951){B}
\tkzDefPoint(99.69173,-7.84591){C}
\tkzCircumCenter(A,B,C)\tkzGetPoint{D}
\tkzDrawArc[black](D,C)(A)
\tkzDefPoint(-92.38795,38.26834){A}
\tkzDefPoint(-83.49643,42.54355){B}
\tkzDefPoint(-85.26402,52.24986){C}
\tkzCircumCenter(A,B,C)\tkzGetPoint{D}
\tkzDrawArc[black](D,A)(C)
\tkzDefPoint(-7.84591,99.69173){A}
\tkzDefPoint(-14.65951,92.55650){B}
\tkzDefPoint(-23.34454,97.23699){C}
\tkzCircumCenter(A,B,C)\tkzGetPoint{D}
\tkzDrawArc[black](D,C)(A)
\tkzDefPoint(-98.56449,16.88314){A}
\tkzDefPoint(-89.66401,44.27601){B}
\tkzDefPoint(0.0,0.0){C}
\tkzDrawArc[<-,line width=0.9mm, red](C,B)(A)
\tkzDefPoint(-88.52323,46.51492){A}
\tkzDefPoint(-71.59350,69.81669){B}
\tkzDefPoint(0.0,0.0){C}
\tkzDrawArc[<-,line width=0.9mm, yellow](C,B)(A)
\tkzDefPoint(-69.81669,71.59350){A}
\tkzDefPoint(-46.51492,88.52323){B}
\tkzDefPoint(0.0,0.0){C}
\tkzDrawArc[<-,line width=0.9mm, cyan](C,B)(A)
\tkzDefPoint(-44.27601,89.66401){A}
\tkzDefPoint(-16.88314,98.56449){B}
\tkzDefPoint(0.0,0.0){C}
\tkzDrawArc[<-,line width=0.9mm, pink](C,B)(A)
\tkzDefPoint(-14.40128,98.95758){A}
\tkzDefPoint(14.40128,98.95758){B}
\tkzDefPoint(0.0,0.0){C}
\tkzDrawArc[<-,line width=0.9mm, lime](C,B)(A)
\tkzDefPoint(16.88314,98.56449){A}
\tkzDefPoint(44.27601,89.66401){B}
\tkzDefPoint(0.0,0.0){C}
\tkzDrawArc[<-,line width=0.9mm, olive](C,B)(A)
\tkzDefPoint(46.51492,88.52323){A}
\tkzDefPoint(69.81669,71.59350){B}
\tkzDefPoint(0.0,0.0){C}
\tkzDrawArc[<-,line width=0.9mm, blue](C,B)(A)
\tkzDefPoint(71.59350,69.81669){A}
\tkzDefPoint(88.52323,46.51492){B}
\tkzDefPoint(0.0,0.0){C}
\tkzDrawArc[->,line width=0.9mm, pink](C,B)(A)
\tkzDefPoint(89.66401,44.27601){A}
\tkzDefPoint(98.56449,16.88314){B}
\tkzDefPoint(0.0,0.0){C}
\tkzDrawArc[->,line width=0.9mm, yellow](C,B)(A)
\tkzDefPoint(98.95758,14.40128){A}
\tkzDefPoint(98.95758,-14.40128){B}
\tkzDefPoint(0.0,0.0){C}
\tkzDrawArc[<-,line width=0.9mm, green](C,B)(A)
\tkzDefPoint(98.56449,-16.88314){A}
\tkzDefPoint(89.66401,-44.27601){B}
\tkzDefPoint(0.0,0.0){C}
\tkzDrawArc[<-,line width=0.9mm, gray](C,B)(A)
\tkzDefPoint(88.52323,-46.51492){A}
\tkzDefPoint(71.59350,-69.81669){B}
\tkzDefPoint(0.0,0.0){C}
\tkzDrawArc[->,line width=0.9mm, cyan](C,B)(A)
\tkzDefPoint(69.81669,-71.59350){A}
\tkzDefPoint(46.51492,-88.52323){B}
\tkzDefPoint(0.0,0.0){C}
\tkzDrawArc[<-,line width=0.9mm, orange](C,B)(A)
\tkzDefPoint(44.27601,-89.66401){A}
\tkzDefPoint(16.88314,-98.56449){B}
\tkzDefPoint(0.0,0.0){C}
\tkzDrawArc[->,line width=0.9mm, lime](C,B)(A)
\tkzDefPoint(14.40128,-98.95758){A}
\tkzDefPoint(-14.40128,-98.95758){B}
\tkzDefPoint(0.0,0.0){C}
\tkzDrawArc[->,line width=0.9mm, red](C,B)(A)
\tkzDefPoint(-16.88314,-98.56449){A}
\tkzDefPoint(-44.27601,-89.66401){B}
\tkzDefPoint(0.0,0.0){C}
\tkzDrawArc[->,line width=0.9mm, green](C,B)(A)
\tkzDefPoint(-46.51492,-88.52323){A}
\tkzDefPoint(-69.81669,-71.59350){B}
\tkzDefPoint(0.0,0.0){C}
\tkzDrawArc[->,line width=0.9mm, orange](C,B)(A)
\tkzDefPoint(-71.59350,-69.81669){A}
\tkzDefPoint(-88.52323,-46.51492){B}
\tkzDefPoint(0.0,0.0){C}
\tkzDrawArc[->,line width=0.9mm, olive](C,B)(A)
\tkzDefPoint(-89.66401,-44.27601){A}
\tkzDefPoint(-98.56449,-16.88314){B}
\tkzDefPoint(0.0,0.0){C}
\tkzDrawArc[->,line width=0.9mm, gray](C,B)(A)
\tkzDefPoint(-98.95758,-14.40128){A}
\tkzDefPoint(-98.95758,14.40128){B}
\tkzDefPoint(0.0,0.0){C}
\tkzDrawArc[->,line width=0.9mm, blue](C,B)(A)
\node [black,circle,draw,fill=white,scale=0.31,line width=1mm] (11) at (-15.64345,-98.76883) {};
\node [black,circle,draw,fill=white,scale=0.31,line width=1mm] (73) at (-98.76883,15.64345) {};
\node [black,circle,draw,fill=white,scale=0.31,line width=1mm] (74) at (45.39905,89.10065) {};
\node [black,circle,draw,fill=white,scale=0.31,line width=1mm] (75) at (-70.71068,-70.71068) {};
\node [black,circle,draw,fill=white,scale=0.31,line width=1mm] (76) at (70.71068,-70.71068) {};
\node [black,circle,draw,fill=white,scale=0.31,line width=1mm] (77) at (-70.71068,70.71068) {};
\node [black,circle,draw,fill=white,scale=0.31,line width=1mm] (78) at (89.10065,45.39905) {};
\node [black,circle,draw,fill=white,scale=0.31,line width=1mm] (79) at (-45.39905,89.10065) {};
\node [black,circle,draw,fill=white,scale=0.31,line width=1mm] (80) at (89.10065,-45.39905) {};
\node [black,circle,draw,fill=white,scale=0.31,line width=1mm] (81) at (-89.10065,-45.39905) {};
\node [black,circle,draw,fill=white,scale=0.31,line width=1mm] (82) at (15.64345,98.76883) {};
\node [black,circle,draw,fill=white,scale=0.31,line width=1mm] (83) at (45.39905,-89.10065) {};
\node [black,circle,draw,fill=white,scale=0.31,line width=1mm] (84) at (-45.39905,-89.10065) {};
\node [black,circle,draw,fill=white,scale=0.31,line width=1mm] (85) at (98.76883,15.64345) {};
\node [black,circle,draw,fill=white,scale=0.31,line width=1mm] (86) at (-89.10065,45.39905) {};
\node [black,circle,draw,fill=white,scale=0.31,line width=1mm] (87) at (15.64345,-98.76883) {};
\node [black,circle,draw,fill=white,scale=0.31,line width=1mm] (88) at (-15.64345,98.76883) {};
\node [black,circle,draw,fill=white,scale=0.31,line width=1mm] (89) at (70.71068,70.71068) {};
\node [black,circle,draw,fill=white,scale=0.31,line width=1mm] (90) at (-98.76883,-15.64345) {};
\node [black,circle,draw,fill=white,scale=0.31,line width=1mm] (91) at (98.76883,-15.64345) {};
\end{tikzpicture}

%% file: K9_gen5_dual_1conn_dual.tikz
\begin{tikzpicture}[scale=0.065]
\node [circle,black,draw,scale=0.94] (1) at (-24.80815,50.14780) {1};
\node [circle,black,draw,scale=0.94] (2) at (-5.06260,58.25455) {2};
\node [circle,black,draw,scale=0.94] (3) at (-40.42462,42.00139) {3};
\tkzDefPoint(-15.64345,-98.76883){4}
\node [circle,black,draw,scale=0.94] (5) at (-75.85253,-1.87427) {5};
\node [circle,black,draw,scale=0.94] (6) at (-51.20564,23.26989) {6};
\node [circle,black,draw,scale=0.94] (7) at (14.65949,46.46503) {7};
\node [circle,black,draw,scale=0.94] (8) at (58.82605,38.37550) {8};
\node [circle,black,draw,scale=0.94] (9) at (35.08359,-36.13631) {9};
\node [circle,black,draw,scale=0.94] (10) at (-6.99116,-74.87861) {10};
\node [circle,black,draw,scale=0.94] (11) at (23.02069,-85.47199) {11};
\node [circle,black,draw,scale=0.94] (12) at (-25.85599,-21.96584) {12};
\node [circle,black,draw,scale=0.94] (13) at (-6.29652,-48.78622) {13};
\node [circle,black,draw,scale=0.94] (14) at (56.78515,-14.79703) {14};
\node [circle,black,draw,scale=0.94] (15) at (41.63936,10.22673) {15};
\node [circle,black,draw,scale=0.94] (16) at (-48.60042,-6.74654) {16};
\node [circle,black,draw,scale=0.94] (17) at (19.14413,-62.42988) {17};
\node [circle,black,draw,scale=0.94] (18) at (3.59979,-83.04678) {18};
\node [circle,black,draw,scale=0.94] (19) at (9.07867,10.43337) {19};
\tkzDefPoint(0.00000,-100.00000){20}
\tkzDefPoint(-30.90170,95.10565){21}
\tkzDefPoint(95.10565,30.90170){22}
\tkzDefPoint(-0.00000,100.00000){23}
\tkzDefPoint(-58.77853,80.90170){24}
\tkzDefPoint(-80.90170,58.77853){25}
\tkzDefPoint(-95.10565,-30.90170){26}
\tkzDefPoint(95.10565,-30.90170){27}
\tkzDefPoint(98.76883,-15.64345){28}
\tkzDefPoint(-98.76883,-15.64345){29}
\tkzDefPoint(-45.39905,89.10065){30}
\tkzDefPoint(89.10065,45.39905){31}
\tkzDefPoint(89.10065,-45.39905){32}
\tkzDefPoint(15.64345,-98.76883){33}
\tkzDefPoint(15.64345,98.76883){34}
\tkzDefPoint(-89.10065,45.39905){35}
\tkzDefPoint(-45.39905,-89.10065){36}
\tkzDefPoint(45.39905,89.10065){37}
\tkzDefPoint(45.39905,-89.10065){38}
\tkzDefPoint(-70.71068,-70.71068){39}
\tkzDefPoint(-98.76883,15.64345){40}
\tkzDefPoint(98.76883,15.64345){41}
\tkzDefPoint(-70.71068,70.71068){42}
\tkzDefPoint(-15.64345,98.76883){43}
\tkzDefPoint(-89.10065,-45.39905){44}
\tkzDefPoint(70.71068,-70.71068){45}
\tkzDefPoint(70.71068,70.71068){46}
\draw [black] (1) to (2);
\draw [black] (1) to (3);
\draw [black] (1) to (30);
\draw [black] (2) to (21);
\node [draw=none,fill=none,scale=1.20] () at (-32.44678,99.86093) {5};
\draw [black] (2) to (23);
\node [draw=none,fill=none,scale=1.20] () at (-0.00000,105.00000) {6};
\draw [black] (2) to (7);
\draw [black] (3) to (6);
\draw [black] (3) to (24);
\node [draw=none,fill=none,scale=1.20] () at (-61.71745,84.94678) {8};
\draw [black] (4) to (18);
\draw [black] (5) to (26);
\node [draw=none,fill=none,scale=1.20] () at (-99.86093,-32.44678) {2};
\draw [black] (5) to (40);
\draw [black] (5) to (16);
\draw [black] (6) to (25);
\node [draw=none,fill=none,scale=1.20] () at (-84.94678,61.71745) {2};
\draw [black] (6) to (16);
\draw [black] (6) to (35);
\draw [black] (7) to (37);
\draw [black] (7) to (19);
\draw [black] (8) to (22);
\node [draw=none,fill=none,scale=1.20] () at (99.86093,32.44678) {3};
\draw [black] (8) to (15);
\draw [black] (8) to (46);
\draw [black] (9) to (32);
\draw [black] (9) to (13);
\draw [black] (9) to (14);
\draw [black] (10) to (36);
\draw [black] (10) to (17);
\draw [black] (10) to (18);
\draw [black] (11) to (38);
\draw [black] (11) to (20);
\node [draw=none,fill=none,scale=1.20] () at (0.00000,-105.00000) {14};
\draw [black] (11) to (18);
\draw [black] (12) to (39);
\draw [black] (12) to (16);
\draw [black] (12) to (19);
\draw [black] (13) to (39);
\draw [black] (13) to (17);
\draw [black] (14) to (41);
\draw [black] (14) to (27);
\node [draw=none,fill=none,scale=1.20] () at (99.86093,-32.44678) {11};
\draw [black] (15) to (41);
\draw [black] (15) to (19);
\draw [black] (16) to (44);
\draw [black] (17) to (45);
\tkzDefPoint(-20.02985,-97.97349){A}
\tkzDefPoint(-41.38289,-91.03547){B}
\tkzDefPoint(0.0,0.0){C}
\tkzDrawArc[<-,line width=0.9mm, olive](C,B)(A)
\tkzDefPoint(-49.32499,-86.98877){A}
\tkzDefPoint(-67.48897,-73.79186){B}
\tkzDefPoint(0.0,0.0){C}
\tkzDrawArc[<-,line width=0.9mm, orange](C,B)(A)
\tkzDefPoint(-73.79186,-67.48897){A}
\tkzDefPoint(-86.98877,-49.32499){B}
\tkzDefPoint(0.0,0.0){C}
\tkzDrawArc[->,line width=0.9mm, gray](C,B)(A)
\tkzDefPoint(-91.03547,-41.38289){A}
\tkzDefPoint(-97.97349,-20.02985){B}
\tkzDefPoint(0.0,0.0){C}
\tkzDrawArc[<-,line width=0.9mm, yellow](C,B)(A)
\tkzDefPoint(-99.36789,-11.22596){A}
\tkzDefPoint(-99.36789,11.22596){B}
\tkzDefPoint(0.0,0.0){C}
\tkzDrawArc[<-,line width=0.9mm, blue](C,B)(A)
\tkzDefPoint(-97.97349,20.02985){A}
\tkzDefPoint(-91.03547,41.38289){B}
\tkzDefPoint(0.0,0.0){C}
\tkzDrawArc[->,line width=0.9mm, orange](C,B)(A)
\tkzDefPoint(-86.98877,49.32499){A}
\tkzDefPoint(-73.79186,67.48897){B}
\tkzDefPoint(0.0,0.0){C}
\tkzDrawArc[<-,line width=0.9mm, cyan](C,B)(A)
\tkzDefPoint(-67.48897,73.79186){A}
\tkzDefPoint(-49.32499,86.98877){B}
\tkzDefPoint(0.0,0.0){C}
\tkzDrawArc[<-,line width=0.9mm, lime](C,B)(A)
\tkzDefPoint(-41.38289,91.03547){A}
\tkzDefPoint(-20.02985,97.97349){B}
\tkzDefPoint(0.0,0.0){C}
\tkzDrawArc[->,line width=0.9mm, yellow](C,B)(A)
\tkzDefPoint(-11.22596,99.36789){A}
\tkzDefPoint(11.22596,99.36789){B}
\tkzDefPoint(0.0,0.0){C}
\tkzDrawArc[->,line width=0.9mm, cyan](C,B)(A)
\tkzDefPoint(20.02985,97.97349){A}
\tkzDefPoint(41.38289,91.03547){B}
\tkzDefPoint(0.0,0.0){C}
\tkzDrawArc[->,line width=0.9mm, green](C,B)(A)
\tkzDefPoint(49.32499,86.98877){A}
\tkzDefPoint(67.48897,73.79186){B}
\tkzDefPoint(0.0,0.0){C}
\tkzDrawArc[->,line width=0.9mm, olive](C,B)(A)
\tkzDefPoint(73.79186,67.48897){A}
\tkzDefPoint(86.98877,49.32499){B}
\tkzDefPoint(0.0,0.0){C}
\tkzDrawArc[->,line width=0.9mm, red](C,B)(A)
\tkzDefPoint(91.03547,41.38289){A}
\tkzDefPoint(97.97349,20.02985){B}
\tkzDefPoint(0.0,0.0){C}
\tkzDrawArc[->,line width=0.9mm, lime](C,B)(A)
\tkzDefPoint(99.36789,11.22596){A}
\tkzDefPoint(99.36789,-11.22596){B}
\tkzDefPoint(0.0,0.0){C}
\tkzDrawArc[->,line width=0.9mm, blue](C,B)(A)
\tkzDefPoint(97.97349,-20.02985){A}
\tkzDefPoint(91.03547,-41.38289){B}
\tkzDefPoint(0.0,0.0){C}
\tkzDrawArc[<-,line width=0.9mm, pink](C,B)(A)
\tkzDefPoint(86.98877,-49.32499){A}
\tkzDefPoint(73.79186,-67.48897){B}
\tkzDefPoint(0.0,0.0){C}
\tkzDrawArc[<-,line width=0.9mm, red](C,B)(A)
\tkzDefPoint(67.48897,-73.79186){A}
\tkzDefPoint(49.32499,-86.98877){B}
\tkzDefPoint(0.0,0.0){C}
\tkzDrawArc[<-,line width=0.9mm, gray](C,B)(A)
\tkzDefPoint(41.38289,-91.03547){A}
\tkzDefPoint(20.02985,-97.97349){B}
\tkzDefPoint(0.0,0.0){C}
\tkzDrawArc[<-,line width=0.9mm, green](C,B)(A)
\tkzDefPoint(11.22596,-99.36789){A}
\tkzDefPoint(-11.22596,-99.36789){B}
\tkzDefPoint(0.0,0.0){C}
\tkzDrawArc[->,line width=0.9mm, pink](C,B)(A)
\node [circle,black,draw,fill=white,scale=0.94] (4) at (-15.64345,-98.76883) {4};
\node [circle,black,draw,fill=white,scale=0.94] (28) at (98.76883,-15.64345) {4};
\node [circle,black,draw,fill=white,scale=0.94] (29) at (-98.76883,-15.64345) {4};
\node [circle,black,draw,fill=white,scale=0.94] (30) at (-45.39905,89.10065) {4};
\node [circle,black,draw,fill=white,scale=0.94] (31) at (89.10065,45.39905) {4};
\node [circle,black,draw,fill=white,scale=0.94] (32) at (89.10065,-45.39905) {4};
\node [circle,black,draw,fill=white,scale=0.94] (33) at (15.64345,-98.76883) {4};
\node [circle,black,draw,fill=white,scale=0.94] (34) at (15.64345,98.76883) {4};
\node [circle,black,draw,fill=white,scale=0.94] (35) at (-89.10065,45.39905) {4};
\node [circle,black,draw,fill=white,scale=0.94] (36) at (-45.39905,-89.10065) {4};
\node [circle,black,draw,fill=white,scale=0.94] (37) at (45.39905,89.10065) {4};
\node [circle,black,draw,fill=white,scale=0.94] (38) at (45.39905,-89.10065) {4};
\node [circle,black,draw,fill=white,scale=0.94] (39) at (-70.71068,-70.71068) {4};
\node [circle,black,draw,fill=white,scale=0.94] (40) at (-98.76883,15.64345) {4};
\node [circle,black,draw,fill=white,scale=0.94] (41) at (98.76883,15.64345) {4};
\node [circle,black,draw,fill=white,scale=0.94] (42) at (-70.71068,70.71068) {4};
\node [circle,black,draw,fill=white,scale=0.94] (43) at (-15.64345,98.76883) {4};
\node [circle,black,draw,fill=white,scale=0.94] (44) at (-89.10065,-45.39905) {4};
\node [circle,black,draw,fill=white,scale=0.94] (45) at (70.71068,-70.71068) {4};
\node [circle,black,draw,fill=white,scale=0.94] (46) at (70.71068,70.71068) {4};
\end{tikzpicture}

%% file: k10_gen6.tikz
\begin{tikzpicture}[scale=0.06]
\def\vertexscale{1.05}
\def\labelscale{1.10}
\node [circle,black,draw,scale=\vertexscale] (1) at (-0.00000,0.00000) {1};
\node [circle,black,draw,scale=\vertexscale] (2) at (-53.96478,0.23324) {2};
\node [circle,black,draw,scale=\vertexscale] (3) at (26.78040,-46.85149) {3};
\node [circle,black,draw,scale=\vertexscale] (4) at (27.18438,46.61825) {4};
\node [circle,black,draw,scale=\vertexscale] (5) at (-40.47486,-33.66430) {5};
\node [circle,black,draw,scale=\vertexscale] (6) at (49.39157,-18.22011) {6};
\node [circle,black,draw,scale=\vertexscale] (7) at (-8.91671,51.88441) {7};
\node [circle,black,draw,scale=\vertexscale] (8) at (49.82306,18.38533) {8};
\node [circle,black,draw,scale=\vertexscale] (9) at (-40.83369,33.95538) {9};
\node [circle,black,draw,scale=\vertexscale] (10) at (-8.98937,-52.34070) {10};
\tkzDefPoint(-65.93458,75.18398){11}
\tkzDefPoint(-13.05262,-99.14449){12}
\tkzDefPoint(-70.71068,70.71068){13}
\tkzDefPoint(-75.18398,65.93458){14}
\tkzDefPoint(-94.69301,32.14395){15}
\tkzDefPoint(-96.59258,25.88190){16}
\tkzDefPoint(-98.07853,19.50903){17}
\tkzDefPoint(19.50903,-98.07853){18}
\tkzDefPoint(25.88190,-96.59258){19}
\tkzDefPoint(32.14395,-94.69301){20}
\tkzDefPoint(65.93458,-75.18398){21}
\tkzDefPoint(70.71068,-70.71068){22}
\tkzDefPoint(75.18398,-65.93458){23}
\tkzDefPoint(-94.69301,-32.14395){24}
\tkzDefPoint(-96.59258,-25.88190){25}
\tkzDefPoint(-98.07853,-19.50903){26}
\tkzDefPoint(-65.93458,-75.18398){27}
\tkzDefPoint(-70.71068,-70.71068){28}
\tkzDefPoint(-75.18398,-65.93458){29}
\tkzDefPoint(89.68727,-44.22887){30}
\tkzDefPoint(86.60254,-50.00000){31}
\tkzDefPoint(83.14696,-55.55702){32}
\tkzDefPoint(-89.68727,-44.22887){33}
\tkzDefPoint(-86.60254,-50.00000){34}
\tkzDefPoint(-83.14696,-55.55702){35}
\tkzDefPoint(-99.78589,6.54031){36}
\tkzDefPoint(-100.00000,-0.00000){37}
\tkzDefPoint(-99.78589,-6.54031){38}
\tkzDefPoint(-7.84591,-99.69173){39}
\tkzDefPoint(-2.61769,-99.96573){40}
\tkzDefPoint(2.61769,-99.96573){41}
\tkzDefPoint(7.84591,-99.69173){42}
\tkzDefPoint(-43.05111,-90.25853){43}
\tkzDefPoint(-47.71588,-87.88171){44}
\tkzDefPoint(-52.24986,-85.26402){45}
\tkzDefPoint(-56.64062,-82.41262){46}
\tkzDefPoint(-82.41262,56.64062){47}
\tkzDefPoint(-85.26402,52.24986){48}
\tkzDefPoint(-87.88171,47.71588){49}
\tkzDefPoint(-90.25853,43.05111){50}
\tkzDefPoint(65.93458,75.18398){51}
\tkzDefPoint(70.71068,70.71068){52}
\tkzDefPoint(75.18398,65.93458){53}
\tkzDefPoint(99.69173,7.84591){54}
\tkzDefPoint(99.96573,2.61769){55}
\tkzDefPoint(99.96573,-2.61769){56}
\tkzDefPoint(99.69173,-7.84591){57}
\tkzDefPoint(-19.50903,-98.07853){58}
\tkzDefPoint(-25.88190,-96.59258){59}
\tkzDefPoint(-32.14395,-94.69301){60}
\tkzDefPoint(-6.54031,99.78589){61}
\tkzDefPoint(-0.00000,100.00000){62}
\tkzDefPoint(6.54031,99.78589){63}
\tkzDefPoint(98.07853,-19.50903){64}
\tkzDefPoint(96.59258,-25.88190){65}
\tkzDefPoint(94.69301,-32.14395){66}
\tkzDefPoint(55.55702,-83.14696){67}
\tkzDefPoint(50.00000,-86.60254){68}
\tkzDefPoint(44.22887,-89.68727){69}
\tkzDefPoint(19.50903,98.07853){70}
\tkzDefPoint(25.88190,96.59258){71}
\tkzDefPoint(32.14395,94.69301){72}
\tkzDefPoint(82.41262,56.64062){73}
\tkzDefPoint(85.26402,52.24986){74}
\tkzDefPoint(87.88171,47.71588){75}
\tkzDefPoint(90.25853,43.05111){76}
\tkzDefPoint(-56.64062,82.41262){77}
\tkzDefPoint(-52.24986,85.26402){78}
\tkzDefPoint(-47.71588,87.88171){79}
\tkzDefPoint(-43.05111,90.25853){80}
\tkzDefPoint(94.69301,32.14395){81}
\tkzDefPoint(96.59258,25.88190){82}
\tkzDefPoint(98.07853,19.50903){83}
\tkzDefPoint(-32.14395,94.69301){84}
\tkzDefPoint(-25.88190,96.59258){85}
\tkzDefPoint(-19.50903,98.07853){86}
\tkzDefPoint(44.22887,89.68727){87}
\tkzDefPoint(50.00000,86.60254){88}
\tkzDefPoint(55.55702,83.14696){89}
\tkzDefPoint(-60.87614,79.33533){90}
\tkzDefPoint(60.87614,79.33533){91}
\tkzDefPoint(92.38795,-38.26834){92}
\tkzDefPoint(-92.38795,-38.26834){93}
\tkzDefPoint(60.87614,-79.33533){94}
\tkzDefPoint(-38.26834,92.38795){95}
\tkzDefPoint(13.05262,-99.14449){96}
\tkzDefPoint(92.38795,38.26834){97}
\tkzDefPoint(-38.26834,-92.38795){98}
\tkzDefPoint(-99.14449,13.05262){99}
\tkzDefPoint(-13.05262,99.14449){100}
\tkzDefPoint(79.33533,-60.87614){101}
\tkzDefPoint(38.26834,92.38795){102}
\tkzDefPoint(-60.87614,-79.33533){103}
\tkzDefPoint(79.33533,60.87614){104}
\tkzDefPoint(-79.33533,60.87614){105}
\tkzDefPoint(99.14449,13.05262){106}
\tkzDefPoint(38.26834,-92.38795){107}
\tkzDefPoint(-79.33533,-60.87614){108}
\tkzDefPoint(13.05262,99.14449){109}
\tkzDefPoint(-99.14449,-13.05262){110}
\tkzDefPoint(99.14449,-13.05262){111}
\tkzDefPoint(-92.38795,38.26834){112}
\draw [black] (1) to (2);
\draw [black] (1) to (9);
\draw [black] (1) to (7);
\draw [black] (1) to (4);
\draw [black] (1) to (8);
\draw [black] (1) to (6);
\draw [black] (1) to (3);
\draw [black] (1) to (10);
\draw [black] (1) to (5);
\draw [black] (2) to (5);
\draw [black] (2) to (24);
\node [draw=none,fill=none,scale=\labelscale] () at (-99.42766,-33.75114) {4};
\draw [black] (2) to (25);
\node [draw=none,fill=none,scale=\labelscale] () at (-101.42221,-27.17600) {6};
\draw [black] (2) to (37);
\node [draw=none,fill=none,scale=\labelscale] () at (-105.00000,-0.00000) {7};
\draw [black] (2) to (36);
\node [draw=none,fill=none,scale=\labelscale] () at (-104.77519,6.86733) {3};
\draw [black] (2) to (17);
\node [draw=none,fill=none,scale=\labelscale] () at (-102.98245,20.48448) {8};
\draw [black] (2) to (16);
\node [draw=none,fill=none,scale=\labelscale] () at (-101.42221,27.17600) {10};
\draw [black] (2) to (9);
\draw [black] (3) to (6);
\draw [black] (3) to (23);
\node [draw=none,fill=none,scale=\labelscale] () at (78.94318,-69.23131) {2};
\draw [black] (3) to (22);
\node [draw=none,fill=none,scale=\labelscale] () at (74.24621,-74.24621) {7};
\draw [black] (3) to (68);
\node [draw=none,fill=none,scale=\labelscale] () at (52.50000,-90.93267) {5};
\draw [black] (3) to (69);
\node [draw=none,fill=none,scale=\labelscale] () at (46.44031,-94.17164) {4};
\draw [black] (3) to (20);
\node [draw=none,fill=none,scale=\labelscale] () at (33.75114,-99.42766) {9};
\draw [black] (3) to (19);
\node [draw=none,fill=none,scale=\labelscale] () at (27.17600,-101.42221) {8};
\draw [black] (3) to (10);
\draw [black] (4) to (7);
\draw [black] (4) to (70);
\node [draw=none,fill=none,scale=\labelscale] () at (20.48448,102.98245) {3};
\draw [black] (4) to (71);
\node [draw=none,fill=none,scale=\labelscale] () at (27.17600,101.42221) {5};
\draw [black] (4) to (88);
\node [draw=none,fill=none,scale=\labelscale] () at (52.50000,90.93267) {6};
\draw [black] (4) to (89);
\node [draw=none,fill=none,scale=\labelscale] () at (58.33487,87.30431) {2};
\draw [black] (4) to (51);
\node [draw=none,fill=none,scale=\labelscale] () at (69.23131,78.94318) {10};
\draw [black] (4) to (52);
\node [draw=none,fill=none,scale=\labelscale] () at (74.24621,74.24621) {9};
\draw [black] (4) to (8);
\draw [black] (5) to (10);
\draw [black] (5) to (45);
\node [draw=none,fill=none,scale=\labelscale] () at (-54.86235,-89.52722) {8};
\draw [black] (5) to (46);
\node [draw=none,fill=none,scale=\labelscale] () at (-59.47265,-86.53325) {9};
\draw [black] (5) to (27);
\node [draw=none,fill=none,scale=\labelscale] () at (-69.23131,-78.94318) {6};
\draw [black] (5) to (28);
\node [draw=none,fill=none,scale=\labelscale] () at (-74.24621,-74.24621) {4};
\draw [black] (5) to (34);
\node [draw=none,fill=none,scale=\labelscale] () at (-90.93267,-52.50000) {3};
\draw [black] (5) to (33);
\node [draw=none,fill=none,scale=\labelscale] () at (-94.17164,-46.44031) {7};
\draw [black] (6) to (8);
\draw [black] (6) to (56);
\node [draw=none,fill=none,scale=\labelscale] () at (104.96402,-2.74858) {9};
\draw [black] (6) to (57);
\node [draw=none,fill=none,scale=\labelscale] () at (104.67632,-8.23821) {10};
\draw [black] (6) to (64);
\node [draw=none,fill=none,scale=\labelscale] () at (102.98245,-20.48448) {7};
\draw [black] (6) to (65);
\node [draw=none,fill=none,scale=\labelscale] () at (101.42221,-27.17600) {2};
\draw [black] (6) to (31);
\node [draw=none,fill=none,scale=\labelscale] () at (90.93267,-52.50000) {4};
\draw [black] (6) to (32);
\node [draw=none,fill=none,scale=\labelscale] () at (87.30431,-58.33487) {5};
\draw [black] (7) to (9);
\draw [black] (7) to (79);
\node [draw=none,fill=none,scale=\labelscale] () at (-50.10167,92.27580) {10};
\draw [black] (7) to (80);
\node [draw=none,fill=none,scale=\labelscale] () at (-45.20367,94.77145) {8};
\draw [black] (7) to (84);
\node [draw=none,fill=none,scale=\labelscale] () at (-33.75114,99.42766) {5};
\draw [black] (7) to (85);
\node [draw=none,fill=none,scale=\labelscale] () at (-27.17600,101.42221) {3};
\draw [black] (7) to (62);
\node [draw=none,fill=none,scale=\labelscale] () at (-0.00000,105.00000) {2};
\draw [black] (7) to (63);
\node [draw=none,fill=none,scale=\labelscale] () at (6.86733,104.77519) {6};
\draw [black] (8) to (74);
\node [draw=none,fill=none,scale=\labelscale] () at (89.52722,54.86235) {5};
\draw [black] (8) to (75);
\node [draw=none,fill=none,scale=\labelscale] () at (92.27580,50.10167) {10};
\draw [black] (8) to (76);
\node [draw=none,fill=none,scale=\labelscale] () at (94.77145,45.20367) {2};
\draw [black] (8) to (81);
\node [draw=none,fill=none,scale=\labelscale] () at (99.42766,33.75114) {7};
\draw [black] (8) to (82);
\node [draw=none,fill=none,scale=\labelscale] () at (101.42221,27.17600) {3};
\draw [black] (8) to (55);
\node [draw=none,fill=none,scale=\labelscale] () at (104.96402,2.74858) {9};
\draw [black] (9) to (49);
\node [draw=none,fill=none,scale=\labelscale] () at (-92.27580,50.10167) {6};
\draw [black] (9) to (48);
\node [draw=none,fill=none,scale=\labelscale] () at (-89.52722,54.86235) {8};
\draw [black] (9) to (47);
\node [draw=none,fill=none,scale=\labelscale] () at (-86.53325,59.47265) {3};
\draw [black] (9) to (14);
\node [draw=none,fill=none,scale=\labelscale] () at (-78.94318,69.23131) {5};
\draw [black] (9) to (13);
\node [draw=none,fill=none,scale=\labelscale] () at (-74.24621,74.24621) {4};
\draw [black] (9) to (78);
\node [draw=none,fill=none,scale=\labelscale] () at (-54.86235,89.52722) {10};
\draw [black] (10) to (41);
\node [draw=none,fill=none,scale=\labelscale] () at (2.74858,-104.96402) {7};
\draw [black] (10) to (40);
\node [draw=none,fill=none,scale=\labelscale] () at (-2.74858,-104.96402) {9};
\draw [black] (10) to (39);
\node [draw=none,fill=none,scale=\labelscale] () at (-8.23821,-104.67632) {4};
\draw [black] (10) to (58);
\node [draw=none,fill=none,scale=\labelscale] () at (-20.48448,-102.98245) {6};
\draw [black] (10) to (59);
\node [draw=none,fill=none,scale=\labelscale] () at (-27.17600,-101.42221) {2};
\draw [black] (10) to (44);
\node [draw=none,fill=none,scale=\labelscale] () at (-50.10167,-92.27580) {8};
\tkzDefPoint(-65.93458,75.18398){A}
\tkzDefPoint(-57.61035,74.07040){B}
\tkzDefPoint(-56.64062,82.41262){C}
\tkzCircumCenter(A,B,C)\tkzGetPoint{D}
\tkzDrawArc[black](D,A)(C)
\tkzDefPoint(-94.69301,32.14395){A}
\tkzDefPoint(-86.92722,35.34168){B}
\tkzDefPoint(-90.25853,43.05111){C}
\tkzCircumCenter(A,B,C)\tkzGetPoint{D}
\tkzDrawArc[black](D,A)(C)
\tkzDefPoint(19.50903,-98.07853){A}
\tkzDefPoint(12.85682,-92.95202){B}
\tkzDefPoint(7.84591,-99.69173){C}
\tkzCircumCenter(A,B,C)\tkzGetPoint{D}
\tkzDrawArc[black](D,A)(C)
\tkzDefPoint(65.93458,-75.18398){A}
\tkzDefPoint(57.10105,-74.41554){B}
\tkzDefPoint(55.55702,-83.14696){C}
\tkzCircumCenter(A,B,C)\tkzGetPoint{D}
\tkzDrawArc[black](D,A)(C)
\tkzDefPoint(-98.07853,-19.50903){A}
\tkzDefPoint(-92.99628,-12.24319){B}
\tkzDefPoint(-99.78589,-6.54031){C}
\tkzCircumCenter(A,B,C)\tkzGetPoint{D}
\tkzDrawArc[black](D,A)(C)
\tkzDefPoint(-75.18398,-65.93458){A}
\tkzDefPoint(-74.41554,-57.10105){B}
\tkzDefPoint(-83.14696,-55.55702){C}
\tkzCircumCenter(A,B,C)\tkzGetPoint{D}
\tkzDrawArc[black](D,A)(C)
\tkzDefPoint(89.68727,-44.22887){A}
\tkzDefPoint(86.65873,-35.89522){B}
\tkzDefPoint(94.69301,-32.14395){C}
\tkzCircumCenter(A,B,C)\tkzGetPoint{D}
\tkzDrawArc[black](D,C)(A)
\tkzDefPoint(-43.05111,-90.25853){A}
\tkzDefPoint(-35.34168,-86.92722){B}
\tkzDefPoint(-32.14395,-94.69301){C}
\tkzCircumCenter(A,B,C)\tkzGetPoint{D}
\tkzDrawArc[black](D,C)(A)
\tkzDefPoint(75.18398,65.93458){A}
\tkzDefPoint(74.07040,57.61035){B}
\tkzDefPoint(82.41262,56.64062){C}
\tkzCircumCenter(A,B,C)\tkzGetPoint{D}
\tkzDrawArc[black](D,A)(C)
\tkzDefPoint(99.69173,7.84591){A}
\tkzDefPoint(92.95202,12.85682){B}
\tkzDefPoint(98.07853,19.50903){C}
\tkzCircumCenter(A,B,C)\tkzGetPoint{D}
\tkzDrawArc[black](D,C)(A)
\tkzDefPoint(-6.54031,99.78589){A}
\tkzDefPoint(-12.24319,92.99628){B}
\tkzDefPoint(-19.50903,98.07853){C}
\tkzCircumCenter(A,B,C)\tkzGetPoint{D}
\tkzDrawArc[black](D,C)(A)
\tkzDefPoint(32.14395,94.69301){A}
\tkzDefPoint(35.89522,86.65873){B}
\tkzDefPoint(44.22887,89.68727){C}
\tkzCircumCenter(A,B,C)\tkzGetPoint{D}
\tkzDrawArc[black](D,A)(C)
\tkzDefPoint(-60.04212,79.96839){A}
\tkzDefPoint(-39.23360,91.98220){B}
\tkzDefPoint(0.0,0.0){C}
\tkzDrawArc[<-,line width=0.9mm, gray](C,B)(A)
\node [draw=none,fill=none,scale=1.50] () at (-57.50000,99.59292) {J};
\tkzDefPoint(-37.29889,92.78358){A}
\tkzDefPoint(-14.09000,99.00238){B}
\tkzDefPoint(0.0,0.0){C}
\tkzDrawArc[<-,line width=0.9mm, gray](C,B)(A)
\node [draw=none,fill=none,scale=1.50] () at (-29.76419,111.08147) {D};
\tkzDefPoint(-12.01380,99.27572){A}
\tkzDefPoint(12.01380,99.27572){B}
\tkzDefPoint(0.0,0.0){C}
\tkzDrawArc[<-,line width=0.9mm, gray](C,B)(A)
\node [draw=none,fill=none,scale=1.50] () at (-0.00000,115.00000) {I};
\tkzDefPoint(14.09000,99.00238){A}
\tkzDefPoint(37.29889,92.78358){B}
\tkzDefPoint(0.0,0.0){C}
\tkzDrawArc[<-,line width=0.9mm, gray](C,B)(A)
\node [draw=none,fill=none,scale=1.50] () at (29.76419,111.08147) {F};
\tkzDefPoint(39.23360,91.98220){A}
\tkzDefPoint(60.04212,79.96839){B}
\tkzDefPoint(0.0,0.0){C}
\tkzDrawArc[<-,line width=0.9mm, gray](C,B)(A)
\node [draw=none,fill=none,scale=1.50] () at (57.50000,99.59292) {G};
\tkzDefPoint(61.70349,78.69358){A}
\tkzDefPoint(78.69358,61.70349){B}
\tkzDefPoint(0.0,0.0){C}
\tkzDrawArc[<-,line width=0.9mm, gray](C,B)(A)
\node [draw=none,fill=none,scale=1.50] () at (81.31728,81.31728) {A};
\tkzDefPoint(79.96839,60.04212){A}
\tkzDefPoint(91.98220,39.23360){B}
\tkzDefPoint(0.0,0.0){C}
\tkzDrawArc[<-,line width=0.9mm, gray](C,B)(A)
\node [draw=none,fill=none,scale=1.50] () at (99.59292,57.50000) {K};
\tkzDefPoint(92.78358,37.29889){A}
\tkzDefPoint(99.00238,14.09000){B}
\tkzDefPoint(0.0,0.0){C}
\tkzDrawArc[<-,line width=0.9mm, gray](C,B)(A)
\node [draw=none,fill=none,scale=1.50] () at (111.08147,29.76419) {C};
\tkzDefPoint(99.27572,12.01380){A}
\tkzDefPoint(99.27572,-12.01380){B}
\tkzDefPoint(0.0,0.0){C}
\tkzDrawArc[<-,line width=0.9mm, gray](C,B)(A)
\node [draw=none,fill=none,scale=1.50] () at (115.00000,0.00000) {L};
\tkzDefPoint(99.00238,-14.09000){A}
\tkzDefPoint(92.78358,-37.29889){B}
\tkzDefPoint(0.0,0.0){C}
\tkzDrawArc[<-,line width=0.9mm, gray](C,B)(A)
\node [draw=none,fill=none,scale=1.50] () at (111.08147,-29.76419) {E};
\tkzDefPoint(91.98220,-39.23360){A}
\tkzDefPoint(79.96839,-60.04212){B}
\tkzDefPoint(0.0,0.0){C}
\tkzDrawArc[->,line width=0.9mm, gray](C,B)(A)
\node [draw=none,fill=none,scale=1.50] () at (99.59292,-57.50000) {G};
\tkzDefPoint(78.69358,-61.70349){A}
\tkzDefPoint(61.70349,-78.69358){B}
\tkzDefPoint(0.0,0.0){C}
\tkzDrawArc[->,line width=0.9mm, gray](C,B)(A)
\node [draw=none,fill=none,scale=1.50] () at (81.31728,-81.31728) {D};
\tkzDefPoint(60.04212,-79.96839){A}
\tkzDefPoint(39.23360,-91.98220){B}
\tkzDefPoint(0.0,0.0){C}
\tkzDrawArc[<-,line width=0.9mm, gray](C,B)(A)
\node [draw=none,fill=none,scale=1.50] () at (57.50000,-99.59292) {H};
\tkzDefPoint(37.29889,-92.78358){A}
\tkzDefPoint(14.09000,-99.00238){B}
\tkzDefPoint(0.0,0.0){C}
\tkzDrawArc[->,line width=0.9mm, gray](C,B)(A)
\node [draw=none,fill=none,scale=1.50] () at (29.76419,-111.08147) {C};
\tkzDefPoint(12.01380,-99.27572){A}
\tkzDefPoint(-12.01380,-99.27572){B}
\tkzDefPoint(0.0,0.0){C}
\tkzDrawArc[->,line width=0.9mm, gray](C,B)(A)
\node [draw=none,fill=none,scale=1.50] () at (0.00000,-115.00000) {J};
\tkzDefPoint(-14.09000,-99.00238){A}
\tkzDefPoint(-37.29889,-92.78358){B}
\tkzDefPoint(0.0,0.0){C}
\tkzDrawArc[<-,line width=0.9mm, gray](C,B)(A)
\node [draw=none,fill=none,scale=1.50] () at (-29.76419,-111.08147) {B};
\tkzDefPoint(-39.23360,-91.98220){A}
\tkzDefPoint(-60.04212,-79.96839){B}
\tkzDefPoint(0.0,0.0){C}
\tkzDrawArc[->,line width=0.9mm, gray](C,B)(A)
\node [draw=none,fill=none,scale=1.50] () at (-57.50000,-99.59292) {K};
\tkzDefPoint(-61.70349,-78.69358){A}
\tkzDefPoint(-78.69358,-61.70349){B}
\tkzDefPoint(0.0,0.0){C}
\tkzDrawArc[->,line width=0.9mm, gray](C,B)(A)
\node [draw=none,fill=none,scale=1.50] () at (-81.31728,-81.31728) {F};
\tkzDefPoint(-79.96839,-60.04212){A}
\tkzDefPoint(-91.98220,-39.23360){B}
\tkzDefPoint(0.0,0.0){C}
\tkzDrawArc[->,line width=0.9mm, gray](C,B)(A)
\node [draw=none,fill=none,scale=1.50] () at (-99.59292,-57.50000) {H};
\tkzDefPoint(-92.78358,-37.29889){A}
\tkzDefPoint(-99.00238,-14.09000){B}
\tkzDefPoint(0.0,0.0){C}
\tkzDrawArc[->,line width=0.9mm, gray](C,B)(A)
\node [draw=none,fill=none,scale=1.50] () at (-111.08147,-29.76419) {E};
\tkzDefPoint(-99.27572,-12.01380){A}
\tkzDefPoint(-99.27572,12.01380){B}
\tkzDefPoint(0.0,0.0){C}
\tkzDrawArc[->,line width=0.9mm, gray](C,B)(A)
\node [draw=none,fill=none,scale=1.50] () at (-115.00000,-0.00000) {I};
\tkzDefPoint(-99.00238,14.09000){A}
\tkzDefPoint(-92.78358,37.29889){B}
\tkzDefPoint(0.0,0.0){C}
\tkzDrawArc[->,line width=0.9mm, gray](C,B)(A)
\node [draw=none,fill=none,scale=1.50] () at (-111.08147,29.76419) {B};
\tkzDefPoint(-91.98220,39.23360){A}
\tkzDefPoint(-79.96839,60.04212){B}
\tkzDefPoint(0.0,0.0){C}
\tkzDrawArc[->,line width=0.9mm, gray](C,B)(A)
\node [draw=none,fill=none,scale=1.50] () at (-99.59292,57.50000) {L};
\tkzDefPoint(-78.69358,61.70349){A}
\tkzDefPoint(-61.70349,78.69358){B}
\tkzDefPoint(0.0,0.0){C}
\tkzDrawArc[->,line width=0.9mm, gray](C,B)(A)
\node [draw=none,fill=none,scale=1.50] () at (-81.31728,81.31728) {A};
\node [black,circle,draw,fill=white,scale=0.26,line width=1mm] (12) at (-13.05262,-99.14449) {};
\node [black,circle,draw,fill=white,scale=0.26,line width=1mm] (90) at (-60.87614,79.33533) {};
\node [black,circle,draw,fill=white,scale=0.26,line width=1mm] (91) at (60.87614,79.33533) {};
\node [black,circle,draw,fill=white,scale=0.26,line width=1mm] (92) at (92.38795,-38.26834) {};
\node [black,circle,draw,fill=white,scale=0.26,line width=1mm] (93) at (-92.38795,-38.26834) {};
\node [black,circle,draw,fill=white,scale=0.26,line width=1mm] (94) at (60.87614,-79.33533) {};
\node [black,circle,draw,fill=white,scale=0.26,line width=1mm] (95) at (-38.26834,92.38795) {};
\node [black,circle,draw,fill=white,scale=0.26,line width=1mm] (96) at (13.05262,-99.14449) {};
\node [black,circle,draw,fill=white,scale=0.26,line width=1mm] (97) at (92.38795,38.26834) {};
\node [black,circle,draw,fill=white,scale=0.26,line width=1mm] (98) at (-38.26834,-92.38795) {};
\node [black,circle,draw,fill=white,scale=0.26,line width=1mm] (99) at (-99.14449,13.05262) {};
\node [black,circle,draw,fill=white,scale=0.26,line width=1mm] (100) at (-13.05262,99.14449) {};
\node [black,circle,draw,fill=white,scale=0.26,line width=1mm] (101) at (79.33533,-60.87614) {};
\node [black,circle,draw,fill=white,scale=0.26,line width=1mm] (102) at (38.26834,92.38795) {};
\node [black,circle,draw,fill=white,scale=0.26,line width=1mm] (103) at (-60.87614,-79.33533) {};
\node [black,circle,draw,fill=white,scale=0.26,line width=1mm] (104) at (79.33533,60.87614) {};
\node [black,circle,draw,fill=white,scale=0.26,line width=1mm] (105) at (-79.33533,60.87614) {};
\node [black,circle,draw,fill=white,scale=0.26,line width=1mm] (106) at (99.14449,13.05262) {};
\node [black,circle,draw,fill=white,scale=0.26,line width=1mm] (107) at (38.26834,-92.38795) {};
\node [black,circle,draw,fill=white,scale=0.26,line width=1mm] (108) at (-79.33533,-60.87614) {};
\node [black,circle,draw,fill=white,scale=0.26,line width=1mm] (109) at (13.05262,99.14449) {};
\node [black,circle,draw,fill=white,scale=0.26,line width=1mm] (110) at (-99.14449,-13.05262) {};
\node [black,circle,draw,fill=white,scale=0.26,line width=1mm] (111) at (99.14449,-13.05262) {};
\node [black,circle,draw,fill=white,scale=0.26,line width=1mm] (112) at (-92.38795,38.26834) {};
\end{tikzpicture}

%% file: k10_gen6_dual.tikz
\begin{tikzpicture}[scale=0.06]
\def\vertexscale{0.91}
\def\labelscale{1.12}
\node [circle,black,draw,scale=\vertexscale] (1) at (35.89899,35.12576) {1};
\node [circle,black,draw,scale=\vertexscale] (2) at (21.66357,41.77880) {2};
\node [circle,black,draw,scale=\vertexscale] (3) at (-76.32431,47.06260) {3};
\node [circle,black,draw,scale=\vertexscale] (4) at (44.80873,21.70699) {4};
\node [circle,black,draw,scale=\vertexscale] (5) at (-2.05113,40.22011) {5};
\tkzDefPoint(-13.05262,-99.14449){6}
\node [circle,black,draw,scale=\vertexscale] (7) at (4.78419,78.85884) {7};
\node [circle,black,draw,scale=\vertexscale] (8) at (-60.96020,55.77957) {8};
\node [circle,black,draw,scale=\vertexscale] (9) at (-24.76351,-72.29385) {9};
\node [circle,black,draw,scale=\vertexscale] (10) at (40.37167,-1.27777) {10};
\node [circle,black,draw,scale=\vertexscale] (11) at (-26.76735,28.79413) {11};
\node [circle,black,draw,scale=\vertexscale] (12) at (10.51313,63.53650) {12};
\node [circle,black,draw,scale=\vertexscale] (13) at (70.65225,-35.72087) {13};
\node [circle,black,draw,scale=\vertexscale] (14) at (-37.52365,49.58516) {14};
\node [circle,black,draw,scale=\vertexscale] (15) at (84.21764,-26.92906) {15};
\node [circle,black,draw,scale=\vertexscale] (16) at (-43.50998,6.62891) {16};
\node [circle,black,draw,scale=\vertexscale] (17) at (43.33333,-47.99309) {17};
\node [circle,black,draw,scale=\vertexscale] (18) at (-25.07130,-40.33090) {18};
\node [circle,black,draw,scale=\vertexscale] (19) at (-64.59277,33.32661) {19};
\node [circle,black,draw,scale=\vertexscale] (20) at (59.00815,-23.17604) {20};
\node [circle,black,draw,scale=\vertexscale] (21) at (24.46388,-79.05800) {21};
\node [circle,black,draw,scale=\vertexscale] (22) at (-2.44372,-54.06811) {22};
\node [circle,black,draw,scale=\vertexscale] (23) at (20.90504,-60.39265) {23};
\node [circle,black,draw,scale=\vertexscale] (24) at (-36.20533,-10.60971) {24};
\node [circle,black,draw,scale=\vertexscale] (25) at (-1.53593,-15.68527) {25};
\tkzDefPoint(0.00000,-100.00000){26}
\tkzDefPoint(70.71068,70.71068){27}
\tkzDefPoint(-0.00000,100.00000){28}
\tkzDefPoint(50.00000,86.60254){29}
\tkzDefPoint(-25.88190,96.59258){30}
\tkzDefPoint(86.60254,50.00000){31}
\tkzDefPoint(-50.00000,86.60254){32}
\tkzDefPoint(-50.00000,-86.60254){33}
\tkzDefPoint(-96.59258,25.88190){34}
\tkzDefPoint(-86.60254,-50.00000){35}
\tkzDefPoint(-86.60254,50.00000){36}
\tkzDefPoint(-25.88190,-96.59258){37}
\tkzDefPoint(-96.59258,-25.88190){38}
\tkzDefPoint(86.60254,-50.00000){39}
\tkzDefPoint(92.38795,-38.26834){40}
\tkzDefPoint(13.05262,99.14449){41}
\tkzDefPoint(-92.38795,38.26834){42}
\tkzDefPoint(60.87614,79.33533){43}
\tkzDefPoint(-79.33533,-60.87614){44}
\tkzDefPoint(79.33533,-60.87614){45}
\tkzDefPoint(13.05262,-99.14449){46}
\tkzDefPoint(92.38795,38.26834){47}
\tkzDefPoint(-38.26834,-92.38795){48}
\tkzDefPoint(-38.26834,92.38795){49}
\tkzDefPoint(-92.38795,-38.26834){50}
\tkzDefPoint(38.26834,92.38795){51}
\tkzDefPoint(99.14449,-13.05262){52}
\tkzDefPoint(-99.14449,-13.05262){53}
\tkzDefPoint(-60.87614,79.33533){54}
\tkzDefPoint(60.87614,-79.33533){55}
\tkzDefPoint(-60.87614,-79.33533){56}
\tkzDefPoint(79.33533,60.87614){57}
\tkzDefPoint(-99.14449,13.05262){58}
\tkzDefPoint(99.14449,13.05262){59}
\tkzDefPoint(38.26834,-92.38795){60}
\tkzDefPoint(-79.33533,60.87614){61}
\tkzDefPoint(-13.05262,99.14449){62}
\draw [black] (1) to (2);
\draw [black] (1) to (27);
\node [draw=none,fill=none,scale=\labelscale] () at (74.24621,74.24621) {3};
\draw [black] (1) to (4);
\draw [black] (2) to (5);
\draw [black] (2) to (43);
\draw [black] (3) to (34);
\node [draw=none,fill=none,scale=\labelscale] () at (-101.42221,27.17600) {1};
\draw [black] (3) to (36);
\node [draw=none,fill=none,scale=\labelscale] () at (-90.93267,52.50000) {7};
\draw [black] (3) to (61);
\draw [black] (3) to (8);
\draw [black] (4) to (31);
\node [draw=none,fill=none,scale=\labelscale] () at (90.93267,52.50000) {9};
\draw [black] (4) to (10);
\draw [black] (5) to (11);
\draw [black] (5) to (12);
\draw [black] (6) to (22);
\draw [black] (6) to (23);
\draw [black] (7) to (28);
\node [draw=none,fill=none,scale=\labelscale] () at (-0.00000,105.00000) {3};
\draw [black] (7) to (51);
\draw [black] (7) to (12);
\draw [black] (8) to (32);
\node [draw=none,fill=none,scale=\labelscale] () at (-52.50000,90.93267) {24};
\draw [black] (8) to (19);
\draw [black] (9) to (33);
\node [draw=none,fill=none,scale=\labelscale] () at (-52.50000,-90.93267) {4};
\draw [black] (9) to (22);
\draw [black] (9) to (37);
\node [draw=none,fill=none,scale=\labelscale] () at (-27.17600,-101.42221) {14};
\draw [black] (10) to (59);
\draw [black] (10) to (25);
\draw [black] (11) to (16);
\draw [black] (11) to (14);
\draw [black] (12) to (29);
\node [draw=none,fill=none,scale=\labelscale] () at (52.50000,90.93267) {18};
\draw [black] (13) to (45);
\draw [black] (13) to (20);
\draw [black] (13) to (15);
\draw [black] (14) to (49);
\draw [black] (14) to (30);
\node [draw=none,fill=none,scale=\labelscale] () at (-27.17600,101.42221) {9};
\draw [black] (14) to (19);
\draw [black] (15) to (52);
\draw [black] (15) to (39);
\node [draw=none,fill=none,scale=\labelscale] () at (90.93267,-52.50000) {21};
\draw [black] (16) to (53);
\draw [black] (16) to (24);
\draw [black] (17) to (55);
\draw [black] (17) to (23);
\draw [black] (17) to (20);
\draw [black] (18) to (56);
\draw [black] (18) to (35);
\node [draw=none,fill=none,scale=\labelscale] () at (-90.93267,-52.50000) {12};
\draw [black] (18) to (25);
\draw [black] (18) to (22);
\draw [black] (19) to (58);
\draw [black] (20) to (59);
\draw [black] (21) to (60);
\draw [black] (21) to (26);
\node [draw=none,fill=none,scale=\labelscale] () at (0.00000,-105.00000) {15};
\draw [black] (21) to (23);
\draw [black] (24) to (38);
\node [draw=none,fill=none,scale=\labelscale] () at (-101.42221,-27.17600) {8};
\draw [black] (24) to (25);
\tkzDefPoint(-16.57635,-98.61655){A}
\tkzDefPoint(-34.95274,-93.69261){B}
\tkzDefPoint(0.0,0.0){C}
\tkzDrawArc[<-,line width=0.9mm, gray](C,B)(A)
\node [draw=none,fill=none,scale=1.50] () at (-29.76419,-111.08147) {J};
\tkzDefPoint(-41.53536,-90.96600){A}
\tkzDefPoint(-58.01119,-81.45368){B}
\tkzDefPoint(0.0,0.0){C}
\tkzDrawArc[<-,line width=0.9mm, gray](C,B)(A)
\node [draw=none,fill=none,scale=1.50] () at (-57.50000,-99.59292) {K};
\tkzDefPoint(-63.66381,-77.11627){A}
\tkzDefPoint(-77.11627,-63.66381){B}
\tkzDefPoint(0.0,0.0){C}
\tkzDrawArc[->,line width=0.9mm, gray](C,B)(A)
\node [draw=none,fill=none,scale=1.50] () at (-81.31728,-81.31728) {B};
\tkzDefPoint(-81.45368,-58.01119){A}
\tkzDefPoint(-90.96600,-41.53536){B}
\tkzDefPoint(0.0,0.0){C}
\tkzDrawArc[<-,line width=0.9mm, gray](C,B)(A)
\node [draw=none,fill=none,scale=1.50] () at (-99.59292,-57.50000) {I};
\tkzDefPoint(-93.69261,-34.95274){A}
\tkzDefPoint(-98.61655,-16.57635){B}
\tkzDefPoint(0.0,0.0){C}
\tkzDrawArc[<-,line width=0.9mm, gray](C,B)(A)
\node [draw=none,fill=none,scale=1.50] () at (-111.08147,-29.76419) {L};
\tkzDefPoint(-99.54655,-9.51232){A}
\tkzDefPoint(-99.54655,9.51232){B}
\tkzDefPoint(0.0,0.0){C}
\tkzDrawArc[<-,line width=0.9mm, gray](C,B)(A)
\node [draw=none,fill=none,scale=1.50] () at (-115.00000,-0.00000) {D};
\tkzDefPoint(-98.61655,16.57635){A}
\tkzDefPoint(-93.69261,34.95274){B}
\tkzDefPoint(0.0,0.0){C}
\tkzDrawArc[<-,line width=0.9mm, gray](C,B)(A)
\node [draw=none,fill=none,scale=1.50] () at (-111.08147,29.76419) {G};
\tkzDefPoint(-90.96600,41.53536){A}
\tkzDefPoint(-81.45368,58.01119){B}
\tkzDefPoint(0.0,0.0){C}
\tkzDrawArc[<-,line width=0.9mm, gray](C,B)(A)
\node [draw=none,fill=none,scale=1.50] () at (-99.59292,57.50000) {H};
\tkzDefPoint(-77.11627,63.66381){A}
\tkzDefPoint(-63.66381,77.11627){B}
\tkzDefPoint(0.0,0.0){C}
\tkzDrawArc[->,line width=0.9mm, gray](C,B)(A)
\node [draw=none,fill=none,scale=1.50] () at (-81.31728,81.31728) {E};
\tkzDefPoint(-58.01119,81.45368){A}
\tkzDefPoint(-41.53536,90.96600){B}
\tkzDefPoint(0.0,0.0){C}
\tkzDrawArc[->,line width=0.9mm, gray](C,B)(A)
\node [draw=none,fill=none,scale=1.50] () at (-57.50000,99.59292) {L};
\tkzDefPoint(-34.95274,93.69261){A}
\tkzDefPoint(-16.57635,98.61655){B}
\tkzDefPoint(0.0,0.0){C}
\tkzDrawArc[->,line width=0.9mm, gray](C,B)(A)
\node [draw=none,fill=none,scale=1.50] () at (-29.76419,111.08147) {J};
\tkzDefPoint(-9.51232,99.54655){A}
\tkzDefPoint(9.51232,99.54655){B}
\tkzDefPoint(0.0,0.0){C}
\tkzDrawArc[->,line width=0.9mm, gray](C,B)(A)
\node [draw=none,fill=none,scale=1.50] () at (-0.00000,115.00000) {H};
\tkzDefPoint(16.57635,98.61655){A}
\tkzDefPoint(34.95274,93.69261){B}
\tkzDefPoint(0.0,0.0){C}
\tkzDrawArc[<-,line width=0.9mm, gray](C,B)(A)
\node [draw=none,fill=none,scale=1.50] () at (29.76419,111.08147) {A};
\tkzDefPoint(41.53536,90.96600){A}
\tkzDefPoint(58.01119,81.45368){B}
\tkzDefPoint(0.0,0.0){C}
\tkzDrawArc[->,line width=0.9mm, gray](C,B)(A)
\node [draw=none,fill=none,scale=1.50] () at (57.50000,99.59292) {I};
\tkzDefPoint(63.66381,77.11627){A}
\tkzDefPoint(77.11627,63.66381){B}
\tkzDefPoint(0.0,0.0){C}
\tkzDrawArc[->,line width=0.9mm, gray](C,B)(A)
\node [draw=none,fill=none,scale=1.50] () at (81.31728,81.31728) {G};
\tkzDefPoint(81.45368,58.01119){A}
\tkzDefPoint(90.96600,41.53536){B}
\tkzDefPoint(0.0,0.0){C}
\tkzDrawArc[->,line width=0.9mm, gray](C,B)(A)
\node [draw=none,fill=none,scale=1.50] () at (99.59292,57.50000) {K};
\tkzDefPoint(93.69261,34.95274){A}
\tkzDefPoint(98.61655,16.57635){B}
\tkzDefPoint(0.0,0.0){C}
\tkzDrawArc[->,line width=0.9mm, gray](C,B)(A)
\node [draw=none,fill=none,scale=1.50] () at (111.08147,29.76419) {C};
\tkzDefPoint(99.54655,9.51232){A}
\tkzDefPoint(99.54655,-9.51232){B}
\tkzDefPoint(0.0,0.0){C}
\tkzDrawArc[->,line width=0.9mm, gray](C,B)(A)
\node [draw=none,fill=none,scale=1.50] () at (115.00000,0.00000) {D};
\tkzDefPoint(98.61655,-16.57635){A}
\tkzDefPoint(93.69261,-34.95274){B}
\tkzDefPoint(0.0,0.0){C}
\tkzDrawArc[->,line width=0.9mm, gray](C,B)(A)
\node [draw=none,fill=none,scale=1.50] () at (111.08147,-29.76419) {A};
\tkzDefPoint(90.96600,-41.53536){A}
\tkzDefPoint(81.45368,-58.01119){B}
\tkzDefPoint(0.0,0.0){C}
\tkzDrawArc[<-,line width=0.9mm, gray](C,B)(A)
\node [draw=none,fill=none,scale=1.50] () at (99.59292,-57.50000) {F};
\tkzDefPoint(77.11627,-63.66381){A}
\tkzDefPoint(63.66381,-77.11627){B}
\tkzDefPoint(0.0,0.0){C}
\tkzDrawArc[<-,line width=0.9mm, gray](C,B)(A)
\node [draw=none,fill=none,scale=1.50] () at (81.31728,-81.31728) {B};
\tkzDefPoint(58.01119,-81.45368){A}
\tkzDefPoint(41.53536,-90.96600){B}
\tkzDefPoint(0.0,0.0){C}
\tkzDrawArc[<-,line width=0.9mm, gray](C,B)(A)
\node [draw=none,fill=none,scale=1.50] () at (57.50000,-99.59292) {E};
\tkzDefPoint(34.95274,-93.69261){A}
\tkzDefPoint(16.57635,-98.61655){B}
\tkzDefPoint(0.0,0.0){C}
\tkzDrawArc[<-,line width=0.9mm, gray](C,B)(A)
\node [draw=none,fill=none,scale=1.50] () at (29.76419,-111.08147) {C};
\tkzDefPoint(9.51232,-99.54655){A}
\tkzDefPoint(-9.51232,-99.54655){B}
\tkzDefPoint(0.0,0.0){C}
\tkzDrawArc[->,line width=0.9mm, gray](C,B)(A)
\node [draw=none,fill=none,scale=1.50] () at (0.00000,-115.00000) {F};
\node [circle,black,draw,fill=white,scale=\vertexscale] (6) at (-13.05262,-99.14449) {6};
\node [circle,black,draw,fill=white,scale=\vertexscale] (40) at (92.38795,-38.26834) {6};
\node [circle,black,draw,fill=white,scale=\vertexscale] (41) at (13.05262,99.14449) {6};
\node [circle,black,draw,fill=white,scale=\vertexscale] (42) at (-92.38795,38.26834) {6};
\node [circle,black,draw,fill=white,scale=\vertexscale] (43) at (60.87614,79.33533) {6};
\node [circle,black,draw,fill=white,scale=\vertexscale] (44) at (-79.33533,-60.87614) {6};
\node [circle,black,draw,fill=white,scale=\vertexscale] (45) at (79.33533,-60.87614) {6};
\node [circle,black,draw,fill=white,scale=\vertexscale] (46) at (13.05262,-99.14449) {6};
\node [circle,black,draw,fill=white,scale=\vertexscale] (47) at (92.38795,38.26834) {6};
\node [circle,black,draw,fill=white,scale=\vertexscale] (48) at (-38.26834,-92.38795) {6};
\node [circle,black,draw,fill=white,scale=\vertexscale] (49) at (-38.26834,92.38795) {6};
\node [circle,black,draw,fill=white,scale=\vertexscale] (50) at (-92.38795,-38.26834) {6};
\node [circle,black,draw,fill=white,scale=\vertexscale] (51) at (38.26834,92.38795) {6};
\node [circle,black,draw,fill=white,scale=\vertexscale] (52) at (99.14449,-13.05262) {6};
\node [circle,black,draw,fill=white,scale=\vertexscale] (53) at (-99.14449,-13.05262) {6};
\node [circle,black,draw,fill=white,scale=\vertexscale] (54) at (-60.87614,79.33533) {6};
\node [circle,black,draw,fill=white,scale=\vertexscale] (55) at (60.87614,-79.33533) {6};
\node [circle,black,draw,fill=white,scale=\vertexscale] (56) at (-60.87614,-79.33533) {6};
\node [circle,black,draw,fill=white,scale=\vertexscale] (57) at (79.33533,60.87614) {6};
\node [circle,black,draw,fill=white,scale=\vertexscale] (58) at (-99.14449,13.05262) {6};
\node [circle,black,draw,fill=white,scale=\vertexscale] (59) at (99.14449,13.05262) {6};
\node [circle,black,draw,fill=white,scale=\vertexscale] (60) at (38.26834,-92.38795) {6};
\node [circle,black,draw,fill=white,scale=\vertexscale] (61) at (-79.33533,60.87614) {6};
\node [circle,black,draw,fill=white,scale=\vertexscale] (62) at (-13.05262,99.14449) {6};
\end{tikzpicture}